\documentclass[11pt,a4paper,twoside,reqno]{amsart}
\usepackage{amsmath,amsfonts,mathrsfs,amssymb,amsthm,mathtools}
\usepackage[shortlabels]{enumitem}
\usepackage[foot]{amsaddr}

\usepackage[margin=1.2in]{geometry}
\usepackage[UKenglish]{babel}

\usepackage[dvipsnames]{xcolor}
\definecolor{MyColor}{HTML}{0047AB}

\usepackage{algorithm}
\usepackage{algpseudocode}
\usepackage{booktabs}

\usepackage{hyperref}
\hypersetup{
colorlinks=true, 
linktoc=all,     
linkcolor=blue,  
citecolor=blue,
}

\theoremstyle{definition}
\newtheorem{anytheorem}{Theorem}[section] 
\newtheorem{definition}[anytheorem]{Definition}
\newtheorem{theorem}[anytheorem]{Theorem}
\newtheorem{corollary}[anytheorem]{Corollary}
\newtheorem{lemma}[anytheorem]{Lemma}

\newtheorem{remark}[anytheorem]{Remark}

\newtheorem{assumption}[anytheorem]{Assumption}

\allowdisplaybreaks

\DeclareMathOperator*{\argmin}{arg\,min}

\DeclareMathOperator*{\tr}{Tr}

\newcommand{\E}{\ensuremath{\mathbb{E}}}  
\newcommand{\KL}{\text{KL}}

\def\cO{\mathcal{O}}

\title{Mirror descent for constrained stochastic control problems}

\author{
Deven Sethi
\and 
David \v{S}i\v{s}ka
}

\date{\today}

\begin{document}

\begin{abstract}
Mirror descent is a well established tool for solving convex optimization problems with convex constraints.
This article introduces continuous-time mirror descent dynamics for approximating optimal Markov controls for stochastic control problems with the action space being bounded and convex.
We show that if the Hamiltonian is uniformly convex in its action variable then mirror descent converges linearly while if it is uniformly strongly convex relative to an appropriate Bregman divergence, then the mirror flow converges exponentially.
The two fundamental difficulties that must be overcome to prove such results are: first,
the inherent lack of convexity of the map from Markov controls to the corresponding value function.
Second, maintaining sufficient regularity of the value function and the Markov controls along the mirror descent updates.
The first issue is handled using the performance difference lemma, while the second using careful Sobolev space estimates for the solutions of the associated linear PDEs.
Finally, we complement the theoretical analysis with numerical experiments. We provide a model-based algorithm which uses finite difference discretization and discrete mirror descent updates.
Empirical results confirm the theoretically predicted convergence rates.
We also provide a model-free algorithm which uses martingale differences to learn the advantage rate function and applies this in the mirror descent updates of the policy.
\end{abstract}
\maketitle

\section{Introduction}
Stochastic control problems in continuous space and time have widespread applications across fields such as engineering, finance, and economics, see \cite{fleming2012deterministic, krylov2008controlled, pham2009continuous}. 
This paper considers
first order (gradient based) optimization methods for
finite-time horizon stochastic control problems with a bounded state space $\cO\subset\mathbb{R}^d$ and a bounded convex action space $A\subset\mathbb{R}^p$.
We aim to minimize the value function 
\begin{align*}
V^u(t,x)= \E^u_{t,x}\bigg[\int_t^{\min\{T,\tau_\cO\}} \Big(f^{u(r,X_r)}(r,X_r)+\tau\rho^{u(r,X_r)}(r,X_r)\Big)dr+g(X_{\min\{T,\tau_\cO\}})\bigg],
\end{align*}
over Markovian controls $u:[0,T]\times\cO\rightarrow A$.
Here $\tau\geq 0$ is a constant, 
$(X_{r})_{r\geq t}$ is a solution to a stochastic differential equation (SDE) corresponding to the control $u$,
where the control influences only the drift term, not the diffusion term, and the SDE reads as 
\[
X_{t'}=x+\int_{t}^{t'}b(r,X_r,u(r,X_r))dr +\int_{t}^{t'}\sigma(r,X_r)dW^{t,x,u}_r\,.
\]
Moreover, $\tau_\cO$ is the first exit time of $X_r$ from $\cO$.
Finally $f,\rho:[0,T]\times\cO\times A\rightarrow(\mathbb{R},\mathbb{R}\cup\{+\infty\})$ are given functions such that for each $(t,x)\in[0,T]\times\cO$ the maps $a\mapsto\rho^a(t,x)$ and $a\mapsto f^a(t,x)$ are convex. 
We will give a precise statement of the problem and all assumptions together with main results in Section~\ref{sec problem formulation and main results}.

Such problems rarely have closed form solutions and need to be solved numerically.
Numerical methods can be either based on the Bellman PDE (Hamilton--Jacobi--Belmman (HJB) equation), which will be introduced later, see~\eqref{eqn:hjb}, or on the Pontryagin optimality principle.
For problems with Markovian structure Howard's Policy Iteration Algorithm (PIA) based on the Bellman PDE is perhaps prefferable as it naturally provides access to Markov controls. 
To implement the PIA, one has to discretize. 
One only expects stability and convergence of the discretized method if the iteration method converges on an appropriate function space. 
This is why for example convergence of policy iteration methods for Bellman PDE has been studied extensively, see~\cite{smears2014iscontinuous, jacka2017policy, kerimkulov2020exponential, huang2025convergence} and references therein.
Nevertheless, policy iteration can be computationally expensive and so policy gradient methods are sometimes favoured in applications, despite being harder to analyze.

A fundamental challenge to showing convergence of gradient-based methods 
is the non-convexity of the map $u\mapsto V^u$ from the space of Markov controls to extended real numbers even for the problems where the Hamiltonian is convex in the action variable (uniformly in all other variable).
See~\cite[Proposition 2.4]{giegrich2024convergence}.
This is fundamentally different to the case when one optimizes over open-loop controls. 
In this case the Pontryagin optimality principle states that convexity of the Hamiltonian (in state and action) is a sufficient condition for convexity of the overall objective function as a map from open-loop controls to extended real numbers.

Despite this non-convexity, gradient-based methods, commonly referred to as policy-gradient-methods (PGMs),
have become an increasingly popular technique for constructing (nearly) optimal controls.
While extensive research has been conducted on PGMs for discrete-time Markov decision processes (MDPs), see \cite{han2016deep, sutton1998reinforcement, mei2020global, agarwal2021theory, pmlr-v162-leahy22a, pmlr-v97-ahmed19a} and references therein, the continuous-time setting remains relatively less explored, see \cite{munos2006policy, reisinger2023linear} 
and \cite{JMLR:v23:21-1387}.

The central idea behind these methods is to construct a 
sequence of controls based on the gradient of the function $u\mapsto V^u$, that is, given an existing Markov control $u_\text{old}$ one hopes to improve it by setting $u_\text{new} = u_\text{old} - \eta \nabla_a H(\cdot,\nabla V^{u_\text{old}},u_\text{old})$.
As the step size $\eta\searrow 0$ one heuristically obtains 
$\tfrac{d}{ds}u_s = -\nabla_a H(\cdot,\nabla V_s,u_s)$.
Here
$V_s:=V^{u_s}$ is the value function corresponding to the control $u_s$.
For avoidance of doubt: a new time axis has been introduced and $s\geq 0$ is the time in the optimization method.
There still remains the original time axis of the first exit time control problem. 

This paper focuses on the setting where $A$ is a convex bounded set in $\mathbb{R}^p$ meaning that the aforementioned update cannot guarantee $u_s(t,x)\in A$.
One approach is to project $u_s(t,x)$ from $\mathbb R^p$ onto $A$ using best approximation of $u_s(t,x)$ in e.g. the Euclidean norm.
A different, and analytically more tractable method
for solving constrained static optimization problems is the method of Nemirovski and Yudin \cite{nemirovskij1983problem} known as mirror descent.
The key insight of this method is that the gradient step can be done in an unconstrained manner in a dual space, which in our setting will be $\mathbb{R}^p$, and then mapped back to the primal space $A$, via an appropriately chosen mirror map.

In this paper, we consider the following continuous-time mirror flow. 
Let $\psi^*:\mathbb{R}^p\to\mathbb{R}$ be a function satisfying $\nabla \psi^*(\mathbb{R}^p) = \operatorname{ri}(A)$, where $\operatorname{ri}(A)$ denotes the relative interior of $A$.
This gives us the mirror map $\nabla \psi^\ast$.
Fix an initial condition $Z_0 \in B_b(\cO_T; \mathbb{R}^p)$, the space of bounded measurable functions from $(0,T) \times \cO$ to $\mathbb{R}^p$. 
On $B_b(\cO_T;\mathbb{R}^p)$ the continuous-time mirror flow is given by  
\begin{align}
\label{eqn:grad_flow}
\tfrac{d}{ds}Z_s&=-\nabla_a H(\cdot,\nabla V^{u_s},u_s)\text{ where }u_s = \nabla\psi^*(Z_s) \text{ for }s\geq 0.
\end{align}
Let us now illustrate the main results of this paper by analogy with the static optimization setting.
Let $A\subset\mathbb{R}^p$ be a convex set, assume $H:A\rightarrow\mathbb{R}$ is 
a lower bounded, convex, differentiable function
and fix a convex map $\psi:\mathbb{R}^p\rightarrow\mathbb{R}\cup\{+\infty\}$ whose Legendre conjugate, denoted by $\psi^*$, satisfies $\nabla\psi^*(\mathbb{R}^p)=\operatorname{ri}(A)$.
See \cite[Section 26]{rockafellar1970convex} for a comprehensive overview of such maps.
Consider the problem of approximating $a^*\in \argmin_{a\in A}H(a)$.
The continuous time mirror descent  \cite[Section 3.1]{nemirovskij1983problem} 
reads as follows: 
set $a_0=\nabla\psi^*(Z_0)$ for some $Z_0\in\mathbb{R}^p$ and for $s>0$ define 
\begin{align}
\label{eqn:eucli_md}
\tfrac{d}{ds}Z_s = - \nabla H(a_s) \text{ where } a_s =\nabla\psi^*(Z_s).
\end{align}
A simple application of the chain rule shows that $s\mapsto H(a_s)$ is decreasing (non-increasing) whenever $Z$ satisfies \eqref{eqn:eucli_md}. 
Indeed
\begin{equation*}
\begin{split}
\tfrac{d}{ds}H(a_s)&=\nabla H(a_s)\cdot\tfrac{d}{ds}a_s
= \nabla H(a_s)\cdot \nabla^2\psi^*(Z_s)\tfrac{d}{ds}Z_s=-\nabla H(a_s)^\top \nabla^2\psi^*(Z_s)\nabla H(a_s)\leq 0,
\end{split}
\end{equation*}
where the final inequality follows from the convexity of $\psi^*$, see \cite[Corollary 26.4.1]{rockafellar1970convex}.
The equivalent result in our setting is Theorem \ref{thm:decreasing_cost_funct}. 
The main challenge is proving that the chain rule can be applied to the map $s\mapsto V^{u_s}$. 
This relies on delicate PDE regularity estimates for the value functions which are presented in Section \ref{sec:reg_and_pd}.

Suppose that $a^*=\nabla\psi^*(Z^*)$ for some $Z^*\in\mathbb{R}^p$.
If $a^*\notin\operatorname{ri}(A)$ no such $Z^*$ exists, and one instead compares with an arbitrary $\bar a=\nabla\psi^*(\bar Z)$, as we do in Theorem~\ref{thm:lin_conv}.
Convergence can be proved using the Bregman divergence $D_{\psi^*}(Z,Z'):=\psi^*(Z)-\psi^*(Z') - \nabla\psi^*(Z')\cdot(Z-Z')$ which gives access to the Lyapunov function $Z\mapsto D_{\psi^*}(Z,Z^\ast)$, see Nemirovski and Yudin~\cite{nemirovskij1983problem}.
Under the convexity assumption that 
there exists a $\lambda\geq 0$ such that for any $a,a'\in A$, 
\begin{align}
\label{eqn:static_convex}
H(a)-H(a')\geq\nabla H(a')\cdot(a-a')+\tfrac{\lambda}{2}D_\psi(a,a'),
\end{align}
it follows that 
\begin{equation}
\label{eqn:lyap_static}
\begin{split}
\tfrac{d}{ds}D_{\psi^*}(Z_s,Z^*)
=\left(a^* - a_s\right)\cdot \nabla H(a_s)\leq H(a^*)-H(a_s)-\tfrac{\lambda}{2}D_\psi(a^*, a_s).
\end{split}
\end{equation}
When~\eqref{eqn:static_convex} holds with $\lambda=0$ integrating the above from $0$ to $S$ yields a linear rate of convergence: for $S\geq0$ we have $
S(H(a_S)-H(a^*))\leq D_{\psi^*}(Z_0,Z^*)$.
If~\eqref{eqn:static_convex} holds with $\lambda>0$ and if $D_\psi(\nabla\psi^*(y),\nabla\psi^*(y'))=D_{\psi^*}(y',y)$ for all $y,y'$ then~\eqref{eqn:lyap_static} leads to exponential convergence.
The analogous result in our setting is Theorem \ref{thm:lin_conv}, which establishes both linear and exponential convergence of the mirror flow under the assumption that there exists a constant
$\lambda\geq 0$ such that 
\begin{align}
\label{eqn:intro_conv_ham}
H(t,x,z,a)-H(t,x,z,a')\geq \nabla_aH(t,x,z,a')\cdot(a-a')+\tfrac{\lambda}{2}D_{\psi}(a,a').
\end{align}
It is worth recalling that even when~\eqref{eqn:intro_conv_ham} holds, 
the map $u\mapsto V^u$ is non-convex as a function from Markov controls to the extended real numbers.
Of course convexity is only a sufficient (not a necessary) condition for convergence of gradient methods but it still makes the result of Theorem \ref{thm:lin_conv} somewhat surprising.
The ability to bypass this non-convexity relies on two key insights. 
First, we establish a performance difference, 
formalized in Lemma \ref{lem:perf_diff}, 
which shows exactly how the convexity breaks down.
Second, we introduce a Lyapunov function~\eqref{eqn:bregman}, which extends the one used in the static case above.

We highlight an important example which falls within our framework. 
The analysis presented in the paper covers the case where the action space is the probability simplex, i.e. $A=\{a\in \mathbb R^p: a_i\geq 0, \sum_{i=1}^p a_i = 1\}$ for some $p\in\mathbb{N}$.
See Example \ref{sec:example_finite_actions} for more detail.
If the results in~\cite{sethi2025entropy} are applied to the space of probability measures arising from finite and discrete set of actions and $\rho^a(t,x)=\KL(a,a^{(0)})$ for all $(t,x)\in[0,T]\times\cO$ and for some fixed $a^{(0)}\in A$,
then the results on convergence of the mirror descent in~\cite{sethi2025entropy} complements the one studied in this paper.
Here we consider a finite-time horizon problem and include general Bregman divergences.
 In contrast~\cite{sethi2025entropy} deals with infinite time horizon where the regularizer is restricted to the KL divergence, though it accommodates more general spaces of probability measures beyond the finite simplex.

\subsection{The map $\psi$ as a self-concordant barrier and its construction}
\label{sec psi and self concordant barriers}
When $A$ has a simple geometry it is easy to obtain $\psi$ analytically, as we show for an Euclidean ball in Section~\ref{sec:ex_LQR} and for the probability simplex in Section~\ref{sec:example_finite_actions}.

More generally, $\psi$ can be obtained as a self-concordant barrier~\cite[Ch. 2]{nesterov1994interior}
for any bounded convex set $A$.
In particular,~\cite[Theorem 2.5.1]{nesterov1994interior} proves that a universal self-concordant barrier always exists.
The universal theoretical construction can be computationally infeasible.
However~\cite[Ch. 3]{nesterov1994interior} provides a calculus of constructive methods to build barriers for complex convex sets from simpler ones in a computationally efficient manner. 
For the purposes of this paper it suffices to know that algorithms to obtain $\psi$ for convex $A$ exist and, in many practical applications, yield computationally efficient mappings.
Crucially, because the geometry of the action space $A$ is static, the map $\psi$ and its dual mapping $\nabla\psi^*$ only need to be derived once (offline).
During the optimization loop, Algorithm \ref{alg:dynamic_md}, which we will discuss in Section~\ref{sec discrete stepping mirror descent}, one merely evaluates this pre-constructed map, avoiding the need to numerically solve constrained projection problems at every grid point and iteration (as would be the case for projected PG methods).

\subsection{Discrete stepping mirror descent}
\label{sec discrete stepping mirror descent}

Let us briefly discuss how the above connects to implementable algorithms. 
In the case of static optimization the explicit Euler discretization of~\eqref{eqn:eucli_md} yields the discrete stepping mirror descent which is summarized as Algorithm~\ref{alg:static_md}.

\begin{algorithm}[htpb]
\caption{Discrete stepping mirror descent for static optimization}
\label{alg:static_md}
\begin{algorithmic}[1]
\Require Objective function $H:A\to \mathbb R\cup {+\infty}$, dual mirror map $\psi^*$, step sizes $\{\eta_k\}_{k=0}^{N-1}$, initial guess $Z_0 \in \mathbb{R}^p$, total iterations $N$.
\State $a_0 \gets \nabla\psi^*(Z_0)$ \Comment{Obtain initial action}
\For{$k = 0, 1, \dots, N-1$}
\State $Z_{k+1} \gets Z_k - \eta_k \nabla H(a_k)$ \Comment{Unconstrained step downhill in the dual space}
\State $a_{k+1} \gets \nabla\psi^*(Z_{k+1})$ \Comment{Obtain next action in primal space $A$}
\EndFor
\State \Return $a_N$
\end{algorithmic}
\end{algorithm}

To obtain an implementable algorithm from~\eqref{eqn:grad_flow} we would additionally need to discretize the on-policy Bellman equation (presented later in~\eqref{eqn:onpolicy_bellman}), so that we have access to an approximation 
of $\nabla_a H(\cdot, \nabla V_k, u_k)$ where $u_k$ is a discrete representation of the Markov control at iteration $k$.
The on-policy Bellman equation can be discretized in space using the Finite Element Method or the Finite Difference Method (FDM) and implicit or explicit Euler stepping can be used for the time component (or one could use other approaches e.g. ``deep'' Galerkin methods).
In Algorithm~\ref{alg:dynamic_md} we present a complete approach using the FDM combined with mirror descent stepping.

\begin{algorithm}[htpb]
\caption{Discrete stepping mirror descent for stochastic control}
\label{alg:dynamic_md}
\begin{algorithmic}[1]
\Require Dual mirror map $\psi^*$, mirror descent step sizes $\{\eta_k\}_{k=0}^{N-1}$ and total mirror descent iterations $N$.
\Require Initial dual table $Z_0$ defined on a space-time grid $\mathcal{G} = \{(t_m, x_n)\}$. 
\State $u_0 \gets \nabla\psi^*(Z_0)$ \Comment{Obtain initial table of actions on $\mathcal{G}$}
\For{$k = 0, 1, \dots, N-1$}
\State $V_{k} \gets$ Solve on-policy Bellman equation~\eqref{eqn:onpolicy_bellman} with control $u_k$ via FDM
\State $\nabla V_k \gets$ Compute spatial gradient of $V_k$ using finite differences
\State $Z_{k+1} \gets Z_k - \eta_k \nabla_a H(\cdot, \nabla V_k, u_k)$ \Comment{Dual update}
\State $u_{k+1} \gets \nabla\psi^*(Z_{k+1})$ \Comment{Obtain next action in primal space $A$}
\EndFor
\State $V_N \gets$ Solve on-policy Bellman equation~\eqref{eqn:onpolicy_bellman} with final $u_N$ via FDM
\State \Return $u_N, V_N$
\end{algorithmic}
\end{algorithm}

Few brief comments are in order. 
First, let us contrast the Algorithm~\ref{alg:dynamic_md} to Howard's policy improvement algorithm (PIA), for which convergence has been established under various settings~\cite{jacka2017policy, kerimkulov2020exponential, huang2025convergence}.
In PIA we would have control $u_k$, we would solve the on-policy Bellman equation~\eqref{eqn:onpolicy_bellman} with control $u_k$ via FDM
to obtain $V_k$. 
But in PIA we would set
\[
u_{k+1} \in \argmin_{a\in A} H(\cdot, \nabla V_k, a)\,.
\]
In other words we would expect to solve this static minimization at each point in the grid $\mathcal G$ fully at each iteration (possibly using Algorithm~\ref{alg:static_md}, possibly using another static minimization method).
This is inefficient: first, in the first iterations especially, we only have a guess for the value function, so the gradient estimate going into the optimization will be poor. 
There is little to be gained fully minimizing.
Second and foremost the mirror descent updates in Algorithm~\ref{alg:dynamic_md} can be fully vectorized but carrying out the minimization for each $(t_m, x_n)$ could get computationally prohibitively expensive even for modest problem sizes. 

Second, Algorithm~\ref{alg:dynamic_md} has all the known drawbacks of FDMs, in particular, it would suffer from a curse of dimensionality (unless one cleverly uses sparse grid methods etc.) but in principle the same framework could be applied with different approaches to function approximation and methods for approximating solutions to the on policy Bellman PDE~\eqref{eqn:onpolicy_bellman}.

Third and final comment is that while our convergence results for the continuous time mirror flow~\eqref{eqn:grad_flow} give hope that a discretization would also converge, this is of course not implied by results in this paper. 
Certainly linear (or exponential) convergence of static mirror descent stepping can be proved via the Bregman proximal inequality (three point lemma) under convexity (strong convexity) and L-smoothness assumptions. 
It is not clear how to apply to stochastic control problems in continuous time: indeed performance difference lemma takes the expectation with respect to the ``wrong control'' ($u'$ instead of $u$). 
In discrete time MDPs this has been overcome, see~\cite{lan2023policy,lan2023block,kerimkulov2025fisher}. 
Moreover for a full convergence analysis of Algorithm~\ref{alg:dynamic_md} one would have to account for discretization errors arising from the FDM and their accumulation.
This is beyond the scope of this paper.
That said, in Section~\ref{sec numerical experiments} we present numerical experiments indicating convergence of Algorithm~\ref{alg:dynamic_md} as well as comparison with the projected policy gradient method.

\subsection{Connection to reinforcement learning (RL)}
\label{sec:intro_rl_connection}
While the numerical scheme discussed above (Algorithm \ref{alg:dynamic_md}) relies on knowing the exact environmental dynamics to solve the Bellman PDE via finite differences, the continuous-time mirror descent framework naturally extends to model-free reinforcement learning (RL). 

A fundamental observation in continuous-time RL is that the gradient of the Hamiltonian with respect to the policy corresponds directly to the continuous-time advantage rate (also called the little $q$ function).
Since the advantage rate can be learned purely from sample trajectories using martingale temporal difference losses~\cite{jia2023q}, the mirror descent updates can be evaluated entirely from data, without requiring prior knowledge of the underlying drift or running costs. 

We make this connection explicit, and provide a fully model-free mirror descent actor-critic algorithm, in Section~\ref{sec:rl_implementation}.

\subsection{Related works}
There are numerous methods for dynamic optimisation. 
While policy iteration methods~\cite{jacka2017policy, kerimkulov2020exponential, huang2025convergence} are known to converge exponentially they are not always favoured when designing efficient algorithms.
Here we focus on policy gradient methods (PGMs) due to their potential for computational efficiency. 
We are specifically interested in PGMs for Markov controls, rather than open loop controls, as Markov controls are more practical.
As mentioned before, the open-loop setting respects convexity, in the sense that convexity of the unoptimized Hamiltonian leads to convexity of the objective. 
Based on this PGM methods for open-loop controls have been shown to converge~\cite{sethi2024modified,vsivska2024gradient,kerimkulov2025mirror}.

Very recently, convergence of PGMs became relatively well understood in the MDP framework.
For softmax policies with finite state and action spaces  \cite{cayci2024convergence, cen2022fast,lan2023policy,ju2026policy} demonstrate a linear rate of convergence for the value functions for certain PGMs.
Linear convergence is also established for more general regularisers in \cite{zhan2023policy}.
Exponential convergence of mirror descent and its continuous-time counterpart for the entropy regularised MDPs on Polish state and action spaces is established in~\cite{kerimkulov2025fisher}.

For continuous time control problems PGMs are much less understood. 
In~\cite{JMLR:v23:21-1387} an actor-critic PGM algorithm is proposed but without any convergence analysis.
For unconstrained action spaces PGMs are shown to converge at a linear rate using specific structural assumptions~\cite{giegrich2024convergence} and using the specific structure of the LQR control problem~\cite{reisinger2023linear}.
When the action space is constrained updating the control in a PGM is transformed into a constrained optimization problem which can be addressed using mirror descent.
In~\cite{sethi2025entropy} it is demonstrated that  mirror descent converges exponentially for relaxed and  entropy regularized first-exit time problems and moreover the strength of the entropy regularization can be continuously reduced to get convergence to the unregularized control problem, in some specific cases, with linear rate.

\subsection{Our contributions}
\label{sec our contribution}
The major innovation of this paper is applying the mirror descent framework, originally proposed by Nemirovski and Yudin for static convex optimization to the dynamic, continuous-time stochastic control setting. 

The main contribution is the proof of linear and exponential convergence along the continuous-time mirror descent flow~\eqref{eqn:grad_flow}, Theorem~\ref{thm:lin_conv}. 
Unlike static convex optimization, the mapping from Markov controls to the value function is inherently non-convex.
The main challenges that had to be overcome to achieve our convergence results in this dynamic setting are the following:
\begin{enumerate}[i)]
\item We establish that the value function is monotonically decreasing (non-increasing) along the flow~\eqref{eqn:grad_flow} (Theorem \ref{thm:decreasing_cost_funct}).
This requires proving Hadamard differentiability of the objective function, Lemma~\ref{lem:had_diff}, and establishing differentiability along the flow by proving that an appropriate chain rule applies, due to a continuous-time performance difference lemma.
\item We prove existence and uniqueness of solutions to the mirror flow~\eqref{eqn:grad_flow}, Theorem \ref{thm:grad_flow_well_posed}. 
This requires showing that the value functions and the Markov controls maintain their regularity along the flow, which in turn necessitates Sobolev space regularity estimates for the value functions derived from their parabolic PDE representation, Section \ref{sec:reg_and_pd}. 
\end{enumerate}
Finally, we provide two concrete examples which fit the framework of this paper.
The first covers the case where the actions are restricted to a Euclidean ball in $\mathbb R^p$.
The second example highlights how the proposed framework captures entropy-regularized Markovian control problems with finite action spaces. 
These examples are presented in Section \ref{sec:examples}.

\subsection{Organisation of the paper}
\label{sec organisation of the paper}
The paper is organised as follows. 
Section~\ref{sec notation} concludes the introduction by introducing the notation used throughout the paper. 
In Section~\ref{sec problem formulation and main results} we provide all the assumptions and state the main results without presenting proofs. 
In Section~\ref{sec:examples} we present two examples which fit the assumptions of Section~\ref{sec problem formulation and main results} but some details of the verification are postponed until Appendix~\ref{sec validation of assumptions for our examples}. 
The proofs of our results are split into three parts. 
First, results on the regularity of solutions of the HJB equation for our admissible controls 
are proved in Appendix~\ref{sec:app_pde_pfs}.
In Section~\ref{sec:reg_and_pd} we prove the performance difference lemma and local Lipschitz continuity of the value function as a function of admissible (Markov) control.
Section~\ref{sec:well_posed_decrease_conv} gives the key proofs. 
First, we show the value function decreases along~\eqref{eqn:grad_flow}.
Next, we establish existence of solutions to the mirror flow~\eqref{eqn:grad_flow}.
Finally, we prove the linear and exponential convergence for the convex and strongly convex cases, respectively.
In Section~\ref{sec numerical experiments} we present numerical experiments showing that discretizing~\eqref{eqn:grad_flow} can give a practical algorithm.

\subsection{Notation}
\label{sec notation}
For measurable $D\subset\mathbb{R}^m$ and $E\subset\mathbb{R}^n$ for some $m,n\in\mathbb{N}$, $B_b(D;E)$ represents the space of bounded measurable functions from $D$ to $E$ equipped with the norm $\|u\|_{B_b(D;E)}:=\sup_{x\in D}|u(x)|$, where $|\cdot|$ is the standard Euclidean distance in $\mathbb{R}^n$.
For a given set $A\subset\mathbb{R}^p$, we will write $\text{int}(A)$ for the interior of the set $A$.
That is $\text{int}(A)$ is the largest open set in $\mathbb{R}^p$ contained in $A$.
Let $\operatorname{aff}(A)$ denote the affine hull of $A$, i.e.\ the smallest affine subspace of $\mathbb{R}^p$ containing $A$. The relative interior of $A$ is
\[
\operatorname{ri}(A):=\{a\in A:\ \exists\,\delta>0 \text{ such that } B_\delta(a)\cap\operatorname{aff}(A)\subset A\},
\]
see~\cite[Section 6]{rockafellar1970convex}.\footnote{For example, if $A$ is the probability simplex in $\mathbb{R}^p$ then $\operatorname{aff}(A)=\{a:\sum_i a_i=1\}$ and $\operatorname{ri}(A)=\{a\in A: a_i>0 \text{ for all } i\}$, whereas $\text{int}(A)=\emptyset$.}
For a given function $\psi:\mathbb{R}^p\rightarrow\mathbb{R}\cup\{+\infty\}$, $\text{dom}(\psi)=\{a\in\mathbb{R}^p:\psi(a)<\infty\}$.

A domain refers to an open connected subset of $\mathbb{R}^d$. 
Let $\mathbb{R}^d_\infty:=[0,\infty)\times\mathbb{R}^d$ and for a domain $\cO\subset\mathbb{R}^d$ and $T\in(0,\infty)$, $\cO_T:=[0,T)\times\cO$.
The parabolic boundary of $\cO_T$ is $\partial\cO_T:=\{(0,T)\times\partial\cO\}\times\{(t,x)\in\{T\}\times\bar{\cO}\}$. 
For a given function $h:\mathbb{R}^d\rightarrow\mathbb{R}$, $\nabla h$ represents the gradient vector and $\nabla^2 h$ the Hessian matrix.
For a bounded domain $\cO\subset\mathbb{R}^d$, 
$W^{2,1}_q(\cO_T)$ denotes the space of functions 
$h:[0,T]\times\cO\rightarrow\mathbb{R}^d$ 
for which $h$ together with the generalized derivatives $\partial_t h$, $(\nabla h)_i$, $(\nabla^{2}h)_{ij}$ are all in $L_q(\cO_T)$ where $1\leq i,j\leq d$.
The space is equipped with the norm
\[
\|h\|_{W^{2,1}_q(\cO_T)} = \|h\|_{L_q(\cO_T)} + \|\partial_t h\|_{L_q(\cO_T)} + \sum_{i=1}^d\|\nabla_{x_i}h\|_{L_q(\cO_T)} + \sum_{i,j=1}^d\|\nabla_{x_i x_j}h\|_{L_q(\cO_T)}\,.
\]

For a given function $\psi:A\rightarrow\mathbb{R}$, 
the map 
$A\ni a \mapsto D_{\psi}(a,a^0)\in\mathbb{R}\cup\{+\infty\}$ is the Bregman divergence with respect to $a^0\in A$, where $D_\psi(a,a^0)=\psi(a)-\psi(a^0)-\nabla\psi(a^0)\cdot(a-a^0)$.
Finally, throughout the paper, $C$ denotes a generic positive constant whose value may change from line to line.
Unless explicitly stated otherwise, $C$ depends only on the problem primitives (such as the dimensions $d, d', p$, on $q$, the time horizon $T$, the domain $\cO$, the ellipticity constant $\kappa$, and the Lipschitz/growth bound $K$) but not on $\tau\geq 0$. 
It is strictly independent of the mirror flow time $s$, the specific Markov control $u$, and the dual variable $Z$.
Where a constant needs to be referenced later, we use subscripts (e.g., $C_N$, $C_R$).

\section{Problem formulation and main results}
\label{sec problem formulation and main results}
We consider finite-time horizon  stochastic control problems defined on a bounded domain $\cO\subset\mathbb{R}^d$ and with a bounded convex action space $A\subset\mathbb{R}^p$.
Fix $T\in(0,\infty)$, $K>0$, $d,d'\in\mathbb{N}$, 
$q\in {(d+2,\infty)}$,  
$\kappa>0$, $\tau\geq 0$ and $\psi:\mathbb{R}^p\rightarrow\mathbb{R}\cup\{+\infty\}$.
Let $(b,f):\bar{\cO}_T\times A \rightarrow (\mathbb{R}^d,\mathbb{R})$, $\sigma:\mathbb{R}^d_\infty\rightarrow\mathbb{R}^{d\times d'}$ and $g:\bar{\cO}\rightarrow\mathbb{R}$ be given measurable functions.
Finally let $u^{(0)}:\cO_T\rightarrow A$ be a given reference policy.
\begin{assumption}
\label{ass:SDE_well_posed}
The action space $A\subset \mathbb R^p$ is bounded, closed and convex
and the state space $\cO$ is a bounded domain whose boundary is a $C^2$ manifold.
The running cost $f\in B_b(\cO_T\times A;\mathbb{R})$ and the terminal cost $g\in C^2(\bar{\cO})$.
The function
$\sigma\in B_b(\mathbb{R}^d_\infty;\mathbb{R}^{d\times d'})\cap C(\mathbb{R}^d_\infty;\mathbb{R}^{d\times d'})$ and is Lipschitz continuous on $\bar{\cO}_T$, with the Lipschitz constant given by $K$ and uniformly elliptic with $\kappa > 0$ i.e. for any $(t,x)\in[0,\infty)\times\mathbb{R}^d$ and $\xi\in\mathbb{R}^{d'}$,
$|\sigma(t,x)\xi|^2>\kappa|\xi|^2$. 
Moreover $\psi$ is a strictly convex function, continuously differentiable on $\operatorname{ri}(A)$, for which $\operatorname{ri}(\operatorname{dom}(\psi))=\operatorname{ri}(A)$.
Finally $\|b\|_{B_b(\cO_T\times A;\mathbb{R}^d)}+\|\sigma\|_{B_b(\mathbb{R}^d_\infty;\mathbb{R}^{d\times d'})}+\|f\|_{B_b(\cO_T\times A;\mathbb{R})}+\|\psi(u^{(0)})\|_{B_b(\cO_T;\mathbb{R})} +\|\nabla\psi(u^{(0)})\|_{B_b(\cO_T;\mathbb{R}^p)} \leq K$.
\end{assumption}
For each $u\in B_b(\cO_T;A)$ and $(t,x)\in\cO_T$, $X^{t,x,u}$ corresponds to the unique in law weak solution to
\begin{align}
\label{eqn:controlled_SDE}
X_{t'}=x+\int_{t}^{t'}b(r,X_r,u(r,X_r))dr +\int_{t}^{t'}\sigma(r,X_r)dW^{t,x,u}_r,
\end{align}
where $W^{t,x,u}$ is a $d'$-dimensional Brownian motion defined on a probability space.
The proof of the following theorem is given in Appendix \ref{sec:app_pde_pfs}.
\begin{theorem}
\label{thm:SDE_well_posedness_moment_bound}
Let Assumption \ref{ass:SDE_well_posed} hold and assume $u\in B_b(\cO_T;A)$.
Then for each $(t,x)\in\cO_T$,
\eqref{eqn:controlled_SDE} admits a unique in law weak solution on the interval $[0,T^{t,x,u}_\cO]$ where $T^{t,x,u}_\cO:=\inf\{s\geq t : X_s^{t,x,u}\notin\cO_T\}$.
\end{theorem}
For brevity let us introduce the following notation. 
Define $\rho^\cdot:\cO_T\times A\rightarrow\mathbb{R}\cup\{\infty\}$ by $\rho^a(t,x)=D_\psi(a, u^{(0)}(t,x))$.
In particular for a map $u:\cO_T\rightarrow A$,
we define $(\rho^{u}):\cO_T\rightarrow\mathbb{R}\cup\{+\infty\}$ by 
$\rho^{u}(t,x)=\rho^{u(t,x)}(t,x)$.
When the regularization parameter is strictly positive ($\tau > 0$), we must restrict the admissible controls to those with a uniformly bounded regularization penalty.
When the problem is unregularized ($\tau = 0$), the penalty function $\rho^u$ no longer plays a r\^ole and hence any bounded measurable Markov control taking values in $A$ is admissible.
For this reason the class of admissible controls is taken to be 
\begin{align*}
\mathcal{U}^\tau_\psi :=\Big\{u\in B_b(\cO_T;A):\sup_{(t,x)\in \mathcal O_T} \rho^{u}(t,x)<\infty\Big\} \text{ for $\tau>0$ and }\mathcal{U}^0_\psi:=B_b(\cO_T;A).
\end{align*}
We aim to minimize the value function 
\begin{align}
\label{eqn:value_funct}
V^u(t,x)=\E^{u}_{t,x}\bigg[
\int_t^{T_\cO}f(r,X_r,u(r,X_r))+\tau \rho^u(r,X_r)dr + g(X_{T_\cO})
\bigg],
\end{align}
over $u\in\mathcal{U}_\psi^\tau$.
For convenience the superscript $t,x,u$ will be placed on the expectation sign to indicate expectations of quantities which depend on $t,x,u$.
For a given $u\in  \mathcal{U}_\psi^\tau$ the on-policy Bellman equation reads as follows

\begin{equation}
\label{eqn:onpolicy_bellman}
\left\{
\begin{aligned}
\frac{\partial v}{\partial t}+\mathcal{L}^{u}v+f^u+\tau \rho^u&=0 \text{ a.e. in }\cO_T\\
v&=g  \text{ on }\partial\cO_T.
\end{aligned}
\right.
\end{equation}
where
for each $a\in A$ the parabolic operator $\mathcal{L}^a:W^{2,1}_q(Q_T)\rightarrow L_q(Q_T)$ is defined by
\begin{align*}
(\mathcal{L}^av)=\frac{1}{2}\text{Tr}(\sigma\sigma^\top \nabla^2v)+b^a\cdot \nabla v.
\end{align*}
The following theorem shows that the value functions satisfy \eqref{eqn:onpolicy_bellman} and provides a $W^{2,1}_q(\cO_T)$ estimate for the solutions.
The proof is given in Appendix \ref{sec:app_pde_pfs}.
\begin{theorem}
\label{thm:pde_well_posedness}
Suppose Assumption \ref{ass:SDE_well_posed} holds.
If  $u\in \mathcal{U}^\tau_\psi$ and $V^u$ is the associated value function
then
$V^u\in W^{2,1}_q(\cO_T)$ and
satisfies 
\eqref{eqn:onpolicy_bellman}.
Moreover there exists a $C>0$ depending on $|\cO|$, $|A|$, 
$d$, $q$, $\kappa$, $K$ and $T$ such that
\begin{align}
\label{eqn:W_2_p_mu_estimate}
\|V^u\|_{W^{2,1}_q(\cO_T)}\leq C(1+\|g\|_{C^2(\cO)}+\tau\|\rho^u\|_{B_b(\cO_T;\mathbb{R})}).
\end{align}
\end{theorem}
The optimal value function $V^*:\cO_T\rightarrow\mathbb{R}$ is defined by 
\begin{align}
V^*(t,x)=\inf_{u\in \mathcal{U}^\tau_\psi} V^u(t,x).
\end{align}
Define $H:[0,T]\times\cO\times\mathbb{R}^d\times A\rightarrow\mathbb{R}$ by 
\begin{align*}
H(t,x,z,a)=b(t,x,a)\cdot z + f(t,x,a) +\tau \rho^{a}(t,x).
\end{align*}
The Hamiltonian is defined to be 
\begin{align}
\label{eqn:ham}
\mathcal{H}(t,x,z):=\inf_{a\in A}H(t,x,z,a).
\end{align}
The HJB for this problem is 
\begin{equation}
\label{eqn:hjb}
\left\{
\begin{aligned}
\frac{\partial v}{\partial t}+\frac{1}{2}\text{Tr}(\sigma\sigma^\top \nabla^2v)+\mathcal{H}(\cdot,\nabla v) & = 0  \text{ a.e. in }
\cO_T\,,\\
v& = g  \text{ on }\partial\cO_T\,.
\end{aligned}
\right.
\end{equation}
In order to prove existence and uniqueness of solutions to \eqref{eqn:hjb} in $W^{2,1}_q(\cO_T)$ the following structural condition is placed on the Hamiltonian.
\begin{assumption}
\label{ass:ham_structure}
The map $\cO_T\times\mathbb{R}^d\ni (t,x,z)\mapsto\mathcal{H}(t,x,z)\in\mathbb{R}$ is measurable and
for any $(t,x)\in\cO_T$ and $z\in\mathbb{R}^d$, $|\mathcal{H}(t,x,z)|\leq K(1+|z|)$.
Moreover there exists a measurable function $a^*:\cO_T\times \mathbb{R}^d\rightarrow A$ such that $a^*(t,x,z)\in \argmin_{a\in A}H(t,x,z,a)$.
\end{assumption}

\begin{theorem}
\label{thm:HJB_well_posed}
Suppose Assumptions \ref{ass:SDE_well_posed} and \ref{ass:ham_structure} hold.
Then
$V^*\in  W^{2,1}_q(\cO_T)$ and corresponds to the unique solution of \eqref{eqn:hjb}.
Moreover $u^*(t,x)=a^*(t,x,\nabla V^*(t,x))\in \mathcal{U}_\psi^\tau$ is an optimal control. 
\end{theorem}
The proof of Theorem \ref{thm:HJB_well_posed} is given in Appendix \ref{sec:app_pde_pfs}.
We now address the well-posedness of the mirror flow \eqref{eqn:grad_flow}.
Assumptions \ref{ass:mirror_reg} and  \ref{ass:lin_growth_derivative} below ensure that the mirror map $\psi^*$ and the problems primitives are sufficiently regular.
When $\tau>0$ we take $\psi^*:\mathbb{R}^p\rightarrow\mathbb{R}$ to be the Legendre conjugate of $\psi$, that is $\psi^*(y)=\sup_{a\in A}a\cdot y - \psi(a)$. 
When $\tau=0$ it suffices to pick any sufficiently regular $\psi^*$, such that $\nabla\psi^*(\mathbb{R}^p) = \operatorname{ri}(A)$.
Given $\psi^*$ the continuous-time mirror flow is given by  $Z_0=Z^0\in B_b(\cO_T;\mathbb{R}^p)$
\begin{align}
\tfrac{d}{ds}Z_s&=-\nabla_a H(\cdot,\nabla V^{u_s},u_s)\text{ where }u_s = \nabla\psi^*(Z_s),\text{ for }s\geq 0.
\end{align}

\begin{assumption}
\label{ass:mirror_reg}
The mirror map $\psi^*:\mathbb{R}^p\rightarrow\mathbb{R}$
has a uniformly continuous Hessian and for any $y,y'\in\mathbb{R}^p$, $|\nabla^2\psi^*(y)[y']|\leq K|y'|$. 
Moreover $\nabla\psi^*(\mathbb{R}^p)=\operatorname{ri}(A)$.

When $\tau>0$ we further need the following;
for any $y\in\mathbb{R}^p$, $|\psi(\nabla\psi^*(y))|\leq K(1+|y|)$
and 
there exists a Lipschitz continuous map $\mathcal{C}:\mathbb{R}^p\rightarrow \mathbb{R}^p$ such that $\nabla\psi(\nabla\psi^*(y))=\mathcal{C}(y)$.
Moreover,
for any $y,y'\in\mathbb{R}^p$ and for any $\varepsilon\in[0,1]$ the 
convex combination $v_\varepsilon(y,y'):=\varepsilon\nabla\psi^*(y)+(1-\varepsilon)\nabla\psi^*(y')$ 
satisfies $|\nabla\psi(v_\varepsilon(y,y'))|\leq K(1+|y|+|y'|)$.
\end{assumption}
\begin{assumption}
\label{ass:lin_growth_derivative}
For any $(t,x,a)\in\cO_T\times A$, $(|\nabla_af|+|\nabla^2_a f|+|\nabla_ab|+|\nabla^2_a b|)(t,x,a)\leq K$.
\end{assumption}
\begin{theorem}
\label{thm:grad_flow_well_posed}
Suppose Assumptions \ref{ass:SDE_well_posed}, 
\ref{ass:mirror_reg}
and  \ref{ass:lin_growth_derivative} hold. 
Then for any $Z^0\in B_b(\cO_T;\mathbb{R}^p)$
there exists a unique $Z\in \cap_{S>0}C^1([0,S]; B_b(\cO_T;\mathbb{R}^p))$ satisfying
\eqref{eqn:grad_flow}.
\end{theorem} 
Having addressed the well-posedness of \eqref{eqn:grad_flow} we consider its convergence. 
Theorem \ref{thm:decreasing_cost_funct} will show that for all $(t,x)\in\cO_T$ the value function is non-increasing along solutions to \eqref{eqn:grad_flow}. 
Using this, Theorem~\ref{thm:lin_conv} shows that convexity of the Hamiltonian in the action variable, uniformly in $(t,x,z)$ (see Assumption \ref{ass:ham_convex} below), 
is sufficient to bound, pointwise and with linear or exponential rates, the gap $V^{u_S}-V^{\bar u}$ between the value functions along solutions to~\eqref{eqn:grad_flow} and the value function of any comparator control $\bar u=\nabla\psi^*(\bar Z)$ with $\bar Z$ bounded.
Choosing $\bar u$ optimal gives convergence to $V^*$ whenever this is possible.
See also Remark~\ref{rem:comparator}.
The proof of Theorems \ref{thm:decreasing_cost_funct} and \ref{thm:lin_conv} are given in Section \ref{sec:well_posed_decrease_conv}.

\begin{theorem}
\label{thm:decreasing_cost_funct}
Suppose Assumptions \ref{ass:SDE_well_posed}, 
\ref{ass:mirror_reg}
and  \ref{ass:lin_growth_derivative} hold. 
Assume $Z^0 \in B_b(\cO_T;\mathbb{R}^p)$ and let $Z\in \cap_{S>0} C^1([0,S]; B_b(\cO_T;\mathbb{R}^p))$ be a solution to \eqref{eqn:grad_flow}.
Then for each $(t,x)\in\cO_T$ the map $s\mapsto V^{\nabla\psi^*(Z_s)}(t,x)\in \mathbb R$ is differentiable and for each $s\geq 0$
\begin{equation}
\label{eqn:decreasing_cost_funct}
\begin{split}
\tfrac{d}{ds}&V^{u_s}(t,x)=-\E^{u_s}_{t,x}\int_t^{T_\cO}
\tfrac{d}{ds}Z_s(t',X_{s,t'})\cdot
\nabla^2\psi^*(Z_s(t',X_{s,t'}))\big[\tfrac{d}{ds}Z_s(t',X_{s,t'})\big]dt',
\end{split}
\end{equation}
where $(X_{s,t'})_{t'\geq t}$ corresponds to the unique weak solution of \eqref{eqn:controlled_SDE} corresponding to the control $u_s=\nabla\psi^*(Z_s)$.
In particular $s\mapsto V^{\nabla\psi^*(Z_s)}(t,x)$ is non-increasing.
\end{theorem}

\begin{assumption}[Convexity]
\label{ass:ham_convex}
There exists a $\lambda\geq 0$ such that for any $(t,x,z)\in\cO_T\times\mathbb{R}^d$
and for any $a,a'\in A$
\begin{align}
\label{eqn:strong_convex}
H(t,x,z,a)-H(t,x,z,a')\geq \nabla_aH(t,x,z,a')\cdot(a-a')+\frac{\lambda}{2}D_{\psi}(a,a').
\end{align}
\end{assumption}
Let us
introduce the following Lyapunov function
$\mathcal{D}:B_b(\cO_T;\mathbb{R}^p)\times B_b(\cO_T;\mathbb{R}^p)\rightarrow B_b(\cO_T;\mathbb{R})$ defined for all $Z,Z'\in B_b(\cO_T;\mathbb{R}^p)$ by
\begin{equation}
\label{eqn:bregman}
\begin{split}
\mathcal{D}&(Z,Z')(t,x)\\
&:=
\E_{t,x}^{\nabla\psi^*(Z')}
\int_t^{T_\cO}
(\psi^*(Z)-\psi^*(Z'))(t',X_{t'})
- \nabla \psi^*(Z')\cdot(Z-Z')(t',X_{t'})
dt',
\end{split}
\end{equation} 
for each $(t,x)\in\cO_T$.

\begin{theorem}
\label{thm:lin_conv}
\!Let Assumptions~\ref{ass:SDE_well_posed},~\ref{ass:mirror_reg} and~\ref{ass:lin_growth_derivative} hold
and $Z\in \cap_{S>0}C^1([0,S]; B_b(\cO_T;\mathbb{R}^p))$ satisfy~\eqref{eqn:grad_flow}.
Let $\bar Z\in B_b(\cO_T;\mathbb{R}^p)$ be arbitrary and $\bar u:=\nabla\psi^*(\bar Z)$.
If Assumption \ref{ass:ham_convex} holds with $\lambda=0$,
then for each $S>0$ and for any $(t,x)\in\cO_T$
\begin{align}
\label{eqn:lin_conv}
(V^{u_S}-V^{\bar u})(t,x)
\leq \mathcal{D}(Z_0,\bar Z)(t,x)\, S^{-1}.
\end{align}
If Assumption \ref{ass:ham_convex} holds with $\lambda>0$ 
and 
for any $y,y'\in\mathbb{R}^p$, $D_{\psi}(\nabla\psi^*(y),\nabla\psi^*(y'))=D_{\psi^*}(y',y)$ 
then for each $S>0$ and $(t,x)\in\cO_T$
\begin{align}
\label{eqn:exp_con}
\left(V^{u_S}-V^{\bar u}\right)(t,x)
\leq\frac{\lambda}{2}\mathcal{D}(Z_0,\bar Z)(t,x)\Big(e^{\tfrac{\lambda}{2}S}-1\Big)^{-1},
\end{align}
and
\begin{equation}
\label{eqn:strong_convex_conv}
\mathcal{D}(Z_S,\bar Z)(t,x)\leq e^{-\frac{\lambda}{2}S}\mathcal{D}(Z_0,\bar Z)(t,x)+\tfrac{2}{\lambda}\big(1-e^{-\frac{\lambda}{2}S}\big)\big(V^{\bar u}-V^{u_S}\big)(t,x).
\end{equation}
In particular, if $V^{\bar u}(t,x)\le V^{u_S}(t,x)$ for all $S\geq 0$, e.g.\ if $\bar u = u^\ast$ (it is an optimal control), then $\mathcal{D}(Z_S,\bar Z)(t,x)\le e^{-\frac{\lambda}{2}S}\mathcal{D}(Z_0,\bar Z)(t,x)$.
\end{theorem}

\begin{remark}
\label{rem:comparator}
The comparator $\bar Z$ in Theorem~\ref{thm:lin_conv} is arbitrary.
\begin{enumerate}[i)]
\item If an optimal control can be written as $u^*=\nabla\psi^*(Z^*)$ with $Z^*\in B_b(\cO_T;\mathbb{R}^p)$, then taking $\bar Z=Z^*$ in~\eqref{eqn:lin_conv}--\eqref{eqn:strong_convex_conv} gives rates for the optimality gap $V^{u_S}-V^*\geq 0$ and, when $\lambda>0$, $\mathcal{D}(Z_S,Z^*)\le e^{-\frac{\lambda}{2}S}\mathcal{D}(Z_0,Z^*)$.
This is the case in both examples of Section~\ref{sec:examples} when $\tau>0$, see Sections~\ref{sec:ex_LQR} and~\ref{sec:example_finite_actions}.
\item Since $\nabla\psi^*$ takes values in $\operatorname{ri}(A)$, when $\tau=0$ the optimal control may take values on the relative boundary of $A$, as in Example~\ref{sec:ex_LQR} whenever $|N^\top\nabla V^*|>R$, and then no such $Z^*$ exists.
In this case~\eqref{eqn:lin_conv} still gives, for all $S>0$ and $(t,x)\in\cO_T$,
\[
(V^{u_S}-V^*)(t,x)\le\inf_{\bar Z\in B_b(\cO_T;\mathbb{R}^p)}\Big\{\big(V^{\nabla\psi^*(\bar Z)}-V^*\big)(t,x)+S^{-1}\mathcal{D}(Z_0,\bar Z)(t,x)\Big\},
\]
which balances the error of approximating $u^*$ by controls with values in $\operatorname{ri}(A)$ against the optimisation error.
\end{enumerate}
\end{remark}

\section{Examples}
\label{sec:examples}
\subsection{Linear quadratic control problem with $A$ being a ball}
\label{sec:ex_LQR}
Let $M_1, M_2,M_3,N\in\mathcal{M}_{2\times 2}$, the space of $(2\times 2)$ matrices.
Set $b(t,x,a)=M_1x+Na$, $\sigma(t,x)= M_2$, $f(t,x,a)=\tfrac{|x|^2}{2}+\tfrac{|a|^2}{2}$ and $g(x)=x^\top M_3 x$ and assume $M_2$ is a strictly positive definite matrix.
Let $\cO$ be an appropriate domain in $\mathbb{R}^2$ and let $A$ represent the ball of radius $R$ centered at the origin in $\mathbb{R}^2$.
Let $\psi:\mathbb{R}^2\rightarrow\mathbb{R}\cup\{\infty\}$ be defined by $\psi(a)=-\log(R^2-|a|^2) $ for $|a|<R$ and $\psi(a)=+\infty$ otherwise.
Moreover, for each $(t,x)\in\cO_T$ let $u^{(0)}(t,x)=0\in A$.
Assumptions \ref{ass:SDE_well_posed} and \ref{ass:lin_growth_derivative} are immediate. 
A straightforward but tedious calculation shows that $\nabla\psi^*:\mathbb{R}^2\rightarrow A$ is given by 
\begin{align}
\label{eqn:ex_leg_grad}
\nabla\psi^*(y) &= \tfrac{R^2 y}{1 + \sqrt{1 + R^2 |y|^2}}.
\end{align}
The derivation of \eqref{eqn:ex_leg_grad} along with the validation of Assumption \ref{ass:mirror_reg} is provided in Appendix \ref{sec:ap_ex2}.

For $\tau\geq 0$
the Hamiltonian is given by $\mathcal{H}(t,x,z)=\inf_{a\in B_R(\textbf{0})} H(t,x,z,a)$ where
\begin{align*}
H(t,x,z,a) = x^\top M_1^\top z + a^\top N^\top z + \tfrac{|x|^2}{2}+\tfrac{|a|^2}{2}+\tau D_\psi(a, 0).
\end{align*}

For $\tau=0$ completing the square gives
\[
H(t,x,z,a)=(M_1x)\cdot z+\tfrac{|x|^2}{2}+\tfrac12|a+N^\top z|^2-\tfrac12|N^\top z|^2,
\]
so $a^*(t,x,z)$ is the Euclidean projection of $-N^\top z$ onto $A$ and, for each $(t,x,z)\in\cO_T\times\mathbb{R}^2$,
\begin{align*}
\mathcal{H}(t,x,z)&=
\begin{cases}
(M_1 x)\cdot z - \tfrac{|N^\top z|^2}{2} + \tfrac{|x|^2}{2}
& |N^\top z|\leq R \\
(M_1 x)\cdot z - R|N^\top z| + \tfrac{|x|^2}{2}+\tfrac{R^2}{2} & |N^\top z|> R
\end{cases},\\
a^*(t,x,z)
&=
\begin{cases}
-N^\top z & |N^\top z|\leq R\\
-\tfrac{RN^\top z}{|N^\top z|} & |N^\top z|> R
\end{cases}.
\end{align*}
Therefore for 
$\tau=0$
the mirror flow reads as follows: fix $Z^0\in B_b(\cO_T;\mathbb{R}^2)$ and 
\begin{align*}
\tfrac{d}{ds}Z_s = -(N^\top \nabla V^{u_s} + u_s), \enspace u_s=\tfrac{R^2 Z_s}{1+\sqrt{1+R^2|Z_s|^2}}.
\end{align*}
Whenever $|N^\top\nabla V^*(t,x)|>R$ the optimal control $u^*(t,x)=a^*(t,x,\nabla V^*(t,x))$ lies on $\partial A$, so it cannot be written as $\nabla\psi^*(Z^*)$ with $Z^*$ bounded, see Remark~\ref{rem:comparator}.
If $\sup_{\cO_T}|N^\top\nabla V^*|<R$, then $Z^*:=\nabla\psi(u^*)$ is bounded and~\eqref{eqn:lin_conv} with $\bar Z=Z^*$ gives the rate $S^{-1}$.

For $\tau>0$ and $a,a'\in A$ the Bregman divergence is given by
\begin{align}
\label{eqn:Bregman_1}
D_\psi(a, a')
=\log\tfrac{R^2-|a'|^2}{R^2-|a|^2} + \tfrac{2a'}{R^2-|a'|^2}\cdot(a-a')\,.
\end{align}
The validation of Assumptions \ref{ass:ham_structure} and 
\ref{ass:ham_convex} 
is given in Appendix \ref{sec:ap_ex2}.
Finally since $\nabla\psi(\nabla\psi^*(y))=y$ the condition in the statement of Theorem \ref{thm:lin_conv} holds. Moreover, for $\tau>0$ the optimal control satisfies $u^*=\nabla\psi^*(Z^*)$ with $Z^*:=\nabla\psi(u^*)$ bounded (see Appendix~\ref{sec:ap_ex2}), so taking $\bar Z=Z^*$ in Theorem~\ref{thm:lin_conv} we have exponential convergence along the mirror flow
\begin{align*}
\tfrac{d}{ds} Z_s=-[N^\top \nabla V^{u_s} + u_s+\tau \tfrac{2u_s}{R^2-|u_s|^2}], \enspace u_s=\tfrac{R^2 Z_s}{1+\sqrt{1+R^2|Z_s|^2}}.
\end{align*}

\subsection{Entropy regularized relaxed control problem with finite action space}
\label{sec:example_finite_actions}
We consider a control problem where the control can take a finite number of discrete actions which we will take, 
without loss of generality, 
to be $\{1,\ldots,p\}$.
The drift and running cost are given by bounded and measurable $(\beta,\varphi):[0,T]\times \mathbb R^d \times \{1,\ldots,p\}\to (\mathbb R^d,\mathbb R)$ respectively. 
The uncontrolled diffusion $\sigma$ is as in Assumption~\ref{ass:SDE_well_posed}.

To see the relaxed forumlation we take 
$A=\{a=(a_1,\dots,a_p)^\top\in\mathbb{R}^p:\enspace \sum_{i=1}^pa_i=1\text{ and }a_{i}\geq 0 \text{ for all }1\leq i \leq p\}$.
Note that $A$ can be seen as the set of probability measures on $\{1,\ldots,p\}$ i.e. $A=\mathcal P(\{1,\ldots,p\})$.
Let us fix some reference vector 
$a^{(0)}\in A$ and $\tau > 0$.
Let
$\psi:\mathbb{R}^p\rightarrow\mathbb{R}\cup\{+\infty\}$ be given by $\psi(a)=\sum_{i=1}^pa_i\log(a_i)$ for $a\in A$ and $\psi(a)=+\infty$ otherwise.
A short calculation using the definition of Bregman divergence shows that
$\rho^\cdot:A\rightarrow\mathbb{R}$ is given by 
$\rho^a=D_\psi(a, a^{(0)})
=\KL(a, a^{(0)})$.
For a given Markov control $u:\cO_T\rightarrow A$ the value function is given by
\begin{align}
\label{eqn:eg_finite_action_VF}
V^u(t,x)= \E_{t,x}^u \bigg[ \int_t^{T_\cO} \bigg(\sum_{i=1}^p \varphi(r,X_r,i) u_i(r,X_r) + \tau\KL(u(r,X_r), a^{(0)}) \bigg)dr+g(X_{T_\cO})\bigg]\,,
\end{align}
where $X^{t,x,u}$ is the unique solution of the following controlled SDE
\begin{align}
\label{eqn:eg_finite_actionion_SDE}
dX_r & = \sum_{i=1}^{p}\beta(r,X_r,i) u_i(r,X_r)dr +\sigma(r,X_r)dW_r, \enspace r\geq t, \enspace X_t = x\,.
\end{align}
We have that $H(t,x,z,a)=\sum_{i=1}^p a_i\big(\beta(t,x,i)\cdot z + \varphi(t,x,i) + \tau \log \tfrac{a_i}{a_i^{(0)}}\big)$.
The proof that Assumption \ref{ass:ham_structure} is satisfied is given in Appendix \ref{sec_app_2nd_ex}.
Assumption \ref{ass:SDE_well_posed} is immediate since the map $\psi(a)=\sum_{i=1}^pa_i\log a_i$ is strictly convex on the probability simplex.
It can be shown that the Legendre conjugate is given by  $\psi^*:\mathbb{R}^p\rightarrow\mathbb{R}$ and  satisfies $\psi^*(y)=\log(\sum_{i=1}^pe^{y_i})$ with 
\begin{align*}
\nabla\psi^*(y) = \text{softmax}(y),\enspace
\nabla^2\psi^*(y)= \text{diag}(\text{softmax}(y))-\text{softmax}(y)\text{softmax}(y)^\top,
\end{align*}
where $\text{softmax}(y):=(\tfrac{e^{y_1}}{\sum_{j=1}^pe^{y_j}},\dots,\tfrac{e^{y_p}}{\sum_{j=1}^pe^{y_j}})^\top$. 
The validation of Assumptions \ref{ass:mirror_reg} and \ref{ass:ham_convex} is given in Appendix \ref{sec_app_2nd_ex}.

Since $\tau>0$, the optimal control is of the form $u^*=\nabla\psi^*(Z^*)$ with $Z^*$ bounded, see Appendix~\ref{sec_app_2nd_ex}.
Hence Theorem~\ref{thm:lin_conv} with $\bar Z=Z^*$ and $\lambda=2\tau$ gives exponential convergence along the mirror flow.
The mirror flow now reads as follows.
Fix $Z_0\in B_b(\cO_T;\mathbb{R}^p)$ and define
\begin{align*}
\tfrac{d}{ds}Z_s = 
-\left[\beta^\top\nabla V^{u_s} + \varphi+\tau(\log u_s + 1)\right],\enspace
u_s  =\bigg(\tfrac{e^{Z^{(i)}_s}}{\sum_{j=1}^pe^{Z^{(j)}_s}}\bigg)_{i=1}^p.
\end{align*}

\section{Value function regularity and the performance difference}
\label{sec:reg_and_pd}
This section establishes several regularity properties of the map $Z\mapsto V^{\nabla\psi^*(Z)}$, which will
be used to prove the well-posedness and convergence of the mirror flow \eqref{eqn:grad_flow}.
We first prove Lemma \ref{lem:perf_diff}, the so-called performance difference lemma. 
The proof relies on Lemma \ref{lem:fey_kac} which is a generalization of the Feynman--Kac formula for twice weakly differentiable functions, see  \cite[Ch.~2., Sec.~10, Theorem 1]{krylov2008controlled}.
\begin{lemma}
\label{lem:fey_kac}
[\textbf{Feynman--Kac formula}]
Suppose Assumption \ref{ass:SDE_well_posed} holds.
Let $h\in W^{2,1}_q(\cO_T)$, $F\in L^q(\cO_T)$ and $u\in B_b(\cO_T;A)$ be such that $h=0$ a.e  on $\partial\cO_T$ and
\begin{align*}
\tfrac{\partial h}{\partial t}+\mathcal{L}^uh+F=0\text{ a.e. in }\cO_T.
\end{align*}
Then for any $(t,x)\in\cO_T$, $h(t,x) =\E^{u}_{t,x} \left[ \int_t^{T_\cO} F(r,X_r)dr \right]$ where $(X_r)_{r\geq 0}$ corresponds to the unique weak solution of \eqref{eqn:controlled_SDE} with $u$.
\end{lemma}
\begin{lemma}
\label{lem:perf_diff}
[\textbf{Performance difference}]
Suppose Assumptions \ref{ass:SDE_well_posed} and \ref{ass:mirror_reg} hold.
Let
$u,u'\in \mathcal{U}_\psi^\tau$.
Then for each $(t,x)\in\cO_T$
\begin{equation}
\label{eqn:performance_diff} 
\begin{split}
V^{u}(t,x)-V^{u'}(t,x)&=\E^{u'}_{t,x}\bigg[\int_t^{T_\cO}H(\cdot,\nabla V^{u},u)(r,X_r)-H(\cdot,\nabla V^{u},u')(r,X_r)dr\bigg],
\end{split}
\end{equation}
where $X=X^{t,x,u'}$ is the weak solution corresponding to the control $u'$.
\end{lemma}
\begin{proof}
Let $w:=V^u-V^{u'} \in W^{2,1}_q(\cO_T)$.
Due to Theorem \ref{thm:pde_well_posedness} we have $w\in W^{2,1}_q(\cO_T)$.
From~\eqref{eqn:onpolicy_bellman} we see that $w$ satisfies
\begin{align*}
\partial_t w +\tfrac{1}{2}\text{Tr}(\sigma\sigma^\top \nabla^2w)+b^{u'}\cdot \nabla w+[H(\cdot,\nabla V^u,u)-H(\cdot,\nabla V^u,u')]=0,
\end{align*}
a.e. in $\cO_T$ and $w=0$ on $\partial\cO_T$.
Therefore \eqref{eqn:performance_diff} follows from Lemma \ref{lem:fey_kac}.
\end{proof}

We now show that the map $B_b(\cO_T;\mathbb{R}^p)\ni Z\mapsto V^{\nabla\psi^*(Z)}\in W^{2,1}_q(\cO_T)$ is locally Lipschitz, this is Lemma \ref{lem:cont_u_to_V_u}. 
The proof makes use of the $W^{2,1}_q(\cO_T)$ estimate for linear parabolic PDEs given in  Lemma \ref{lem:parabolic_PDE},
whose proof can be found in 
~\cite[Theorem 9.2.3 and Theorem 9.2.5]{wu2006elliptic}.
A key part of the analysis in the rest of this section is the following embedding result given in \cite[Appendix E]{fleming2012deterministic}.
For $q>d+2$  there exists a $C>0$ such that for any $h:\cO_T\rightarrow\mathbb{R}$,
\begin{align}
\label{eqn:embedding_Sob_Easy}
\sup_{t,x\in\cO_T}|h(t,x)|+\sum_{i=1}^d\sup_{t,x\in\cO_T}|\nabla_{x_i}h(t,x)|\leq C\|h\|_{W^{2,1}_q(\cO_T)}.
\end{align}
\begin{lemma}
\label{lem:parabolic_PDE}
Let $\cO\subset\mathbb{R}^d$ be a bounded domain whose boundary is of class $C^2$.
For $1\leq i,j\leq d$ let $a_{ij},b_i,c:\cO_T\rightarrow\mathbb{R}$ be measurable functions for which there exists constants $\lambda,M>0$ such that for all $\xi\in\mathbb{R}^d$, $(x,t)\in\cO_T$,
$a_{ij}(x,t)\xi_i\xi_j\geq\lambda|\xi|^2$ and $\sum_{i,j=1}^d\|a_{ij}\|_{L_\infty(\cO_T)}+\sum_{i=1}^d\|b_i\|_{L_\infty(\cO_T)}+\|c\|_{L_\infty(\cO_T)}\leq M$ with $a_{ij}\in C(\bar{\cO})$.
For each $q\in(1,\infty)$  and $f\in L_q(\cO_T)$ there exists a unique $v\in W^{2,1}_q(\cO_T)$ satisfying the  initial boundary problem $ \tfrac{\partial v}{\partial t}-\sum_{i,j=1}^da_{ij}\nabla_{ij}v+\sum_{i=1}^db_i\nabla_iv+cv = f ,\enspace \text{ a.e in }\cO_T$,
$v=g\in W^{2,1}_q(\cO_T)$ on $\partial\cO_T$ 
and $\|v\|_{W^{2,1}_q(\cO_T)}\leq C(\|f\|_{L_q(\cO_T)}+\|g\|_{W^{2,1}_q(\cO_T)}+\|v\|_{L_q(\cO_T)} )$ where the constant $C>0$ depends on $d$, $q$, $\lambda$, $M$, $T$ and the modulus of continuity of $a_{ij}$.
\end{lemma}
\begin{lemma}
\label{lem:cont_u_to_V_u}
Suppose Assumptions \ref{ass:SDE_well_posed},  
\ref{ass:mirror_reg} and
\ref{ass:lin_growth_derivative} hold.
Assume $Z,Z'\in B_b(\cO_T;\mathbb{R}^p)$.
Then there exists a constant $C>0$ depending on $d$, $q$, $\kappa$, $K$, $T$ and $|\cO|$ 
such that
\begin{equation}
\label{eqn:continuity_v_u}
\begin{split}
\|V^u&-V^{u'}\|_{W^{2,1}_q(\cO_T)}\\
&\leq C(1+\tau)(1+\|g\|_{C^2(\overline{\cO})}+\|Z\|_{B_b(\cO_T;\mathbb{R}^p)}+\|Z'\|_{B_b(\cO_T;\mathbb{R}^p)})\|Z-Z'\|_{B_b(\cO_T;\mathbb{R}^p)},
\end{split}
\end{equation}
where $u=\nabla\psi^*(Z)$ and $u'=\nabla\psi^*(Z')$.
\end{lemma}
\begin{proof}
Let 
$f^u(t,x)=f(t,x,u(t,x))$ and $b^u(t,x)=b(t,x,u(t,x))$.
Set $w=V^{u}-V^{u'}\in W^{2,1}_q(\cO_T)$. 
From \eqref{eqn:onpolicy_bellman} 
$w=0$ on $\partial\cO_T$ and 
\begin{align}
\label{eqn:incremement_pde}
\partial_t w
+\mathcal{L}^u w+F=0\text{ a.e. in }\cO_T,
\end{align}
where $F=(b^u-b^{u'})\cdot \nabla V^{u'}+f^u-f^{u'}+\tau(\rho^u-\rho^{u'})$.
From Assumption \ref{ass:lin_growth_derivative} for each $(t,x)\in\cO_T$ we have 
\begin{align*}
|F(t,x)|&\leq |b^u(t,x)-b^{u'}(t,x)||\nabla V^{u'}(t,x)|+|f^u(t,x)-f^{u'}(t,x)|+\tau|\rho^u(t,x)-\rho^{u'}(t,x)|\\
&\leq K(1+|\nabla V^{u'}(t,x)|)|(u-u')(t,x)|+\tau|\rho^u(t,x)-\rho^{u'}(t,x)|.
\end{align*}
To bound the final term 
let $v_\varepsilon(y,y')=\varepsilon\nabla\psi^*(y)+(1-\varepsilon)\nabla\psi^*(y')$. 
From the Mean Value Theorem for all $(t,x)\in\cO_T$ we know that 
\begin{align*}
(\rho^u-\rho^{u'})(t,x)
= \int_0^1 \nabla\rho^{v_\varepsilon(Z,Z')}(t,x)
d\varepsilon\cdot(u(t,x)-u'(t,x)).
\end{align*}
Since $\nabla\rho^{v_\epsilon(Z,Z')}(t,x)=\nabla\psi(v_\epsilon(Z(t,x),Z'(t,x)))-\nabla\psi(u^{(0)}(t,x))$
it follows 
from this and Assumption \ref{ass:mirror_reg} that 
\begin{align*}
|(\rho^u-\rho^{u'})(t,x)|\leq K(1+\|Z\|_{B_b(\cO_T;\mathbb{R}^p)}+\|Z'\|_{B_b(\cO_T;\mathbb{R}^p)})|u(t,x)-u'(t,x)|,
\end{align*}
which implies 
\begin{align*}
\|F\|&_{B_b(\cO_T;\mathbb{R})}\\
&\leq K[(1+\|\nabla V^{u'}\|_{C^0(\cO_T)})+\tau (1+\|Z\|_{B_b(\cO_T;\mathbb{R}^p)}+\|Z'\|_{B_b(\cO_T;\mathbb{R}^p)})]\|u-u'\|_{B_b(\cO_T;A)}\\
&\leq C(1+\tau)[1+\|g\|_{C^{2}(\cO)}+\|Z\|_{B_b(\cO_T;\mathbb{R}^p)}+\|Z'\|_{B_b(\cO_T;\mathbb{R}^p)}]\|u-u'\|_{B_b(\cO_T;A)},
\end{align*}
where the second inequality is a consequence of~\eqref{eqn:W_2_p_mu_estimate}, which is the Sobolev space estimate for on-policy Bellman PDE, and~\eqref{eqn:embedding_Sob_Easy}, which is the embedding estimate we use.
From Lemma \ref{lem:parabolic_PDE}
there exists a constant $C>0$ such that for any $Z,Z'\in B_b(\cO_T;\mathbb{R}^p)$
\begin{equation}
\label{eqn:W_2_p_proof_help}
\begin{split}
\|w\|_{W^{2,1}_q(\cO_T)}
\leq C(\|F\|_{L_q(\cO_T)}+\|w\|_{L_q(\cO_T)})\\
\end{split}
\end{equation}
From Corollary \ref{cor:max_princ_ext} it follows that 
$-TC(Z,Z')\leq w(t,x)\leq TC(Z,Z')$ a.e. in $\cO_T$ where
\begin{align*}
C(Z,Z'):=C(1+\tau)[1+\|g\|_{C^{2}(\cO)}+\|Z\|_{B_b(\cO_T;\mathbb{R}^p)}+\|Z'\|_{B_b(\cO_T;\mathbb{R}^p)}].
\end{align*}
Therefore $\|V^u-V^{u'}\|_{W^{2,1}_q(\cO_T)}\leq C(Z,Z')\|u-u'\|_{B_b(\cO_T;A)}$.
Let $Z^\varepsilon = (1-\varepsilon)Z+\varepsilon Z'$. From the Mean Value Theorem and Assumption \ref{ass:mirror_reg} 
\begin{align*}
|(u-u')(t,x)|
&\leq\int_0^1| \nabla^2\psi^*(Z^\varepsilon(t,x))[(Z-Z')(t,x)]|d\varepsilon\leq K|(Z-Z')(t,x)|.
\end{align*}
Taking supremums and substituting the above into \eqref{eqn:W_2_p_proof_help} concludes the proof.
\end{proof}
We now want to show that for each $(t,x)$ the map $B_b(\cO_T;\mathbb{R}^p)\ni Z\mapsto V^{\nabla\psi^*(Z)}(t,x)\in\mathbb{R}$ is Hadamard differentiable, see Lemma \ref{lem:mirror_diff}. 
A key step in this is to prove the Hadamard differentiability of the map  $B_b(\cO_T;\mathbb{R}^p)\ni Z\mapsto \nabla\psi^*(Z)\in B_b(\cO_T;A)$. 
This is given in Lemma \ref{lem:mirror_diff}.
\begin{definition}
Let $X,Y$ be Banach spaces. 
We say $\mathcal{I}:X\rightarrow Y$  is Hadamard differentiable if there exists $\partial\mathcal{I}:X\rightarrow\mathcal{L}(X,Y)$, called the differential of $\mathcal{I}$, such that for all $x,v\in X$
and all sequences $\{h_n\}_{n\in\mathbb{N}}\subset(0,1)$ and $\{v_n\}_{n\in\mathbb{N}}\subset X$ such that $\lim_{n\rightarrow\infty}h_n=0$ and $\lim_{n\rightarrow\infty}v_n=v$, we have 
\begin{align*}
\lim_{n\rightarrow\infty}h_n^{-1}(\mathcal{I}(x+h_n v_n)-\mathcal{I}(x))=\partial\mathcal{I}(x)[v].
\end{align*}
\end{definition}

\begin{lemma}
\label{lem:mirror_diff}
Suppose Assumption \ref{ass:mirror_reg} holds.
Then the map $\mathcal{F}:B_b(\cO_T;\mathbb{R}^p)\rightarrow B_b(\cO_T; A)$ defined by
$Z\mapsto \nabla\psi^*(Z)$ is Hadamard differentiable and for any $Z,Z'\in B_b(\cO_T;\mathbb{R}^p)$, we have 
$\partial\mathcal{F}(Z)[Z']=\nabla^2\psi^*(Z)[Z']$.
\end{lemma}
\begin{proof}
Let $\{h_n\}_{n\in\mathbb{N}}\subset(0,1)$ satisfy $\lim_{n\rightarrow\infty}h_n=0$, 
let $\{Z'_n\}_{n\in\mathbb{N}}\subset B_b(\cO_T;\mathbb{R}^p)$ satisfy $\lim_{n\rightarrow\infty}Z'_n=Z'$ in $B_b(\cO_T;\mathbb{R}^p)$.
To show the map $\mathcal{F}$ is Hadamard differentiable we want to show that
\begin{align*}
\lim_{n\rightarrow \infty}\|h^{-1}_n(\nabla\psi^*(Z+h_n Z'_n)-\nabla\psi^*(Z))-\nabla^2\psi^*(Z)[Z']\|_{B_b(\cO_T;A)}=0.
\end{align*}
To this end
define $Z^{n,\varepsilon} = \varepsilon (Z+h_n Z')+(1-\varepsilon)Z$. 
From the Mean Value Theorem and Assumption \ref{ass:mirror_reg} we have that for any $(t,x)\in\cO_T$ and $n\in\mathbb{N}$
\begin{align*}
h_n^{-1}&(\nabla\psi^*(Z+h_n Z'_n)-\nabla\psi^*(Z))(t,x)-\nabla^2\psi^*(Z)[Z'(t,x)]\\
&=\int_0^1 h_n^{-1}\nabla^2\psi^*(Z^{n,\varepsilon}(t,x))[h_n Z'_n(t,x)]d\varepsilon-\nabla^2\psi^*(Z(t,x))[Z'(t,x)])\\
&=\int_0^1 \nabla^2\psi^*(Z^{n,\varepsilon}(t,x))[ Z'_n(t,x)]-\nabla^2\psi^*(Z(t,x))[Z'(t,x)]d\varepsilon\\
&=\int_0^1 (\nabla^2\psi^*(Z^{n,\varepsilon}(t,x))-\nabla^2\psi^*(Z(t,x)))[Z'_n(t,x)]d\varepsilon+\int_0^1 \nabla^2\psi^*(Z)[Z'_n-Z'](t,x)d\varepsilon,
\end{align*}
where
the first equality follows from the positive homogenenity of directional derivative ~\cite[Section 2.2.1]{bonnans2013perturbation}.
From Assumption \ref{ass:mirror_reg} we have 
\begin{align*}
|h_n^{-1}&(\nabla\psi^*(Z+h_n Z'_n)-\nabla\psi^*(Z))(t,x)-\nabla^2\psi^*(Z(t,x))[Z'(t,x)]|\\
&\leq \int_0^1 |\nabla^2\psi^*(Z^{n,\varepsilon}(t,x))-\nabla^2\psi^*(Z(t,x))||Z'_n(t,x)|+|\nabla^2\psi^*(Z)[(Z'_n-Z')(t,x)]|d\varepsilon\\
&\leq \int_0^1 |\nabla^2\psi^*(Z^{n,\varepsilon}(t,x))-\nabla^2\psi^*(Z(t,x))||Z'_n(t,x)|d\varepsilon+K\|Z'_n - Z'\|_{B_b(\cO_T;\mathbb{R}^p)}\\
&\leq M\int_0^1\sup_{t,x\in\cO_T}|\nabla^2\psi^*(Z^{n,\varepsilon}(t,x))-\nabla^2\psi^*(Z(t,x))|d\varepsilon+K\|Z'_n - Z'\|_{B_b(\cO_T;\mathbb{R}^p)},
\end{align*}
where $M:= \max_{n\in\mathbb{N}}\|Z'_n\|_{B_b(\cO_T;\mathbb{R}^p)}<\infty$.
Therefore 
\begin{align*}
\|h_n^{-1}&(\nabla\psi^*(Z+h_n Z'_n)-\nabla\psi^*(Z))-\nabla^2\psi^*(Z)[Z']\|_{B_b(\cO_T;A)}\\
&\leq M\int_0^1 \|\nabla^2\psi^*(Z^{n,\varepsilon})-\nabla^2\psi^*(Z)\|_{B_b(\cO_T;\mathbb{R}^{p\times p})}d\varepsilon+K\|Z'_n - Z'\|_{B_b(\cO_T;\mathbb{R}^p)}.
\end{align*}
We want to show the above converges to zero as $n\rightarrow\infty$.
It suffices to prove 
\begin{align*}
\lim_{n\rightarrow\infty}\int_0^1 \|\nabla^2\psi^*(Z^{n,\varepsilon})-\nabla^2\psi^*(Z)\|_{B_b(\cO_T;\mathbb{R}^{p\times p})}d\varepsilon=0.
\end{align*}
From the dominated convergence theorem
\begin{align*}
\lim_{n\rightarrow\infty}\int_0^1 &\|\nabla^2\psi^*(Z^{n,\varepsilon})-\nabla^2\psi^*(Z)\|_{B_b(\cO_T;\mathbb{R}^{p\times p})}d\varepsilon\\
&=\int_0^1\lim_{n\rightarrow\infty}\|\nabla^2\psi^*(Z^{n,\varepsilon})-\nabla^2\psi^*(Z)\|_{B_b(\cO_T;\mathbb{R}^{p\times p})}d\varepsilon
=0
\end{align*}
where the final inequality follows from the fact $Z^{n,\varepsilon}\rightarrow Z$ uniformly as $n\rightarrow\infty$ and the Hessian is uniformly continuous due to Assumption \ref{ass:mirror_reg}.
This concludes the proof.
\end{proof}
\begin{lemma}
\label{lem:had_diff}
Suppose Assumptions \ref{ass:SDE_well_posed}, 
\ref{ass:mirror_reg} and \ref{ass:lin_growth_derivative} hold.
Then for each $(t,x)\in \cO_T$
the map $B_b(\cO_T;\mathbb{R}^p)\ni Z\mapsto V^{\nabla\psi^*(Z)}(t,x)\in\mathbb{R}$ is Hadamard differentiable and for each $Z,Z'\in B_b(\cO_T;\mathbb{R}^p)$  
\begin{equation}
\begin{split}
\partial &V^{\nabla\psi^*(Z)}(t,x)[Z']\\
&=\E^{\nabla\psi^*(Z)}_{t,x}\int_t^{T_\cO}\nabla_aH(\cdot,\nabla V^{\nabla\psi^*(Z)},\nabla\psi^*(Z))(t',X_{t'})\cdot \nabla^2\psi^*(Z)[Z'](t',X_{t'})dt'.
\end{split}
\end{equation}
\end{lemma}
\begin{proof}
Let
$\{Z'_n\}_{n\in\mathbb{N}}\subset B_b(\cO_T;\mathbb{R}^p)$ satisfy $Z'_n\rightarrow Z'$ in $B_b(\cO_T;\mathbb{R}^p)$ as $n\rightarrow\infty$,
$\{h_n\}_{n\in\mathbb{N}} \subset (0,1)$ satisfy $h_n\rightarrow 0$ as $n\rightarrow\infty$.
For $k=b,f,\rho$ set $k^n=k^{\nabla\psi^*(Z+h_nZ'_n)}$ and $k^\infty=k^{\nabla\psi^*(Z)}$. 
Let $V^n=V^{\nabla\psi^*(Z+h_nZ'_n)}$ and  $V^\infty=V^{\nabla\psi^*(Z)}$.
Set $\E=\E^{\nabla\psi^*(Z)}_{t,x}$ and 
for each $\varepsilon\in[0,1]$ and $n\in\mathbb{N}$ let 
\begin{align}
\mathcal{I}^{(n)}_\varepsilon(t,x):=\varepsilon \nabla\psi^*(Z(t,x)+h_nZ'_n(t,x))+(1-\varepsilon)\nabla\psi^*(Z(t,x)).
\end{align}
From Lemma \ref{lem:perf_diff} for each $(t,x)\in\cO_T$
\begin{align*}
\tfrac{ (V^n-V^\infty)(t,x)}{h_n}&=
\E
\int_t^{T_\cO} \tfrac{(b^n-b^\infty)\cdot \nabla V^n}{h_n}
(t,X_t)
dt'+
\E \int_t^{T_{\cO}}\tfrac{(f^n-f^\infty)}{h_n}(t',X_{t'})dt'\\
&\enspace\enspace+ \tau
\E \int_t^{T_{\cO}}\tfrac{(\rho^n-\rho^\infty)}{h_n}(t',X_{t'})dt'.
\end{align*}
We will show that
\begin{align*}
\tfrac{(V^{n}-V^\infty)(t,x)}{h_n}&-\E\int_t^{T_\cO}\nabla_aH(\cdot,\nabla  V^\infty,\nabla\psi^*(Z))(t',X_{t'})\cdot \nabla^2\psi^*(Z)[Z'](t',X_{t'})dt'\\
&=:I_n^{(1)}+I^{(2)}_n + \tau I^{(3)}_n,
\end{align*}
converges to zero as $n\rightarrow\infty$,
where
\begin{align*}
I_n^{(1)} &= \E\int_t^{T_{\cO}}\tfrac{(b^n-b^\infty)\cdot \nabla V^n}{h_n}(t',X_{t'}^u) - \nabla_a(b^\infty(t',X_{t'})\cdot \nabla V^\infty(t',X_{t'}))\cdot \nabla^2\psi^*(Z)[Z'](t',X_{t'})dt',\\
I^{(2)}_n &= \E\int_t^{T_{\cO}}
\tfrac{ f^n-f^\infty}{h_n}(t',X_{t'}) - \nabla_a f^\infty(t',X_{t'})\cdot  \nabla^2\psi^*(Z)[Z'](t',X_{t'})dt',\\
I^{(3)}_n &= \E\int_t^{T_{\cO}}
\tfrac{ \rho^n-\rho^\infty}{h_n}(t',X_{t'}) - \nabla_a \rho^\infty(t',X_{t'})\cdot  \nabla^2\psi^*(Z)[Z'](t',X_{t'})dt'.
\end{align*}
For $I_n^{(1)}$ note that for each $(t,x)\in\cO_T$
\begin{align*}
\tfrac{(b^n-b^\infty)\cdot \nabla  V^n}{h_n}(t,x)=\int_0^1
\nabla_a(b(\cdot,\mathcal{I}^{(n)}_\varepsilon)\cdot \nabla  V^n)(t,x))\cdot
\Delta_n(t,x)d\varepsilon,
\end{align*}
where $\Delta_n(t,x)=h_n^{-1}(\nabla\psi^*((Z+h_nZ'_n)(t,x))-\nabla\psi^*(Z(t,x)))$. 
Then 
\begin{align*}
|I_n^{(1)}|&\leq\E\int_t^{T_{\cO}}
\left[\int_0^1|\nabla_a(b(\cdot,\mathcal{I}^{(n)}_\varepsilon)\cdot \nabla  V^n)(t',X_{t'})
-
\nabla_a(b^\infty\cdot \nabla  V^\infty)(t',X_{t'})|d\varepsilon\right]
|\Delta_n(t',X_{t'})|
dt'\\
&\enspace\enspace+\E\int_t^{T_{\cO}}|\nabla_a(b^\infty\cdot \nabla  V^\infty)(t',X_{t'})||\Delta_n(t',X_{t'})- \nabla^2\psi^*(Z)[Z'](t',X_{t'})|dt'
\end{align*}
Thus
\begin{align*}
|I_n^{(1)}| & \leq\E\int_t^{T_{\cO}}
\left[\int_0^1|\nabla_ab(t',X_{t'},\mathcal{I}^{(n)}_\varepsilon(t',X_{t'}))
- \nabla_ab^\infty(t',X_{t'})|
|\nabla  V^n(t',X_{t'})|d\varepsilon\right]
|\Delta_n(t',X_{t'})|
dt'\\
&\enspace\enspace+K\|\nabla  V^n-\nabla  V^\infty\|_{C^0(\cO_T)}\E\int_t^{T_{\cO}}
|\Delta_n(t',X_{t'})|
dt'\\
&\enspace\enspace
+K\|\nabla  V^{\infty}\|_{C^0(\cO_T)}\E\int_t^{T_{\cO}}|\Delta_n(t',X_{t'})- \nabla^2\psi^*(Z)[Z'](t',X_{t'})|dt'.
\end{align*}
Since the sequence $\{Z'_n\}_{n\in\mathbb{N}}$ converges to $Z'$ in $B_b(\cO_T;\mathbb{R}^p)$, it is uniformly bounded by some constant $M := 1+\sup_{n\in\mathbb{N}} \|Z'_n\|_{B_b(\cO_T;\mathbb{R}^p)} < \infty$.
Due to Theorem~\ref{thm:pde_well_posedness},~\eqref{eqn:embedding_Holder_Sob} and Assumption~\ref{ass:ham_structure}, there exists a constant $C>0$ such that for any $n\in\mathbb{N}$
\begin{equation}
\label{eqn:temp1}
\begin{split}
\|\nabla  V^n\|_{C^0(\cO_T)} & \leq C(1+\|Z+h_nZ'_n\|_{B_b(\cO_T;\mathbb{R}^p)})
\leq C(1+\|Z\|_{B_b(\cO_T;\mathbb{R}^p)} + M )\\
& \leq C M (1+\|Z\|_{B_b(\cO_T;\mathbb{R}^p)})
\end{split}
\end{equation}
and
\begin{equation}
\label{eqn:temp2}
\begin{split}
\|\nabla  V^n-\nabla  V^\infty\|_{C^0(\cO_T)}&\leq Ch_n(1+\|Z\|_{B_b(\cO_T;\mathbb{R}^p)}+\|Z'\|_{B_b(\cO_T;\mathbb{R}^p)})\|Z_n'\|_{B_b(\cO_T;\mathbb{R}^p)}\\
&\leq C M h_n(1+\|Z\|_{B_b(\cO_T;\mathbb{R}^p)}+\|Z'\|_{B_b(\cO_T;\mathbb{R}^p)})\,.
\end{split}
\end{equation}
From Lemma \ref{lem:cont_u_to_V_u}, Assumption \ref{ass:lin_growth_derivative} and~\eqref{eqn:temp1}-\eqref{eqn:temp2} it follows that 
\begin{align*}
|I_n^{(1)}|&\leq C M (1+\|Z'\|+\|Z\|)\\
&\quad\times\bigg(\|\Delta_n\|_{B_b(\cO_T;A)}\E\int_t^{T_\cO}|\nabla\psi^*((Z+h_nZ'_n)(t',X_{t'}))-\nabla\psi^*(Z(t',X_{t'}))|dt'\\
&\quad +h_n\|\Delta_n\|_{B_b(\cO_T;A)}+\|\Delta_n- \nabla^2\psi^*(Z)[Z']\|_{B_b(\cO_T;A)}\bigg)\\
&\leq C M (1+\|Z'\|+\|Z\|)\bigg(\|\Delta_n\|_{B_b(\cO_T;A)}h_n\|Z'_n\|_{B_b(\cO_T;\mathbb{R}^p)}+h_n\|\Delta_n\|_{B_b(\cO_T;A)}\\
&\quad +\|\Delta_n- \nabla^2\psi^*(Z)[Z']\|_{B_b(\cO_T;A)}\bigg)\\
&\leq C M^2 (1+\|Z'\|+\|Z\|)\bigg(2\|\Delta_n\|_{B_b(\cO_T;A)}h_n+\|\Delta_n- \nabla^2\psi^*(Z)[Z']\|_{B_b(\cO_T;A)}\bigg).
\end{align*}
To see why the right hand side above converges to zero 
note that from Assumption \ref{ass:mirror_reg}, $\|\Delta_n\|_{B_b(\cO_T;A)}\leq K\|Z_n'\| \leq K M$.
Also $\|\Delta_n-\nabla^2\psi^*(Z)[Z']\|_{B_b(\cO_T;A)}$ converges to zero as $n\rightarrow\infty$ due to Lemma \ref{lem:mirror_diff}.

To prove the convergence of  $I^{(2)}_n$ we proceed as follows.
From
Assumption  \ref{ass:lin_growth_derivative} for any $(t,x)\in\cO_T$ we have that 
\begin{align*}
&f^n(t,x)-f^{\infty}(t,x)=\int_0^1\nabla_af(t,x,\mathcal{I}^\varepsilon_n(t,x))\cdot[\nabla\psi^*((Z+h_nZ_n)(t,x))-\nabla\psi^*(Z(t,x))]d\varepsilon.
\end{align*}
Using similar arguments as for $I^{(1)}_n$
it can be shown that 
\begin{align*}
|I^{(2)}_n|\leq C M^2 h_n+K\|\Delta_n-\nabla^2\psi^*(Z)[Z']\|_{B_b(\cO_T;A)},
\end{align*}
where the constant $C$ is independent of $n$.
For $I^{(3)}_n$ note that  
\begin{align*}
\rho^n(t,x)-\rho^\infty(t,x)
&=\rho^{\nabla\psi^*(Z+h_n Z'_n)}(t,x)-\rho^{\nabla\psi^*(Z)}(t,x)\\
&= 
\int_0^1\mathcal{G}^n_\varepsilon(t,x)\cdot(\nabla\psi^*(Z+h_nZ_n')-\nabla\psi^*(Z))(t,x)d\varepsilon,
\end{align*}
where $\mathcal{G}^{n}_\varepsilon(t,x)=\nabla\rho^{v_\varepsilon^n(Z,Z'_n)}(t,x)$ and $v_\varepsilon^n(y,y')(t,x)=\varepsilon\nabla\psi^*((y+h_ny'))+(1-\varepsilon)\nabla\psi^*(y)$.
Therefore 
\begin{align*}
|I^{(3)}_n| &\leq \E\int_t^{T_{\cO}}
\int_0^1|\mathcal{G}^n_\varepsilon(t',X_{t'})- \nabla\rho^\infty(t',X_{t'})||\Delta_n(t',X_{t'})|d\varepsilon\\
&\enspace\enspace
+\E\int_t^{T_\cO}
|\nabla \rho^\infty(t',X_{t'}))||\Delta_n(t',X_{t'})-\nabla^2\psi^*(Z)[Z'](t',X_{t'})|dt'\\
&\leq
\|\Delta_n\|_{B_b(\cO_T;A)}
\E\int_t^{T_\cO}\int_0^1|(\mathcal{G}^n_\varepsilon-\nabla\rho^\infty)(t',X_{t'})|d\varepsilon dt'\\
&\quad+\|\nabla \rho^\infty\|_{B_b(\cO_T;\mathbb{R}^p)}
\|\Delta_n-\nabla^2\psi^*(Z)[Z']\|_{B_b(\cO_T;A)}.
\end{align*}
To show the above converges to zero we need to prove $\E\int_t^{T_\cO}|\mathcal{G}^n_\varepsilon(t',X_{t'}) - \nabla\rho^\infty(t',X_{t'})|dt'\rightarrow 0$ as $n\rightarrow\infty$. 
To that end note that from Assumption \ref{ass:mirror_reg} for all $\varepsilon\in[0,1]$, $(t,x)\in\cO_T$ and $n\in\mathbb{N}$ we have 
\begin{align*}
|\mathcal{G}^n_\varepsilon(t,x)-\nabla \rho^\infty(t,x)|&\leq|\nabla \rho^\infty(t,x)|+|\mathcal{G}^n_\varepsilon(t,x)|\\
&\leq K(1+\|Z\|_{B_b(\cO_T;\mathbb{R}^p)}+\|Z'\|_{B_b(\cO_T;\mathbb{R}^p)})+|\nabla\psi(\nabla\psi^*(Z(t,x)))|\\
&\leq K(1+\|Z\|_{B_b(\cO_T;\mathbb{R}^p)}+\|Z'\|_{B_b(\cO_T;\mathbb{R}^p)})+\|\mathcal{C}(Z)\|_{B_b(\cO_T;\mathbb{R}^p)}.
\end{align*}
The Dominated Convergence Theorem then asserts that  
\begin{align*}
\lim_{n\rightarrow\infty} \E\int_t^{T_\cO}&\int_0^1|\mathcal{G}^n_\varepsilon(t',X_{t'})-\nabla\rho^\infty(t',X_{t'})|d\varepsilon dt'\\
&=\E\int_t^{T_\cO}\int_0^1\lim_{n\rightarrow\infty}|\mathcal{G}^n_\varepsilon(t',X_{t'})-\nabla\rho^\infty(t',X_{t'})|d\varepsilon dt'.
\end{align*}
Recall that for any $(t,x)\in\cO_T$ we have 
\begin{align*}
\mathcal{G}^n_\varepsilon(t,x)-\nabla\rho(\nabla\psi^*(Z(t,x)))
=\nabla\psi(v_\varepsilon^n(Z(t,x),Z'_n(t,x)))-\nabla\psi(\nabla\psi^*(Z(t,x))),
\end{align*}
therefore it suffices to prove 
\begin{align}
\label{eqn:uni_conv}
\lim_{n\rightarrow\infty} \sup_{\varepsilon\in[0,1]}\sup_{t,x\in\cO_T}|\nabla\psi(v_\varepsilon^n(Z(t,x),Z'_n(t,x)))-\nabla\psi(\nabla\psi^*(Z(t,x)))|=0.
\end{align}
Firstly, 
since $\lim_{n\rightarrow\infty}Z'_n=Z'$ in $B_b(\cO_T;\mathbb{R}^p)$ the set $\{(Z+h_nZ'_n)(t,x):n\in\mathbb{N},(t,x)\in\cO_T\}$ is contained in a compact in $\mathbb{R}^p$.
Therefore using the continuity of $\nabla\psi^*$ the sets
$\{\nabla\psi^*((Z+h_nZ'_n)(t,x)): n\in\mathbb{N}, (t,x)\in\cO_T\}$ and $\{\nabla\psi^*(Z(t,x)):(t,x)\in\cO_T\}$ are contained in compact subsets in $\operatorname{ri}(A)$.
In particular from the uniform continuity of $\nabla\psi^*$ this implies $\{v_\varepsilon^n(Z(t,x),Z'_n(t,x))):n\in\mathbb{N},\epsilon\in[0,1],(t,x)\in\cO_T\}$ is contained in a compact subset of $\operatorname{ri}(A)$.
From Assumption~\ref{ass:SDE_well_posed} $\nabla\psi$ is uniformly continuous on compact subsets of $\operatorname{ri}(A)$, 
meaning that in order to prove \eqref{eqn:uni_conv} it suffices to prove $v_\varepsilon^n(Z,Z+h_nZ'_n)$ converges to $v_\varepsilon^n(Z,Z)=\nabla\psi^*(Z)$ uniformly.
To this end we have 
\begin{align*}
|v_\varepsilon^n(Z(t,x),(Z+h_nZ'_n)(t,x)) - \nabla\psi^*(Z(t,x))|&=\varepsilon|\nabla\psi^*((Z+h_nZ'_n)(t,x))-\nabla\psi^*(Z(t,x))|\\
&\leq K\varepsilon h_n|Z'_n(t,x)|.
\end{align*}
Taking supremums over $\varepsilon\in[0,1]$ and $(t,x)\in\cO_T$ and recalling $\lim_{n\rightarrow\infty}Z'_n=Z'$ in $B_b(\cO_T;\mathbb{R}^p)$ concludes the proof.
\end{proof}

\section{Existence of unique solutions and convergence of the mirror flow}
\label{sec:well_posed_decrease_conv}
This section is devoted to proving the well-posedness and convergence of the mirror flow \eqref{eqn:grad_flow}.
The regularity properties from Section \ref{sec:reg_and_pd} will be used to show the map $B_b(\cO_T;\mathbb{R}^p)\ni Z\mapsto \nabla_a H(\cdot, \nabla V^{\nabla\psi^*(Z)},\nabla\psi^*(Z))\in B_b(\cO_T;\mathbb{R}^p)$ is locally Lipschitz continuous.
Theorem \ref{thm:grad_flow_well_posed} then follows from a truncation argument.
Define $\mathcal{I}:B_b(\cO_T;\mathbb{R}^p)\rightarrow B_b(\cO_T;\mathbb{R}^p)$ by 
\begin{align}
\mathcal{I}(Z)=-\nabla_aH(\cdot,\cdot,\nabla V^{\nabla\psi^*(Z)},\nabla\psi^*(Z)),
\end{align}
and for each $N>0$ define the truncated operator $\mathcal{I}_N:B_b(\cO_T;\mathbb{R}^p)\rightarrow B_b(\cO_T;\mathbb{R}^p)$
\begin{align*}
\mathcal{I}_N(Z)=\begin{cases}
\mathcal{I}(Z) & \|Z\|_{B_b(\cO_T;\mathbb{R}^p)}\leq N\\
\mathcal{I}\bigg(\tfrac{N Z}{\|Z\|_{B_b(\cO_T;\mathbb{R}^p)}}\bigg) & \|Z\|_{B_b(\cO_T;\mathbb{R}^p)}> N
\end{cases}.
\end{align*}
\begin{lemma}
\label{lem:well_posed_trunc}
Let Assumptions \ref{ass:SDE_well_posed},  
\ref{ass:mirror_reg} and
\ref{ass:lin_growth_derivative} hold. 
Then 
for each $N>0$ and $Z^0\in B_b(\cO_T;\mathbb{R}^p)$
there exists a unique $Z\in\cap_{S>0}C^{1}([0,S];B_b(\cO_T;\mathbb{R}^p))$ such that $\tfrac{d}{ds}Z_s=\mathcal{I}_N(Z)$ with $Z_0 = Z^0$.
\end{lemma}
\begin{proof}
Throughout this proof $\|\cdot\|:=\|\cdot\|_{B_b(\cO_T;\mathbb{R}^p)}$. 
We will first show that the map $B_b(\cO_T;\mathbb{R}^p)\ni Z\mapsto \mathcal{I}_N(Z)$ is globally Lipschitz. 
Using repeated indices to represent summations we can write
\begin{equation}
\label{eqn:diff_ops}
\begin{split}
\mathcal{I}(Z)&-\mathcal{I}(Z')\\
&
=
(\nabla_{a_j}b_i^{\nabla\psi^*(Z')}-\nabla_{a_j}b_i^{\nabla\psi^*(Z)})
\nabla_{x_i}V^{\nabla\psi^*(Z)} 
+ \nabla_{a_j}b_i^{\nabla\psi^*(Z')}
(\nabla_{x_i}V^{\nabla\psi^*(Z')}-\nabla_{x_i}V^{\nabla\psi^*(Z)}) \\
&\enspace\enspace+ (\nabla_{a_j}f^{\nabla\psi^*(Z')} - \nabla_{a_j}f^{{\nabla\psi^*(Z)}})+\tau(\nabla_{a_j}\rho^{\nabla\psi^*(Z')} - \nabla_{a_j}\rho^{{\nabla\psi^*(Z)}}).
\end{split}
\end{equation}
From Assumption \ref{ass:mirror_reg} we have that
\begin{align*}
\nabla\rho^{\nabla\psi^*(Z)} - \nabla\rho^{{\nabla\psi^*(Z')}}
=
\nabla\psi(\nabla\psi^*(Z))-\nabla\psi(\nabla\psi^*(Z'))=\mathcal{C}(Z)-\mathcal{C}(Z').
\end{align*}
From Assumption \ref{ass:lin_growth_derivative} we thus get 
\begin{align*}
\|\mathcal{I}(Z)-\mathcal{I}(Z')\|
&\leq 
K(\|\nabla V^{\nabla\psi^*(Z)}\|_{B_b(\cO_T;\mathbb{R}^d)}+1)\|\nabla\psi^*(Z)-\nabla\psi^*(Z')\|\\
&\enspace\enspace+K\|\nabla V^{\nabla\psi^*(Z')}-\nabla V^{\nabla\psi^*(Z)}\|_{B_b(\cO_T;\mathbb{R}^d)}+\tau\|Z-Z'\|.
\end{align*}
From the embedding \eqref{eqn:embedding_Holder_Sob} and Theorem \ref{thm:pde_well_posedness}, we have $\|\nabla V^{\nabla\psi^*(Z)}\|_{B_b(\cO_T;\mathbb{R}^d)} \leq C(1 + \tau\|\rho^{\nabla\psi^*(Z)}\|_{B_b(\cO_T;\mathbb{R})})$.
By expanding the Bregman divergence, using the fact that the action space $A$ is bounded, and applying Assumption~\ref{ass:mirror_reg} (which guarantees $|\psi(\nabla \psi^*(y))| \leq K(1+|y|))$, this is further bounded by $C(1 + \tau(1+\|Z\|))$.
Combining this with Lemma~\ref{lem:cont_u_to_V_u}, and using the fact that $\nabla\psi^*$ is $K$-Lipschitz (Assumption \ref{ass:mirror_reg}), there exists a generic constant $C>0$ such that for any $Z,Z'\in B_b(\cO_T;\mathbb{R}^p)$:
\begin{align*}
\|\mathcal{I}(Z)-\mathcal{I}(Z')\|
&\leq 
C(1+\tau+\tau\|Z\|)
\|Z-Z'\|\\
&\enspace\enspace+C(1+\tau)(1+\|Z\|+\|Z'\|)\|Z-Z'\|+\tau\|Z-Z'\|\\
&\leq C(1+\tau)(1+\|Z\|+\|Z'\|)\|Z-Z'\|.
\end{align*}
In particular  there exists a constant $C>0$ such that for any $Z,Z'\in B_b(\cO_T;\mathbb{R}^p)$
$\|\mathcal{I}_N(Z)-\mathcal{I}_{N}(Z')\|\leq C(1+\tau)(1+2N)\|Z-Z'\|.$
Let $S>0$ be arbitrary but fixed.
Let $X_\eta := C([0,S];B_b(\cO_T;\mathbb{R}^p))$ be equipped with the norm $\|k\|_\eta := \sup_{s\in[0,S]}e^{-\eta s}\|k_s\|$ and define $\Psi:X_\eta \rightarrow X_\eta$ by $\Psi(Z)_s=Z^0+\int_0^s\mathcal{I}_N(Z_{s'})ds'$.
Note that $X_\eta$ is a Banach space.
We will show that $\Psi$ is a contraction on $X_\eta$ for an appropriately chosen $\eta>0$. 
To that end let $C_N=C(1+\tau)(1+2N)$.
Then for any $s>0$
\begin{align*}
\|\Psi(Z)_s-\Psi(Z')_s\|&\leq\int_0^s\|\mathcal{I}_N(Z_{s'})-\mathcal{I}_N(Z'_{s'})\|ds'
\leq C_N\int_0^se^{-\eta s'}\|Z_{s'}-Z'_{s'}\|e^{\eta s'}ds'\\
&\leq C_N\|Z-Z'\|_{\eta} \eta^{-1}(e^{\eta s}-1)
\leq \eta^{-1}C_N\|Z-Z'\|_{\eta}e^{\eta s}.
\end{align*}
Multiplying through by $e^{-\eta s}$, 
taking supremums over $[0,S]$
and choosing $\eta = C_{N} + 1$ implies $\|\Psi(Z)-\Psi(Z')\|_{\eta}\leq\tfrac{C_{N}}{C_{N}+1}\|Z-Z'\|_{\eta}$.
Therefore from Banach's fixed point theorem there exists a unique $Z\in (C([0,S],B_b(\cO_T;\mathbb{R}^p)),\|\cdot\|_{\eta})$ such that $Z_s = Z^0 + \int_0^s\mathcal{I}_N(Z_{s'})ds'$.
Noting that the norms $\|\cdot\|_{\eta}$ and $\|\cdot\|_{0}$ are equivalent we have shown $Z\in (C([0,S];B_b(\cO_T;\mathbb{R}^p)),\|\cdot\|_{0})$.
Also
since it is easy to verify the map $s\mapsto \mathcal{I}_N(Z_s)$ is continuous
the fundamental theorem of calculus implies $u$ is differentiable, i.e. $Z\in C^{1}([0,S];B_b(\cO_T;\mathbb{R}^p))$.
Since $S$ was arbitrary we have that $Z\in\cap_{S>0}C^{1}([0,S];B_b(\cO_T;\mathbb{R}^p))$
This concludes the proof.
\end{proof}

\begin{lemma}
\label{lem:existence_aprior}
Suppose Assumptions \ref{ass:SDE_well_posed},  
\ref{ass:mirror_reg} and
\ref{ass:lin_growth_derivative} hold and $Z\in \cap_{S>0}C([0,S];B_b(\cO_T;\mathbb{R}^p))$ satisfies $\tfrac{d}{ds}Z_s = -\mathcal{I}(Z_s)$ with $Z_0=Z^0$ for some $Z^0\in B_b(\cO_T;\mathbb{R}^p)$.
Then there exists $C>0$ such that for any $s>0$,
$\|Z_s\|\leq (\|Z^0\|+s)e^{Cs}$.
\end{lemma}
\begin{proof}
Let $\|\cdot\|=\|\cdot\|_{B_b(\cO_T;\mathbb{R}^p)}$.
Then 
\begin{align}
\label{eqn:aprioi_bound_1}
\begin{split}
\|Z_s\|&\leq\|Z_0\|+\int_0^s\|\nabla_aH(\cdot,\nabla V^{\nabla\psi^*(Z_{s'})}),\nabla\psi^*(Z_{s'})\|ds'\\
&\leq\|Z_0\|+\int_0^s K\|\nabla V^{\nabla\psi^*(Z_{s'})}\| + \|\nabla_af^{\nabla\psi^*(Z_{s'})}\|+\|\nabla\rho(\nabla\psi^*(Z_{s'}))\|ds'.
\end{split}
\end{align}
From Assumption \ref{ass:mirror_reg}, $\|\nabla\rho(\nabla\psi^*(Z_{s'}))\|\leq K(1+\|Z_{s'}\|)$. 
Therefore Theorem~\ref{thm:pde_well_posedness} asserts there exists a constant independent of $s$ such that $\|\nabla V^{\nabla\psi^*(Z_{s'})}\|\leq C(1+\|\nabla\rho(\nabla\psi^*(Z_{s'}))\|)\leq C(1+\|Z_{s'}\|)$.
Substituting the above into \eqref{eqn:aprioi_bound_1} implies $\|Z_s\|\leq(\|Z_0\|+s)+C\int_0^s\|Z_{s'}\|ds'$. 
Gr{\"o}nwall's inequality concludes the proof.
\end{proof}

\begin{proof}
[Proof of Theorem \ref{thm:grad_flow_well_posed}]
Fix some $S>0$ and set $M>0$ equal to the constant from Lemma \ref{lem:existence_aprior}, i.e. $M=(\|Z_0\|_{B_b(\cO_T;\mathbb{R}^p)}+S)e^{CS}$.
From Lemma \ref{lem:well_posed_trunc} there exists a unique $\tilde{Z}\in\cap_{S>0}C^1([0,S];B_b(\cO_T;\mathbb{R}^p))$ such that $\tfrac{d}{ds}\tilde{Z}_s= \mathcal{I}_{2M}(\tilde{Z}_s)$.
Define $S_M=\inf\{S>0:\|\tilde{Z}\|_{B_b(\cO_T;\mathbb{R}^p)}\geq 2M\}$. 
For all $s\in [0,S_M]$, $\|\tilde{Z}_s\|_{B_b(\cO_T;\mathbb{R}^p)}\leq 2M$ and therefore $\tfrac{d}{ds}\tilde{Z}_s= \mathcal{I}(\tilde{Z}_s)$.
We claim that $S_M\geq S$.
To that end assume for the moment $S_{M}<S$.
Then we are led to the following contradiction
$2M\leq \sup_{s\in[0,S]}\|\tilde{Z}_s\|_{B_b(\cO_T;\mathbb{R}^p)}\leq (\|Z_0\|_{B_b(\cO_T;\mathbb{R}^p)}+S)e^{CS}=M$,
where the first inequality follows from the definition of $S_M$ and the assumption $S_M<S$ and the second inequality follows from Lemma \ref{lem:existence_aprior}.
Therefore we have shown that for any $S>0$, there exists $Z\in C^1([0,S];B_b(\cO_T;\mathbb{R}^p))$ such that $\tfrac{d}{ds}Z_s = \mathcal{I}(Z_s)$, with $Z_0 = Z^0$.
\end{proof}

\begin{proof}
[Proof of Theorem \ref{thm:decreasing_cost_funct}]
The map $[0,\infty)\ni s\mapsto V^{\nabla\psi^*(Z_s)}\in W^{2,1}_q(\cO_T)$ can be viewed as the composition of the maps
$[0,\infty)\ni s\mapsto Z_s\in B_b(\cO_T;\mathbb{R}^p)$ and 
$B_b(\cO_T;\mathbb{R}^p)\ni Z\mapsto V^{\nabla\psi^*(Z)}\in W^{2,1}_q(\cO_T)$.
The first map is differentiable by the assumptions in the statement of the theorem and the second is Hadamard differentiable by Lemma \ref{lem:had_diff}.
Therefore for each $(t,x)\in\cO_T$ the chain rule asserts that 
\begin{align*}
&\tfrac{d}{ds}V^{\nabla\psi^*(Z_s)}(t,x)
=\partial V^{\nabla\psi^*(Z_s)}(\tfrac{d}{ds}\nabla\psi^*(Z_s))(t,x)
=\partial V^{\nabla\psi^*(Z_s)}(\nabla^2\psi^*(Z_s)[\tfrac{d}{ds}Z_s])(t,x)\\
&=\E^{t,x,\nabla\psi^*(Z_s)} \int_0^{T_\cO}\nabla_a H(\cdot,\nabla V^{\nabla\psi^*(Z_s)},\nabla\psi^*(Z_s))(t',X_{s,t'})\cdot\nabla^2\psi^*(Z_s)[\tfrac{d}{ds}Z_s](t',X_{s,t'})dt'\\
&=-\E^{t,x,\nabla\psi^*(Z_s)} \int_0^{T_\cO}
[\tfrac{d}{ds}Z_s]^\top(t',X_{s,t'})\cdot\nabla^2\psi^*(Z_s)[\tfrac{d}{ds}Z_s](t',X_{s,t'})dt'
\end{align*}
where $(X_{s,t'})_{t'\geq t}$ is the weak solution to the controlled SDE corresponding to the control $\nabla\psi^*(Z_s)$.
From Assumption \ref{ass:mirror_reg}, $\psi^*$ is convex and twice continuously differentiable. 
In particular its Hessian is positive semi definite, meaning that $\tfrac{d}{ds}V^{\nabla\psi^*(Z_s)}(t,x)\leq 0$ for all $(t,x)\in\cO_T$.
\end{proof}

\begin{proof}
[Proof of Theorem \ref{thm:lin_conv}]
For each $s\geq 0$ let $u_s=\nabla\psi^*(Z_s)$.
Since $\bar Z$ is bounded, $\bar u\in\mathcal{U}^\tau_\psi$; for $\tau>0$ this follows from Assumption~\ref{ass:SDE_well_posed} and~\ref{ass:mirror_reg} exactly as for $u_s$ in the proof of Lemma~\ref{lem:well_posed_trunc}.
In $\mathcal{D}(Z_s,\bar Z)$, see~\eqref{eqn:bregman}, the expectation is with respect to the controlled process with the fixed control $\bar u$, which does not depend on $s$.
Since $Z\in C^1([0,S];B_b(\cO_T;\mathbb{R}^p))$, $\nabla\psi^*$ is bounded and Lipschitz (Assumption~\ref{ass:mirror_reg}) and $T_\cO-t\le T$, we may differentiate under the expectation and obtain
\begin{equation}
\label{eqn:chain_rule_bregman}
\begin{split}
\tfrac{d}{ds}\mathcal{D}(Z_s,\bar Z)(t,x)&=
\E^{\bar u}_{t,x} \int_t^{T_\cO}
\tfrac{d}{ds}\Big(\psi^*(Z_s)-\psi^*(\bar Z)
-\nabla\psi^*(\bar Z)\cdot(Z_s-\bar Z)
\Big)(t',X_{t'})
dt'\\
&=\E^{\bar u}_{t,x} \int_t^{T_\cO}
\Big[\big(\nabla\psi^*(Z_s)-\nabla\psi^*(\bar Z)\big)\cdot\tfrac{d}{ds}Z_s\Big](t',X_{t'})
dt'.
\end{split}
\end{equation}
Using the definition of the mirror flow~\eqref{eqn:grad_flow}, Assumption~\ref{ass:ham_convex} with $a=\bar u$ and $a'=u_s$, and Lemma~\ref{lem:perf_diff} with $u=u_s$ and $u'=\bar u$, we get
\begin{equation}
\label{eqn:lyap_bound}
\begin{split}
\tfrac{d}{ds}\mathcal{D}(Z_s,\bar Z)(t,x)
&=
\E^{\bar u}_{t,x} \int_t^{T_\cO}
\Big[ \big(\bar u-u_s\big)\cdot \nabla_aH(\cdot,\nabla V^{u_s},u_s)\Big](t',X_{t'})
dt'\\
&\leq
\E^{\bar u}_{t,x} \int_t^{T_\cO}
\Big[H(\cdot,\nabla V^{u_s},\bar u)-H(\cdot,\nabla V^{u_s},u_s)\Big](t',X_{t'})
dt'\\
&\quad-\tfrac{\lambda}{2}\E^{\bar u}_{t,x}\int_t^{T_\cO}D_\psi\big(\bar u(t',X_{t'}),u_s(t',X_{t'})\big)dt'\\
&=\left(V^{\bar u}-V^{u_s}\right)(t,x)-\tfrac{\lambda}{2}\mathcal{D}(Z_s,\bar Z)(t,x).
\end{split}
\end{equation}
In the last equality the first term is given by Lemma~\ref{lem:perf_diff}. For the second term (only needed when $\lambda>0$) we used the assumption $D_{\psi}(\nabla\psi^*(y),\nabla\psi^*(y'))=D_{\psi^*}(y',y)$ with $y=\bar Z$ and $y'=Z_s$, i.e.\ $D_\psi(\bar u,u_s)=D_{\psi^*}(Z_s,\bar Z)$, which is the integrand in~\eqref{eqn:bregman}.

If $\lambda=0$ then integrating~\eqref{eqn:lyap_bound} from $0$ to $S$ and using the non-negativity of~\eqref{eqn:bregman} gives
\begin{align*}
\int_0^S \left(V^{u_s}-V^{\bar u}\right)(t,x)\,ds\leq \mathcal{D}(Z_0,\bar Z)(t,x)-\mathcal{D}(Z_S,\bar Z)(t,x)\leq\mathcal{D}(Z_0,\bar Z)(t,x).
\end{align*}
By Theorem \ref{thm:decreasing_cost_funct}, $V^{u_S}(t,x)\leq V^{u_s}(t,x)$ for all $s\in[0,S]$, and hence
\begin{align*}
S\left(V^{u_S}-V^{\bar u}\right)(t,x)
\leq \int_0^S \left(V^{u_s}-V^{\bar u}\right)(t,x)\,ds\leq \mathcal{D}(Z_0,\bar Z)(t,x).
\end{align*}
This concludes the proof of \eqref{eqn:lin_conv}.
If $\lambda>0$ then~\eqref{eqn:lyap_bound} implies
\begin{align}
\label{eqn:help}
\tfrac{d}{ds}\Big(e^{\tfrac{\lambda}{2}s}\mathcal{D}(Z_s,\bar Z)(t,x)\Big) \leq
e^{\tfrac{\lambda}{2}s}\left(V^{\bar u}-V^{u_s}\right)(t,x).
\end{align}
Integrating from $0$ to $S$ gives
\begin{align}
\label{eqn:help_integrated}
e^{\tfrac{\lambda}{2}S}\mathcal{D}(Z_S,\bar Z)(t,x)-\mathcal{D}(Z_0,\bar Z)(t,x)\leq \int_0^Se^{\tfrac{\lambda}{2}s}\left(V^{\bar u}-V^{u_s}\right)(t,x)\,ds.
\end{align}
Using the non-negativity of~\eqref{eqn:bregman} and Theorem \ref{thm:decreasing_cost_funct} as before,
\begin{align*}
\left(V^{u_S}-V^{\bar u}\right)(t,x)\int_0^Se^{\tfrac{\lambda}{2}s}\,ds
\leq\int_0^Se^{\tfrac{\lambda}{2}s}\left(V^{u_s}-V^{\bar u}\right)(t,x)\,ds\leq \mathcal{D}(Z_0,\bar Z)(t,x).
\end{align*}
Since $\int_0^Se^{\frac{\lambda}{2}s}ds=\tfrac{2}{\lambda}\big(e^{\frac{\lambda}{2}S}-1\big)$, this proves~\eqref{eqn:exp_con}.
Finally, by Theorem~\ref{thm:decreasing_cost_funct}, $\left(V^{\bar u}-V^{u_s}\right)(t,x)\leq\left(V^{\bar u}-V^{u_S}\right)(t,x)$ for $s\in[0,S]$, so~\eqref{eqn:help_integrated} gives
\begin{align*}
e^{\tfrac{\lambda}{2}S}\mathcal{D}(Z_S,\bar Z)(t,x)\leq\mathcal{D}(Z_0,\bar Z)(t,x)+\tfrac{2}{\lambda}\big(e^{\tfrac{\lambda}{2}S}-1\big)\left(V^{\bar u}-V^{u_S}\right)(t,x).
\end{align*}
Multiplying by $e^{-\frac{\lambda}{2}S}$ proves~\eqref{eqn:strong_convex_conv}.
\end{proof}

\section{Numerical experiments}
\label{sec numerical experiments}

In this section, we present numerical experiments to demonstrate the practical implementation and convergence properties of Algorithm~\ref{alg:dynamic_md} which is a possible discretization of the continuous-time mirror descent flow~\eqref{eqn:grad_flow}. 

\subsection{Implementation details}
For a given Markov control $u_k$, solving the on-policy Bellman equation~\eqref{eqn:onpolicy_bellman} requires solving a linear second-order parabolic PDE.
We discretize the state space $\cO$ using a regular grid and apply an upwind finite difference scheme to approximate the spatial derivatives.
Time is discretized uniformly with step size $dt$ using an implicit Euler scheme for stability.
At each time step, this yields a sparse linear system which can be solved efficiently (we use the iterative stabilized biconjugate gradient method with previous time step value as initial guess) to yield the approximate value function $V_k$.
Here $k$ is the mirror descent step index.
The spatial gradient $\nabla V_k$ is then approximated using finite differences, enabling the dual space gradient descent step.
In both settings $d=2$, $\cO$ is the disc of radius $5$, $T=1$, $\sigma=I$, the drift is the action, the running cost is $\frac{1}{100}|a|^2$ and the terminal cost is $g(x)=2|x|^2$; for Example~\ref{sec:ex_LQR} this corresponds to $M_1=0$, $N=M_2=I$, $M_3=2I$ with running cost $\frac{1}{100}|a|^2$ in place of $\tfrac{|x|^2}{2}+\tfrac{|a|^2}{2}$, and for Example~\ref{sec:example_finite_actions} the actions are a uniform $21\times21$ grid on $[-10,10]^2$.

We test this framework on two distinct settings corresponding to the examples discussed in Section~\ref{sec:examples}:
\begin{enumerate}
\item \textbf{Continuous actions on a ball:} We consider the Linear Quadratic Regulator (LQR) problem with actions constrained to a ball of radius $R$, as formulated in Example~\ref{sec:ex_LQR}.
We compare the exact mirror descent (using the mirror map derived in Section~\ref{sec:ex_LQR}) against a Projected Gradient Descent (PGA) baseline, which updates the action in the primal space and explicitly projects it onto $A$ using Euclidean orthogonal projection.
\item \textbf{Soft policies over finite action spaces:} We consider the entropy-regularized relaxed control problem outlined in Example~\ref{sec:example_finite_actions}. The mirror map corresponds to the softmax function, and the dual update takes place in $\mathbb{R}^p$.
We compare this against a Soft Policy Iteration (SPI) baseline and against a PGA baseline.
The PGA first does $a \gets a - \eta \nabla_a H(\cdot, \nabla V, a)$ with $\eta > 0$ being the learning rate.
This updates action ``probabilities'' in the Euclidean space $\mathbb R^p$.
To ensure they are in $A = \mathcal P(\{1,\ldots,p\})$, we use an $L^2$ projection onto the simplex which has $\mathcal O(p \log p)$ cost.
\end{enumerate}

To isolate the performance of the optimization routines from FDM truncation errors, we evaluate the resulting controls via Monte Carlo (MC) rollouts using the Euler--Maruyama scheme. 
The parameters used for the experiments herein are in Tables~\ref{tab:dpp_hyperparameters} and~\ref{tab:dpp_discrete_hyperparameters} in Section~\ref{sec experiment parameters}.

\subsection{Initial validation}
As no closed form solution exists, we performed several validations of the results the algorithm produces. 
The system, control and cost trajectories are reasonable, see Figure~\ref{fig validation}.
The value function agrees with MC simulation (for the fixed control) and the policy vector field is as one would expect, see Figure~\ref{fig value and field}.
Finally, we compare in MC evaluation with the LQR policy for the unconstrained problem on the whole of $\mathbb R^2$, with its actions projected onto $A$, which takes no account of the exit cost.
From $x_0=3.5(\cos\frac{\pi}{4},\sin\frac{\pi}{4})$, where the boundary rarely matters, this baseline attains a slightly higher mean return, namely $-0.29 \pm 0.002$ against $-0.35 \pm 0.001$ (standard errors), which we attribute to the space and time discretization of the grid-based policy.
From $x_0=4.8(\cos\frac{\pi}{4},\sin\frac{\pi}{4})$, close to the boundary, the policy obtained from MD is much better, with mean return $-0.78\pm0.04$ against $-4.39\pm0.15$ (standard errors), consistent with the LQR policy leaving the domain, and so incurring the exit cost, much more often.
When started close to the boundary, the policy obtained also applies stronger actions early, see Figure~\ref{fig validation}.

\subsection{Convergence rates}
In Section~\ref{sec:well_posed_decrease_conv}, Theorem~\ref{thm:lin_conv} establishes continuous-time mirror flow convergence. 
To empirically evaluate the convergence for the discretized algorithm, we generate $\{V_k\}_{k=1}^N$ using Algorithm~\ref{alg:dynamic_md}.
We obtain a reference approximation $V^*$ using PIA run with as many iterations as needed so that it stops updating. 

We evaluate convergence using several distinct metrics:
\begin{enumerate}[i)]
\item \textbf{Value Distance:} We report $L^\infty$ distance and relative $L^2$ distance between $V_k$ and $V^*$.
\item \textbf{Return Gap:} The suboptimality gap in the true SDE environment, measured as $(\text{MC Return}(u^*) - \text{MC Return}(u_k))$.
\end{enumerate}

\begin{figure}[htpb]
\centering
\includegraphics[width=0.9\textwidth]{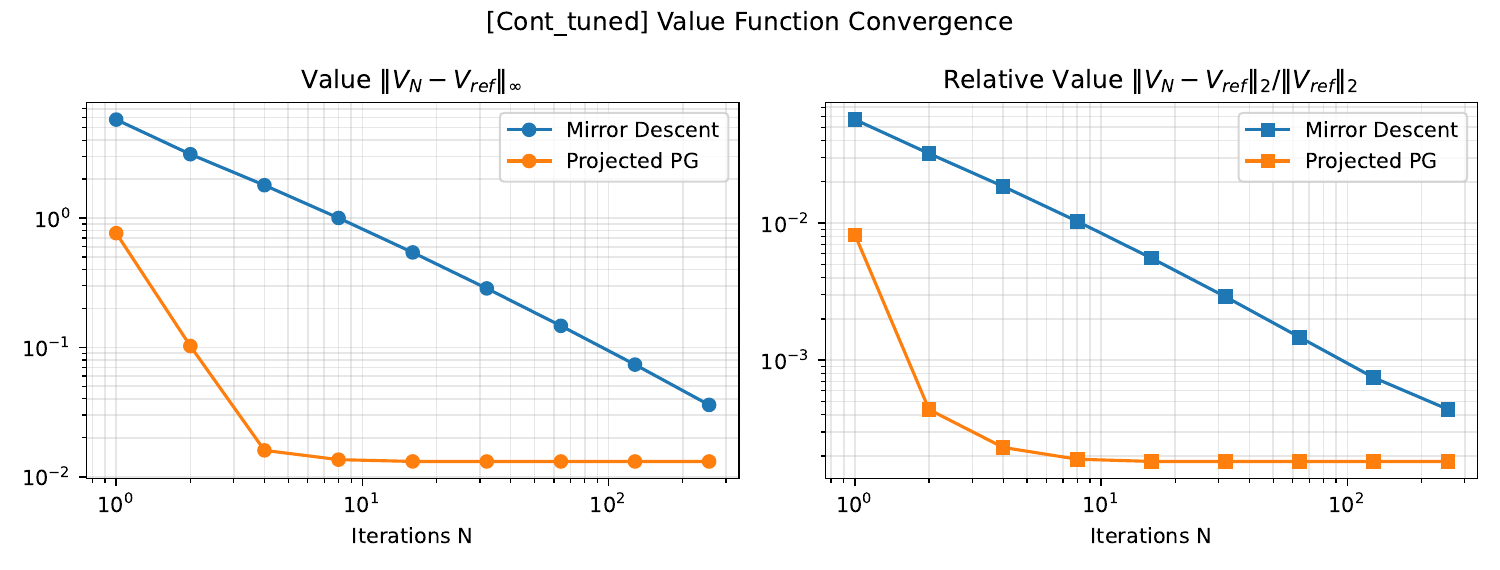}
\caption{Convergence of the value functions of Algorithm~\ref{alg:dynamic_md} and of projected PG towards the reference $V^*$ computed by policy iteration, in the setting of Section~\ref{sec:ex_LQR} with $\tau=0$, in the sup norm (left) and relative $L^2$ norm (right). Step sizes $\eta=0.03$ (mirror descent) and $30$ (projected PG), tuned separately for best performance after 30 iterations across a range of possible step sizes.}
\label{fig:convergence_value_cont}
\end{figure}

\begin{figure}[htpb]
\centering
\includegraphics[width=0.9\textwidth]{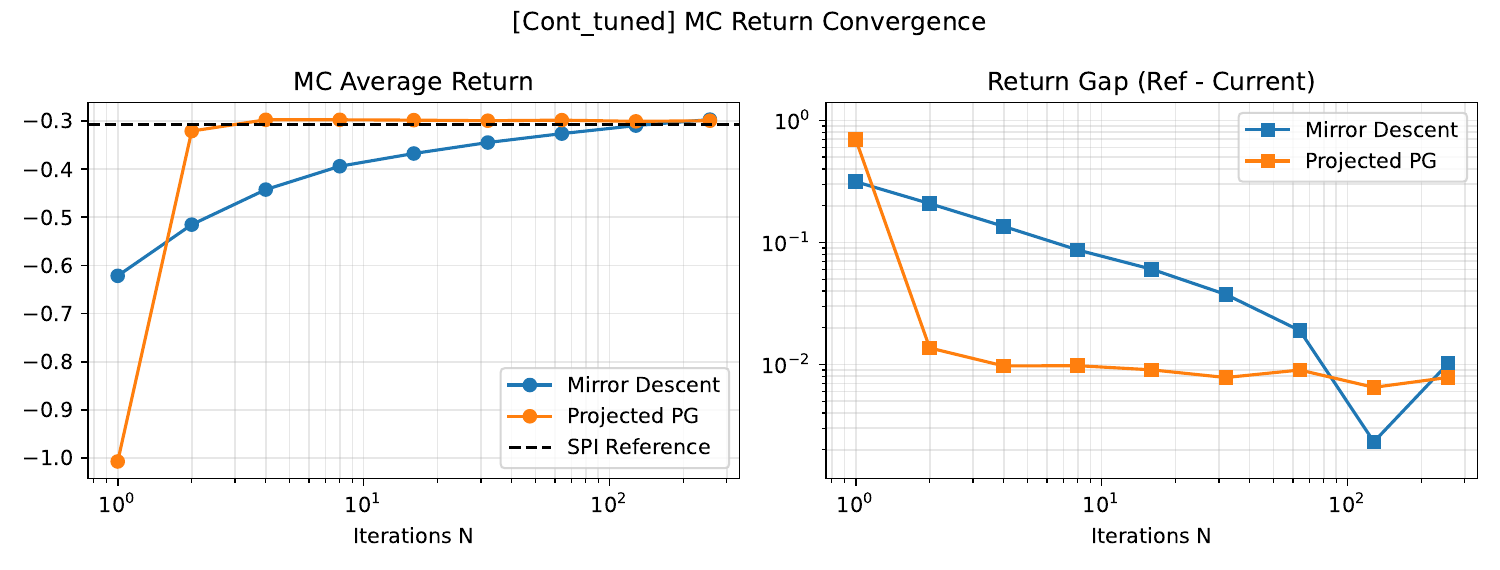}
\caption{Monte Carlo return gap in the true SDE environment for the methods and step sizes of Figure~\ref{fig:convergence_value_cont}. 
Both methods match the reference return to within about $0.01$, the level of the space and time discretization.}
\label{fig:convergence_return_cont}
\end{figure}

In the setting of Section~\ref{sec:ex_LQR},
Figures~\ref{fig:convergence_value_cont} and~\ref{fig:convergence_return_cont} plot these metrics across iterations $k$.
The step sizes of the two methods are not directly comparable: since $\nabla^2\psi^*(0)=\tfrac{R^2}{2}I$, near the origin a dual step of size $\eta$ moves the action by about $\tfrac{R^2}{2}\eta=50\eta$ times the gradient, and by much less close to the boundary of $A$, where $\nabla^2\psi^*$ decays.
We therefore tuned the step size of each method separately, by the Monte Carlo cost after $30$ iterations from $x_0=4.8(\cos\frac{\pi}{4},\sin\frac{\pi}{4})$ and use $\eta=0.03$ for mirror descent and $30$ for projected PG. 
Both reach the same cost within Monte Carlo error ($0.77\pm0.04$ and $0.73\pm0.04$).
On the log-log scale the value error of Algorithm~\ref{alg:dynamic_md} decays like $k^{-1}$: the local slope approaches $-1$.
Here $\tau=0$ and the optimal control takes values on the boundary of $A$ in part of the domain, so Theorem~\ref{thm:lin_conv} applies through Remark~\ref{rem:comparator}~ii); the mirror descent iterates approach the boundary as $Z_k$ grows, with $R-|u_k|\approx|Z_k|^{-1}=O\big((\eta k)^{-1}\big)$.
Projected PG converges much faster on this example: with a step size of the order of $1/(2\cdot 0.01)=50$, the inverse curvature of the running cost in the action, a projected gradient step is essentially a policy improvement step.
For a ball the Euclidean projection is trivial; the advantages of the mirror map are for sets where it is not, see Section~\ref{sec psi and self concordant barriers}.

\begin{figure}[htpb]
\centering
\includegraphics[width=0.9\textwidth]{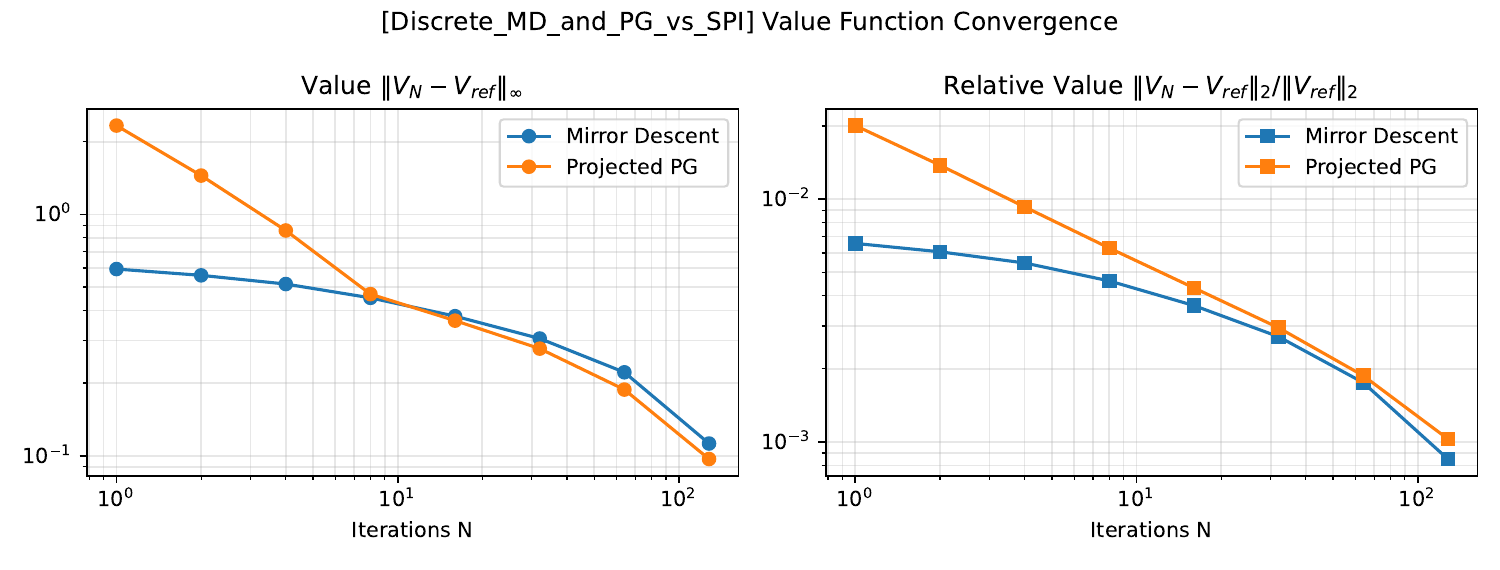}
\caption{Convergence of the evaluated value function $V_k$ towards the optimal value function $V^*$, demonstrating the convergence of the objective function in context of the example in Section~\ref{sec:example_finite_actions} with $\tau=10^{-6}$ and the projected PG method.}
\label{fig:convergence_value}
\end{figure}

\begin{figure}[htpb]
\centering
\includegraphics[width=0.9\textwidth]{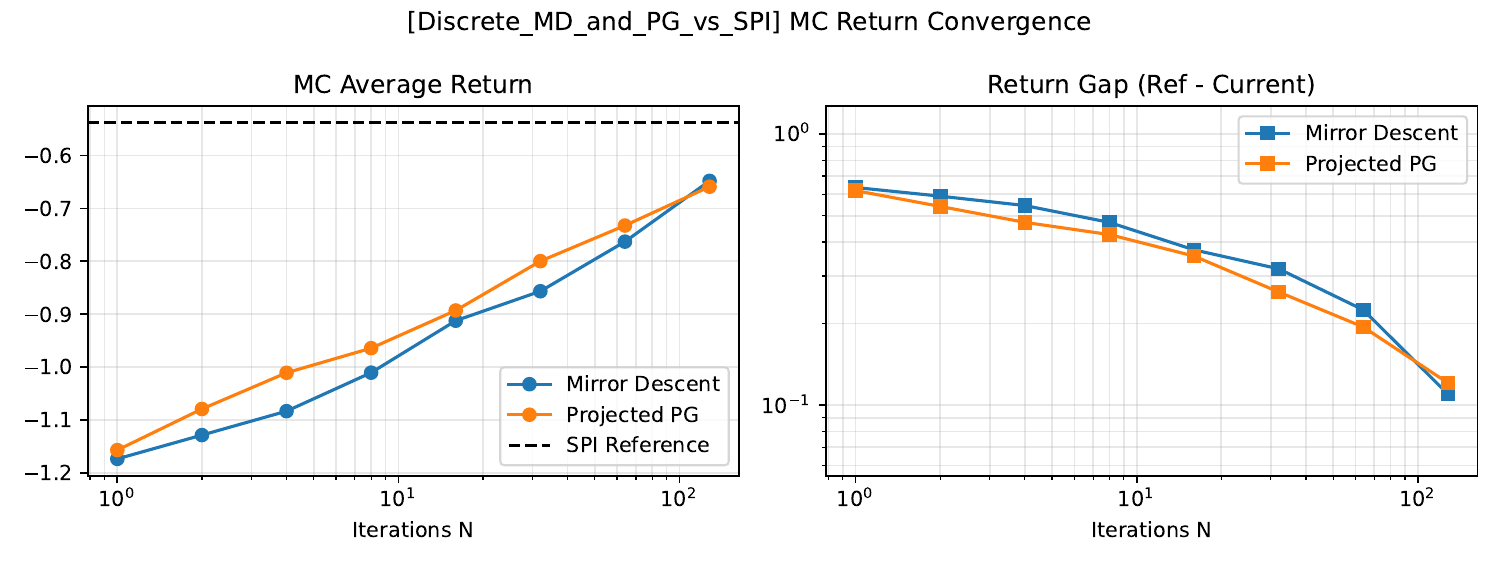}
\caption{Monte Carlo return gap evaluated in the true SDE environment in context of the example in Section~\ref{sec:example_finite_actions} and the projected PG method.}
\label{fig:convergence_return}
\end{figure}

In the setting of Section~\ref{sec:example_finite_actions},
Figures~\ref{fig:convergence_value} and~\ref{fig:convergence_return} plot these metrics across iterations $k$.
Here $\tau=10^{-6}$, so that the regularised problem is a close approximation of the unregularised one; with $\lambda=2\tau$ and $\eta=1$ (so that $S=k$) we have $\tfrac{\lambda}{2}S\le 1.3\times10^{-4}$ over the iterations shown.
Since $\tfrac{\lambda}{2}\big(e^{\frac{\lambda}{2}S}-1\big)^{-1}=S^{-1}\big(1+O(\lambda S)\big)$, the bound~\eqref{eqn:exp_con} reduces to~\eqref{eqn:lin_conv} in this regime, so these figures illustrate convergence but do not test the exponential rate.
To test the exponential regime we repeated the experiment with $\tau\in\{0.02,0.05,0.1\}$, all other parameters unchanged, computing the reference $V^*$ by SPI with the same $\tau$.
Figure~\ref{fig:convergence_semilog} shows $\|V_k-V^*\|_\infty$ on a semi-logarithmic scale.
The plot shows faster convergence than predicted for few initial steps but later settles at the predicted exponential rate.
Over the final iterations of each run the measured rates are $0.104$, $0.050$ and $0.020$ per iteration for $\tau=0.1$, $0.05$ and $0.02$, in agreement with the rate $\tfrac{\lambda}{2}=\tau$ in~\eqref{eqn:exp_con}.

\begin{figure}[htpb]
\centering
\includegraphics[width=0.9\textwidth]{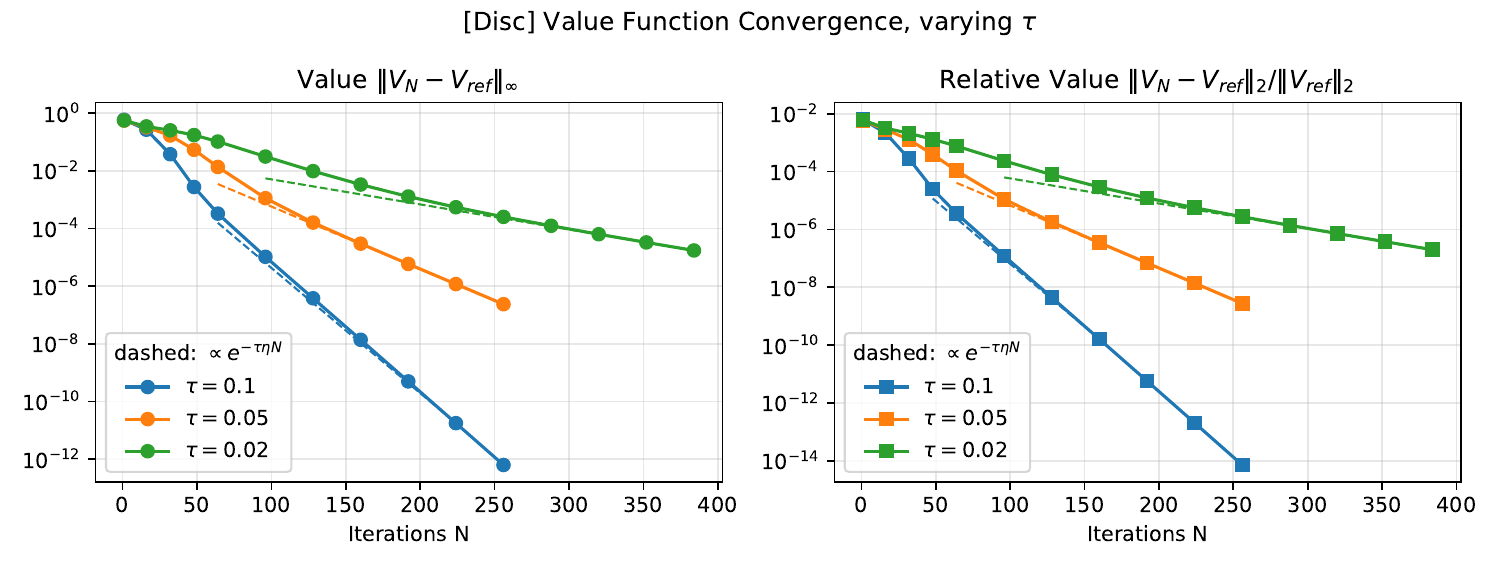}
\caption{Value function error of Algorithm~\ref{alg:dynamic_md} in the setting of Section~\ref{sec:example_finite_actions} with $\eta=1$ and $\tau\in\{0.02,0.05,0.1\}$, in the sup norm (left) and relative $L^2$ norm (right), on a semi-logarithmic scale; other parameters as in Table~\ref{tab:dpp_discrete_hyperparameters}. Dashed lines have slope $-\tau$, the rate $e^{-\frac{\lambda}{2}S}$ with $\lambda=2\tau$ and $S=k$ predicted by Theorem~\ref{thm:lin_conv}.}
\label{fig:convergence_semilog}
\end{figure}

Use \href{https://github.com/deterministicdavid/md\_constrained\_stoch\_control}{\texttt{https://github.com/deterministicdavid/md\_constrained\_stoch\_control}} to find the code to reproduce the presented results.

\begin{figure}
\centering
\includegraphics[width=0.9\textwidth]{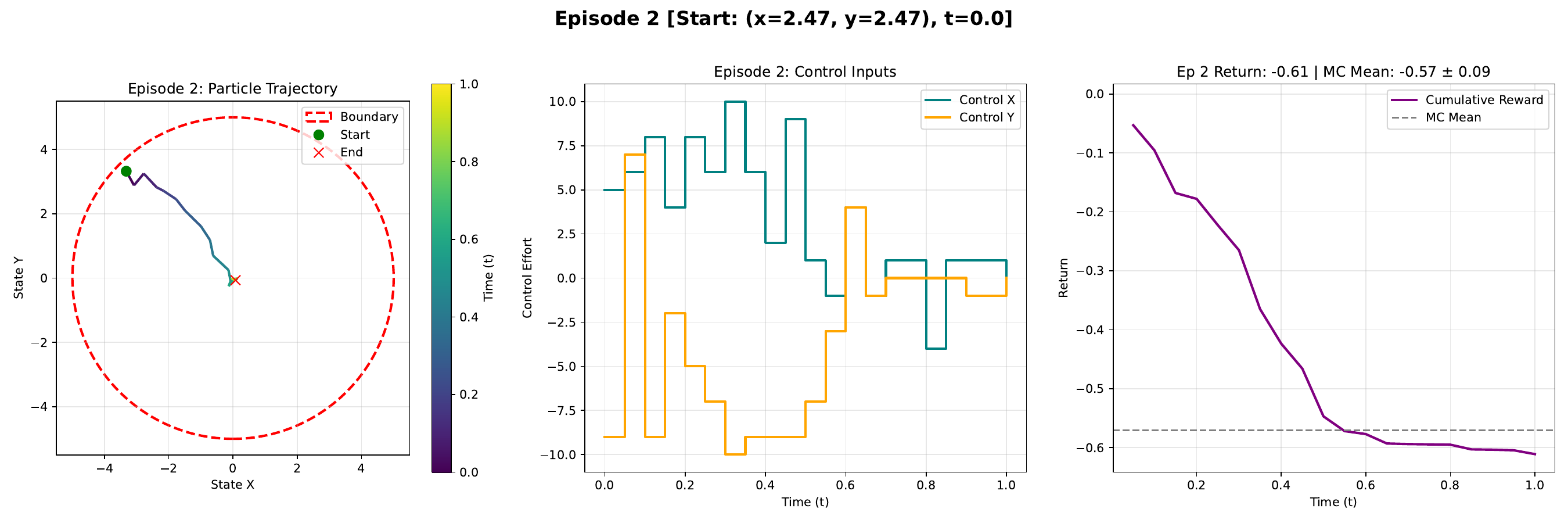}
\caption{Example trajectories with control found with the model-free, RL-based tabular mirror descent actor critic Algorithm~\ref{alg:rl_md} for the problem of Section~\ref{sec:example_finite_actions}.}
\label{fig validation rl}
\end{figure}

\subsection{Reinforcement learning implementation}
\label{sec:rl_implementation}
While the algorithms presented in the preceding sections rely on solving the Bellman PDE using finite differences, the mirror descent framework naturally extends to model-free Reinforcement Learning (RL), establishing a theoretical foundation for continuous-time actor-critic methods.
We will focus on the setting of soft policies over finite action spaces (Example in Section~\ref{sec:example_finite_actions}).

Let us briefly mention that classical Q-learning, policy gradient and mirror descent methods developed for discrete time Markov Decision process all rely on the Q function which does not have a continuous time counterpart~\cite{tallec2019making}.
Instead one can observe that the Q function in MDPs is the uncentred first variation of the value function with respect to the (soft) policy. 
The continuous time equivalent of this is the diffusion generator plus the cost and in the light of this what follows will perhaps feel natural.

In the setting of soft policies over finite action spaces (Example in Section~\ref{sec:example_finite_actions})
evaluating the dual update step in the mirror flow \eqref{eqn:grad_flow} requires computing the gradient of the unoptimized Hamiltonian with respect to the action probabilities (recall that action probabilities are denoted by $a$, the discrete actions are simply indexed $\{i=1,\ldots,p\}$).
From the definition of the Hamiltonian, the $i$-th component of this gradient is:
\begin{align}
\label{eqn:rl_ham_grad}
[\nabla_a H(\cdot, \nabla V, a)]_i = \beta(\cdot, i) \cdot \nabla V + \varphi(\cdot, i) + \tau \Big(1 + \log \tfrac{a_i}{a^{(0)}_i}\Big).
\end{align}

To evaluate this model-free, we must connect it to quantities that can be learned from rollout data.
In continuous-time RL, the analogue to the state-action value function is the little $q$ function (or advantage rate)~\cite{jia2023q}, defined via the diffusion generator plus the costs (but excluding the regularization term as we treat how we regularize as part of algorithm design and thus known):
\begin{equation*}
q(t, x, i) := \partial_t V(t,x) + \tfrac{1}{2}\tr\big(\sigma(t,x)\sigma(t,x)^\top \nabla^2 V(t,x)\big) + \beta(t, x, i) \cdot \nabla V(t, x) + \varphi(t, x, i).
\end{equation*}
By applying It\^o's formula, the advantage rate function can be directly connected to the continuous-time martingale temporal difference (TD) loss, allowing both $V$ and $q$ to be learned simultaneously from trajectory data $(t, X_t, A_t, C_t, X_{t+dt})$ sampled at discrete steps of length $dt>0$, without knowing $\beta$, $\varphi$, or $\sigma$, see~\cite[Theorem 6]{jia2023q}.

Comparing $q(t,x,i)$ to \eqref{eqn:rl_ham_grad}, we see that:
\begin{equation*}
[\nabla_a H(t, x, \nabla V, a)]_i = q(t, x, i) - B(t,x) + \tau \Big(1 + \log \tfrac{a_i}{a^{(0)}_i}\Big),
\end{equation*}
where
$B(t,x) := \partial_t V(t,x) + \tfrac{1}{2}\tr\big(\sigma(t,x)\sigma(t,x)^\top \nabla^2 V(t,x)\big)$.

Because the control only influences the drift, the terms $\partial_t V$ and the trace of the Hessian form a state-dependent baseline $c(t,x)$ that is strictly independent of the action $i$. Under our chosen mirror map (the softmax function), shifting all components of the dual variable $Z(t,x)$ by a uniform scalar does not change the resulting probabilities in the primal space, as $\text{softmax}(Z - c\mathbf{1}) = \text{softmax}(Z)$. 

Consequently, we can safely ignore the $i$-independent baseline $B(t,x)$ in the dual update. Assuming a uniform reference policy $a^{(0)}$ (whose constants also disappear), the model-free mirror descent update rule for the $i$-th component of the dual variable $Z(t,x) \in \mathbb{R}^p$ simplifies to:
\begin{align}
Z_{k+1}^{(i)}(t, x) = Z_k^{(i)}(t, x) - \eta_k \Big( q_k(t, x, i) + \tau \log a_k^{(i)}(t, x) \Big).
\end{align}

We detail a fully tabular, off-policy actor-critic implementation of this continuous-time mirror descent algorithm in Algorithm \ref{alg:rl_md}. 

\begin{algorithm}[htpb]
\caption{Model-free tabular mirror descent actor-critic}
\label{alg:rl_md}
\begin{algorithmic}[1]
\Require Learning rates $\alpha_V, \alpha_q, \eta$, entropy coefficient $\tau$, time step $dt$, discount $\gamma$.
\Require Initialize tabular value $V(t, x)$, continuous $q$-factor $q(t, x,i)$, and dual logits $Z(t, x,i)$.
\State Initialize policy $a(t, x) \gets \text{softmax}(Z(t, x))$
\For{each environment step}
    \State Observe time $t$ and state $x$
    \State Sample action $i \sim a(t, x)$
    \State Execute action $i$, observe cost $c$ and next state $x'$ at time $t' = t + dt$
    \State Store transition $(t, x, i, c, t', x')$ in Replay Buffer
    \For{gradient step $= 1, \dots, N_{\text{grad}}$}
        \State Sample batch of transitions from Replay Buffer
        \State \textbf{Critic Evaluation:}
        \State \quad Compute Martingale TD error: $\delta \gets V(t,x) - (c - q(t, x, i))dt - \gamma V(t', x')$
        \State \quad Update Value table: $V(t, x) \gets V(t, x) - \alpha_V \delta$
        \State \quad Update Advantage table: $q(t, x, i) \gets q(t, x, i) - \alpha_q \delta \cdot dt$
        \State \textbf{Actor Dual Update (Mirror Descent):}
        \State \quad For all available actions $j \in \{1,\dots,p\}$ at state $(t, x)$:
        \State \quad \quad Compute variation: $\Delta_j \gets q(t, x, j) + \tau \log a_j(t, x)$
        \State \quad \quad Update dual variables: $Z_j(t, x) \gets Z_j(t, x) - \eta \Delta_j$
        \State \textbf{Actor Primal Projection:}
        \State \quad Update policy probabilities: $a(t, x) \gets \text{softmax}(Z(t, x))$
    \EndFor
\EndFor
\end{algorithmic}
\end{algorithm}

To empricaly validate that Algorithm~\ref{alg:rl_md} works we plot the learned value function and policy vector field in Figure~\ref{fig value and field rl} and an evaluation plot for the learned policy in Figure~\ref{fig validation rl}.
The parameters used to run the experiment and to produce Figures~\ref{fig value and field rl} and~\ref{fig validation rl} are in Table~\ref{tab:rl_hyperparameters} in Section~\ref{sec experiment parameters}.

Use \href{https://github.com/deterministicdavid/md\_constrained\_stoch\_control}{\texttt{https://github.com/deterministicdavid/md\_constrained\_stoch\_control}} to find the code to reproduce the presented results.

\begin{figure}
\centering
\includegraphics[width=0.49\textwidth]{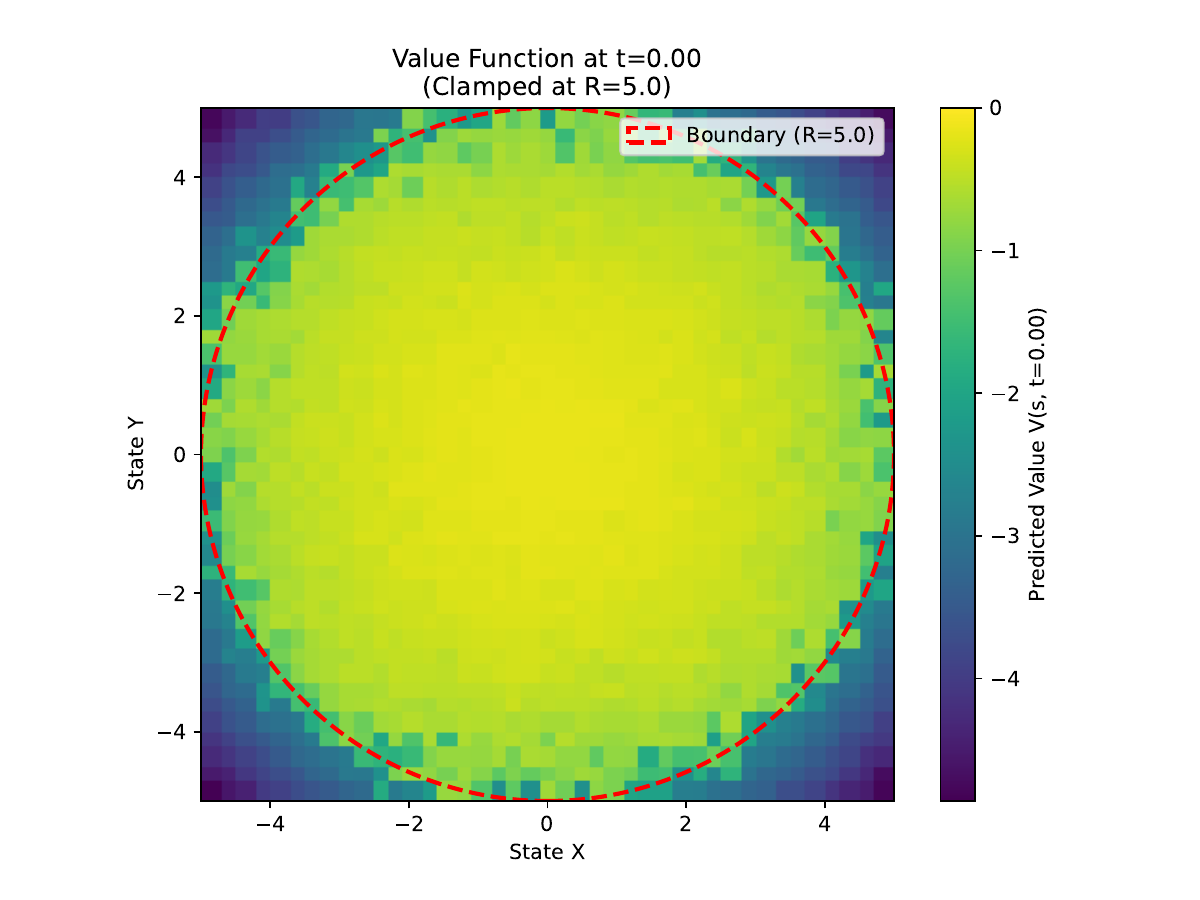}
\includegraphics[width=0.38\textwidth]{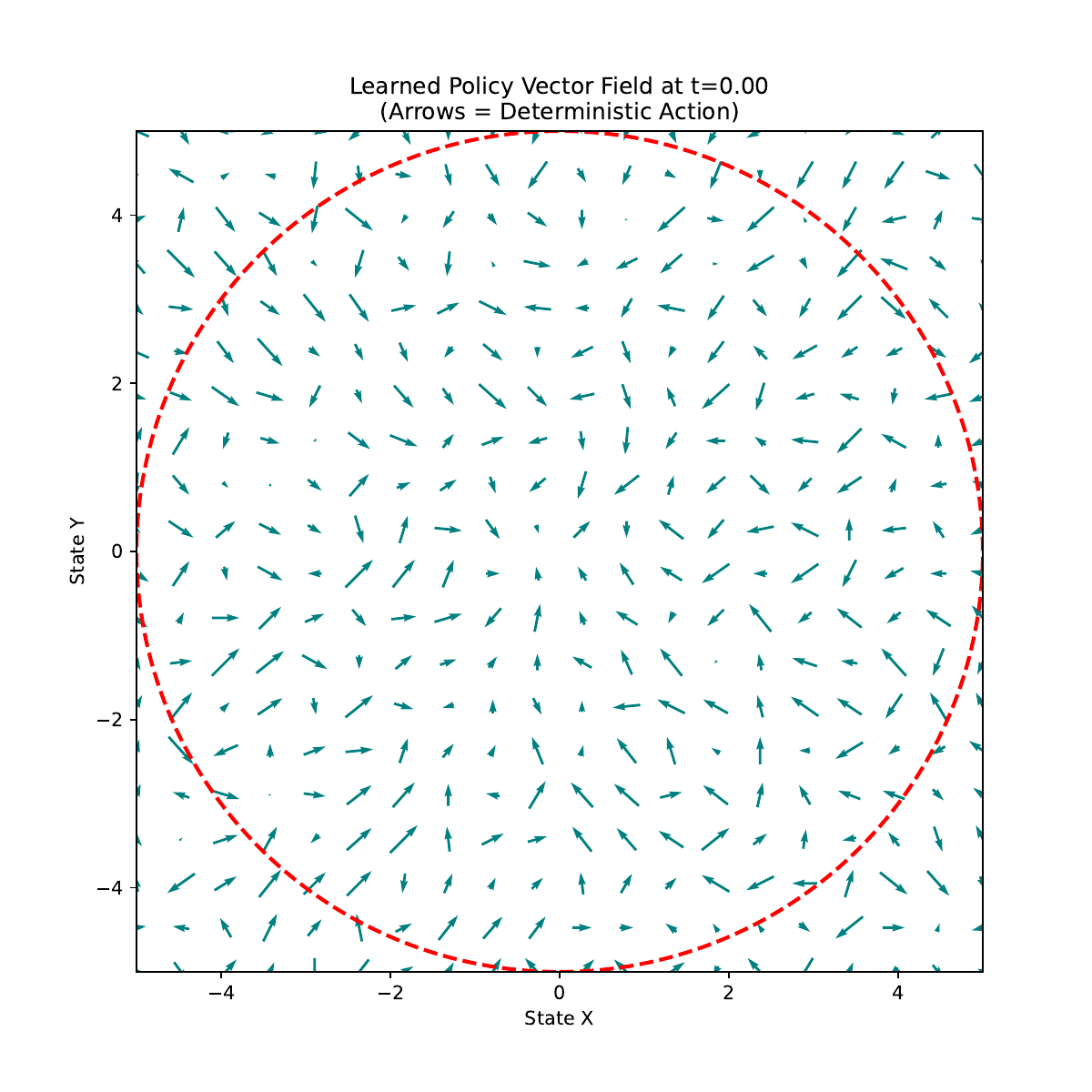}
\caption{Value function (left) and policy vector field (right) at time $t=0$ found with the model-free tabular mirror descent actor critic Algorithm~\ref{alg:rl_md} for the problem of Section~\ref{sec:example_finite_actions}.
The value function correctly indicates that away from the boundary the cost of pushing the state to the origin will be low. The learned actions are broadly correct, pointing towards the origin.}
\label{fig value and field rl}
\end{figure}

\appendix
\section{Proof of Theorems \ref{thm:SDE_well_posedness_moment_bound},
\ref{thm:pde_well_posedness}
and \ref{thm:HJB_well_posed}}
\label{sec:app_pde_pfs}
\begin{proof}
[Proof of Theorem \ref{thm:pde_well_posedness}]
Let $b^u:\cO_T\rightarrow\mathbb{\mathbb{R}^d}$ be defined by $b^u(t,x)=b(t,x,u(t,x))$.
Define 
$ \tilde{b}^u:[0,\infty)\times\mathbb{R}^d\rightarrow\mathbb{R}^d$ by 
\begin{align}
\label{eqn:exten_b}
\tilde{b}^u(t,x)
=\begin{cases}
b^u(t,x) & (t,x)\in\cO_T \\
0 & \text{otherwise}
\end{cases}.
\end{align}
From Assumption \ref{ass:SDE_well_posed} for each $u\in\mathcal{U}^\tau_\psi$
there exists a constant $C>0$
such that 
\begin{align*}
\|\tilde{b}^u\|_{B_b([0,\infty)\times\mathbb{R}^d;\mathbb{R}^d)}+\|\sigma\|_{B_b([0,\infty)\times\mathbb{R}^d;\mathbb{R}^{d\times d'})}\leq C.
\end{align*}
Since $\sigma$ is uniformly elliptic it follows from~\cite[Theorem 7.2.1]{stroock1997multidimensional} that
for each $t,x\in[0,T]\times\mathbb{R}^d$ there exists a unique in law $(X^{t,x,u}_{t'})_{t'\geq t}$ satisfying $X^{t,x,u}_{t'}=x+\int_{t}^{t'}\tilde{b}^u(r,X^{t,x,u}_r)dr+\int_t^{t'}\sigma(r,X^{t,x,u}_r)dW_r$.
Let $\tau^{t,x,u}_\cO$ be the first exit time of $X^{t,x,u}$ from the region $\cO$.
The uniform ellipticity implies $\tau^{t,x,u}_\cO<\infty$ a.s.
Therefore for each $(t,x)\in\cO_T$, $(X^{t,x,u}_{t'})_{t'\in[t,T\wedge\tau^{t,x,u}_\cO]}$ satisfies \eqref{eqn:controlled_SDE}.
\end{proof}

We first prove Corollary \ref{cor:max_princ_ext} which is a straightforward application of the maximum principle for strong solutions to parabolic PDEs presented in \cite[Appendix E]{fleming2012deterministic}.
\begin{corollary}
\label{cor:max_princ_ext}
Suppose Assumption \ref{ass:SDE_well_posed} holds.
Let $b\in B_b(\cO_T; \mathbb{R}^d)$
and $w\in W^{2,1}_q(\cO_T)$  satisfy $w(t,x)=0$ on $\partial\cO_T$ and 
\begin{align*}
\partial_t w(t,x)+\tfrac{1}{2}\text{Tr}(\sigma\sigma^\top \nabla^2w)(t,x)+b(t,x)\cdot w(t,x)\geq -C \enspace (\text{resp.} \leq C) \enspace \text{a.e. } (t,x)\in\cO_T.
\end{align*}
Then $w(t,x)\leq CT \enspace (\text{resp.}\geq -CT)$, for a.e. $(t,x)\in\cO_T$.
\end{corollary}
\begin{proof}
In order to apply the maximum principle for strong solutions \cite[Appendix E]{fleming2012deterministic} we introduce the function $\phi(t,x)=C(T-t)$.
Then
\begin{align*}
\begin{aligned}
\tfrac{\partial (w-\phi)}{\partial t}+\tfrac{1}{2}\text{Tr}(\sigma\sigma^\top \nabla^2(w-\phi))+ b\cdot\nabla (w-\phi)\geq 0 &\quad\text{ a.e. } \in\cO_T \\
(w-\phi)(\cdot,t)=C(t-T)\leq 0  &\quad\text{ on }\partial\cO_T
\end{aligned}.
\end{align*}
From the maximum principle, $(w-\phi)(t,x)\leq 0$ a.e. in $\cO_T$ which implies $w(t,x)\leq CT$ for a.e. $(t,x)\in\cO_T$.
For the reverse inequality 
let $\tilde{\phi}(t,x)=C(t-T)$, then $w-\tilde{\phi}$ satisfies 
\begin{align*}
\begin{aligned}
\tfrac{\partial (w-\tilde{\phi})}{\partial t}+\tfrac{1}{2}\text{Tr}(\sigma\sigma^\top \nabla^2(w-\tilde{\phi}))+ b\cdot\nabla (w-\tilde{\phi})\leq 0 &\quad\text{ a.e. in }\cO_T\\
(w-\tilde{\phi})(\cdot,t)=C(T-t)\geq 0 & \quad\text{ on }\partial\cO_T
\end{aligned}.
\end{align*}
Applying the maximum principle implies that for a.e. $(t,x)\in\cO_T$, $-(w-\tilde{\phi})(t,x)\leq 0$ which means $w(t,x)\geq -CT$.
\end{proof}
Following the notation introduced in ~\cite[Appendix $E$]{fleming2012deterministic}
we define the following norms of H{\"o}lder  type.
For $0<\mu\leq 1$ and
for a given 
$h:\cO_T\rightarrow\mathbb{R}$ define
\begin{align*}
\|h\|_{C^{\mu}(\cO_T)}&:=\|h\|_{C^0(\cO_T)}
+[h]_{\mu;\tfrac{\mu}{2}},\\
\|h\|_{C^{1+\mu}(\cO_T)}&:=\|h\|_{C^\mu(\cO_T)}+\sum_{i=1}^d\|\nabla_{x_i}h\|_{C^\mu(\cO_T)},
\end{align*}
where
$\|h\|_{C^0(\cO_T)}:=\sup_{(t,x)\in\cO_T}|h(t,x)|$ and 
\begin{align}
[h]_{\mu;\tfrac{\mu}{2}}:=\sup_{t\in[0,T],x,y\in\cO}\tfrac{|h(t,x)-h(t,y)|}{|x-y|^\mu}
+ \sup_{s,t\in[0,T],x\in\cO}\tfrac{|h(t,x)-h(s,x)}{|t-s|^{\tfrac{\mu}{2}}}.
\end{align}
Let $ C^{0,1}(\cO_T)$ represent the space of all function $ h:\cO_T\rightarrow\mathbb{R}$ which are continuous, have a continuous first derivative in space, equipped with the norm $\|h\|_{C^0(\cO_T)}+\sum_{i=1}^d\|\nabla_{x_i} h\|_{C^0(\cO_T)}$.
In order to prove Theorem \ref{thm:HJB_well_posed} recall the following embedding results:
from Arzel\`a-Ascoli's Theorem $C^{1+\mu}(\cO_T)$ is compactly embedded in $C^{0,1}(\cO_T)$ 
and 
as outlined on \cite[Appendix E, Equation E.9]{fleming2012deterministic}
if $q>d+2$ 
then there exists a constant $C$ depending on $\cO_T$ and $q$ such that for any $h\in W^{2,1}_q(\cO_T)$ 
\begin{align}
\label{eqn:embedding_Holder_Sob}
\|h\|_{C^{1+\mu}(\cO_T)}\leq C\|h\|_{W^{2,1}_q(\cO_T)} \text{ with }\mu=1-\tfrac{d+2}{q}.
\end{align}

\begin{proof}
[Proof of Theorem \ref{thm:pde_well_posedness}]
Let $f^u(t,x):= f(t,x,u(t,x))$.
Consider the following boundary value problem
\begin{equation}
\label{eqn:pf_bdry_prob}
\begin{aligned}
\tfrac{\partial v}{\partial t} + \mathcal{L}^{u}v = - f^u  -\tau \rho^u & \text{ a.e.  in $\cO_T$}\\
v=g & \text{ on $\partial\cO_T$}\,.
\end{aligned}
\end{equation}
From Assumption \ref{ass:SDE_well_posed}, the definition of the class $\mathcal{U}^\tau_\psi$
and Lemma \ref{lem:parabolic_PDE} we see that that there exists 
a unique solution $v\in W^{2,1}_q(\cO_T)$ 
to \eqref{eqn:pf_bdry_prob} such that
\begin{equation}
\label{eqn:helper_bound_1}
\begin{split}
\|v\|_{W^{2,1}_q(\cO_T)}\leq C(1+\tau\|\rho^u\|_{B_b(\cO_T;\mathbb{R})}+\|g\|_{W^{2}_q(\cO)} +\|v\|_{L_q(\cO_T)}),
\end{split}
\end{equation}
where constant $C>0$ depends on $d$, $q$, $\kappa$, $K$ and $T$.
In order to bound $\|v\|_{L_q(\cO_T)}$, we first identify $v$ with $V^u$. 

For each $t,x$ let $X=X^{t,x,u}$.
Applying the generalised Ito's formula  \cite[Ch.~2., Sec.~10, Theorem 1]{krylov2008controlled}  to $v$ and $X$ implies
\begin{align*}
g(X_{T_\cO})&= v(T_{\cO},X_{T_{\cO}})\\
&= v(t,x)+\int_{t}^{T_{\cO}}(\partial_t v^u+\mathcal{L}^uv)(t',X_{t'})dt'+\int_t^{T_{\cO}}\nabla v(t',X_{t'})\sigma(t',X_{t'})dW_{t'} .
\end{align*}
From Assumption \ref{ass:SDE_well_posed} 
and \eqref{eqn:embedding_Holder_Sob}
the stochastic integral is a martingale. 
Therefore
\begin{align*}
v(t,x)
=\E^{u}_{t,x}
\left[\int_t^{T_{\cO}}f(t',X_{t'},u(t',X_{t'}))+\tau\rho^u(t',X_{t'})dt'+g(X_{T_{\cO}})\right]=V^u(t,x).
\end{align*}
Returning to  \eqref{eqn:helper_bound_1} 
it remains to bound the $\|V^u\|_{L_q(\cO_T)}$.
From Assumption \ref{ass:SDE_well_posed}
\begin{align*}
|V^u(t,x)|
&\leq
\E^{u}_{t,x} \left[\int_t^{T_{\cO}}K(1+\tau|\rho^u(t',X_{t'})|)dt'+K\right]\leq C_{K,A,\cO}(1+\tau\|\rho^u\|_{B_b(\cO_T;\mathbb{R})}),
\end{align*}
which concludes the proof.
\end{proof}
The next lemma establishes an a priori estimate necessary to apply
Leray-Schauder's Fixed Point Theorem ~\cite[Theorem 10.1.1]{wu2006elliptic} to prove
Theorem \ref{thm:HJB_well_posed}.
\begin{lemma}
\label{lem:LS-FP-Helper}
Suppose Assumptions \ref{ass:SDE_well_posed} and \ref{ass:ham_structure} hold. 
There exists a $C\geq 0$, 
such that for any $v\in W^{2,1}_q(\cO_T)$ and any $\eta\in[0,1]$ which satisfy
\begin{equation}
\begin{aligned}
\tfrac{\partial v}{\partial t}+\tfrac{1}{2}\text{Tr}(\sigma\sigma^\top \nabla^2v)+\eta\mathcal{H}(\cdot,\nabla v) &= 0, \text{ a.e. in $\cO_T$}\\
v&=\eta g \text{ on } \partial\cO_T\,,
\end{aligned}
\end{equation}
we have $\|v\|_{W^{2,1}_q(\cO_T)}\leq C(1+\|g\|_{C^2(\cO)})$. 
\end{lemma}
\begin{proof}
Let $w=v-\eta g$, then $w\in W^{2,1}_q(\cO_T)$ and equals zero on the parabolic boundary in the classical sense (since $q>d+2$).
Therefore
\begin{align*}
\tfrac{\partial (w+\eta g)}{\partial t}+\tfrac{1}{2}\text{Tr}(\sigma\sigma^\top \nabla ^2(w+\eta g))+\eta\mathcal{H}(\cdot,\nabla (w+\eta g)) = 0,  
\end{align*}
a.e in $O_T$.
Since $g$ depends only on $x$ the above is equivalent to 
\begin{align}
\label{eqn:PDE_aprior_bound}
\tfrac{\partial w}{\partial t}+\tfrac{1}{2}\text{Tr}(\sigma\sigma^\top \nabla^2w)
=
-\tfrac{\eta}{2}\text{Tr}(\sigma\sigma^\top \nabla ^2g)
-\eta\mathcal{H}(\cdot,\nabla (w+\eta g)),
\end{align}
a.e in $O_T$.
From Assumption \ref{ass:SDE_well_posed}  
$\text{Tr}(\sigma\sigma^\top \nabla^2g)\in L_q(\cO_T)$ and from Assumption \ref{ass:ham_structure} the map $(t,x)\mapsto\mathcal{H}(t,x,(\nabla w+\eta \nabla g)(t,x))$ is measurable.
The structural condition in \ref{ass:ham_structure} implies that the right hand side of \eqref{eqn:PDE_aprior_bound} is in $L_q(\cO_T)$.
From Lemma \ref{lem:parabolic_PDE} there exists a constant $C>0$ such that for any $v$
\begin{equation}
\label{eqn:first_bound}
\begin{split}
\|w\|_{W^{2,1}_q(\cO_T)}&\leq 
C(\|
-\tfrac{\eta}{2}\text{Tr}(\sigma\sigma^\top \nabla^2g) - \eta \mathcal{H}(\cdot,\nabla(w+\eta g))\|_{L_q(\cO_T)}+\|w\|_{L_q(\cO_T)})\\
&\leq 
C(\left\|\nabla^2g\|_{L_q(\cO)} + \eta\| \mathcal{H}(\cdot,\nabla(w+\eta g))\right\|_{L_q(\cO_T)}+\|w\|_{L_q(\cO_T)})\\
&\leq C(1+\|g\|_{C^2(\cO)}+\|\nabla w\|_{L_q(\cO_T)}+\|w\|_{L_q(\cO_T)})
\end{split}
\end{equation}
In order to bound $\|\nabla w\|_{L_q(\cO_T)}$ we will use an interpolation inequality. 
Let $\varepsilon>0$.
There exists a $C_\varepsilon>0$ such that for any $w\in W^{2,1}_q(\cO_T)$
\begin{align*}
\|\nabla w\|_{L_q(\cO_T)}&\leq\left(\int_0^T\|w(t,\cdot)\|^q_{W^{1}_q(\cO)}dt\right)^{\tfrac{1}{q}}\\
&\leq \left(\int_0^T(\varepsilon\|w(t,\cdot)\|_{W^2_q(\cO)}+C_\varepsilon\|w(t,\cdot)\|_{L_q(\cO)})^q dt\right)^{\tfrac{1}{q}}\\
&\leq C_q \left(
\int_0^T \varepsilon^q\|w(t,\cdot)\|^q _{W^2_q(\cO)}dt\right)^{\tfrac{1}{q}}
+
C_q \left(
C_\varepsilon^q \int_0^T\|w(t,\cdot)\|^q_{L_q(\cO)} dt\right)^{\tfrac{1}{q}}\\
&\leq 
\varepsilon C_q\|w\|_{W^{2,1}_q(\cO_T)}
+C_q C_\varepsilon\|w\|_{L_q(\cO_T)},
\end{align*}
where the second inequality follows from \cite[Theorem 7.28]{gilbarg1977elliptic}.
Substituting this back into~\eqref{eqn:first_bound} and choosing $\varepsilon > 0$ sufficiently small such that $\varepsilon C_q \leq \frac{1}{2}$, we can absorb the $\|w\|_{W^{2,1}_q(\cO_T)}$ term into the left-hand side.
This yields a new generic constant $C$ (independent of $w$) such that
$\|w\|_{W^{2,1}_q(\cO_T)}\leq C(1+\|g\|_{C^2(\cO_T)}+\|w\|_{L_q(\cO_T)})$.
In order to bound the term $\|w\|_{L_q(\cO_T)}$ we will use Corollary \ref{cor:max_princ_ext}. 
From Assumption \ref{ass:ham_structure} for each $w$ there exists  $u^*$ such that 
\begin{align*}
\mathcal{H}(\cdot,\nabla(w+\eta g))&=b^{u^*}\cdot(\nabla w+\eta\nabla g)+f^{u^*}+\tau \rho^{u^*},
\end{align*}
and using the non-negativity of the Bregman divergence and the definition of the Hamiltonian we 
have the following upper and lower bounds
\begin{align}
\mathcal{H}(\cdot,\nabla(w+\eta g))&\geq b^{u^*}\cdot\nabla(w+\eta g)+f^{u^*},\label{eqn:lower}\\
\mathcal{H}(\cdot,\nabla(w+\eta g))&\leq b^{u^{(0)}}\cdot\nabla(w+\eta g)+f^{u^{(0)}}\label{eqn:upper}.
\end{align}
We will first bound $w$ from above. 
Substituting \eqref{eqn:upper} into \eqref{eqn:PDE_aprior_bound} implies
\begin{equation}
\begin{aligned}
\tfrac{\partial w}{\partial t}+\tfrac{1}{2}\text{Tr}(\sigma\sigma^\top \nabla^2w)+\eta b^{u^{(0)}}\cdot\nabla w &\geq - F^{u^{(0)}} \text{ a.e in $\cO_T$}\\
w&=0  \text{ on }\partial\cO_T
\end{aligned}.
\end{equation}
where $F^a(t,x)=\tfrac{\eta}{2}\text{Tr}(\sigma\sigma^\top \nabla^2g)(t,x)+ \eta f(t,x,a)
+ \eta^2 \nabla g(x)\cdot b(t,x,a)$.
From Assumption \ref{ass:SDE_well_posed}
there exists a constant $C>0$ such that for any $a\in A$ 
\begin{align}
\label{eqn:helphelp}
\|F^{a}\|_{B_b(\cO_T;\mathbb{R})}\leq \tfrac{1}{2}\|\nabla^2g\|_{C^0(\cO)}+K(1+\|\nabla g\|_{C^0(\cO)})=:C
\end{align}
Therefore 
\begin{equation}
\begin{aligned}
\tfrac{\partial w}{\partial t}+\tfrac{1}{2}\text{Tr}(\sigma\sigma^\top \nabla^2w)+\eta b^{u^{(0)}}\cdot\nabla w(t,x) &\geq - C  \text{ a.e. in }\cO_T \\
w&=0  \text{ on }\partial\cO_T
\end{aligned}.
\end{equation}
From Corollary \ref{cor:max_princ_ext},  $w(t,x)\leq CT$ a.e. in $\cO_T$.
For the lower bound substituting \eqref{eqn:lower} into \eqref{eqn:PDE_aprior_bound} implies 
\begin{align*}
&\tfrac{\partial w}{\partial t}+\tfrac{1}{2}\text{Tr}(\sigma\sigma^\top \nabla^2w)+\eta b^{u^*}\cdot\nabla w\leq
-\tfrac{\eta}{2}\text{Tr}(\sigma\sigma^\top \nabla ^2g)-\eta f^{u^*}-\eta^2 b^{u^*}\cdot \nabla g=-F^{u^*}.
\end{align*}
Therefore 
\begin{equation}
\begin{aligned}
\tfrac{\partial w}{\partial t}+\tfrac{1}{2}\text{Tr}(\sigma\sigma^\top \nabla^2w)+\eta b^{u^*}\cdot\nabla w&\leq C  \text{ a.e. in }\cO_T\\
w&=0  \text{ on }\partial\cO_T
\end{aligned},
\end{equation}
where the constant $C$ is as in \eqref{eqn:helphelp}.
Corollary \ref{cor:max_princ_ext} concludes the proof.
\end{proof}

\begin{proof}
[Proof of Theorem \ref{thm:HJB_well_posed}]
Define the mapping $T:C^{0,1}(\cO_T)\rightarrow C^{0,1}(\cO_T)$ by $v=Tu$, where $v$ is the solution to 
\begin{align}
\label{eqn:fixed_point_pde}
\tfrac{\partial v}{\partial t}+\tfrac{1}{2}\text{Tr}(\sigma\sigma^\top \nabla^2 v)+\mathcal{H}(\cdot,\nabla u) = 0 \text{ a.e. in }\cO_T
\end{align}
and $v=g$  on $\partial\cO_T$.
The map $T$ is well-defined.
To see this let $u\in C^{0,1}(\cO_T)$.
Then from Assumption \ref{ass:ham_structure} $(t,x)\mapsto\mathcal{H}(t,x,\nabla u(t,x))$ is measurable and 
$|\mathcal{H}(t,x,\nabla u(t,x))|\leq K(1+\|u\|_{C^{0,1}(\cO_T)})$.
Therefore $\mathcal{H}(\cdot,\nabla  u(\cdot))\in L_q(\cO_T)$. 
Lemma \ref{lem:parabolic_PDE} and \eqref{eqn:embedding_Holder_Sob} implies $v\in W^{2,1}_q(\cO_T)\subset C^{0,1}(\cO_T)$.

To show the map $T$ is continuous let
$\{u_n\}_{n\in\mathbb{N}}$ converge to $u$ in $C^{0,1}(\cO_T)$, $v_n=Tu_n$ and $v=Tu$. 
Then $v-v_n\in W^{2,1}_q(\cO_T)$ satisfies
\begin{align*}
\tfrac{\partial (v-v_n)}{\partial t}(t,x)+\tfrac{1}{2}\text{Tr}(\sigma\sigma^\top \nabla^2(v-v_n))+\mathcal{H}(t,x,\nabla u(t,x))-\mathcal{H}(t,x, \nabla u_n(t,x)) = 0,
\end{align*}
for a.e $(t,x)\in\cO_T$,
and $v-v_n=0$ on $\partial\cO_T$.
From 
\cite[Appendix E, p. 207]{fleming2012deterministic}  and \eqref{eqn:embedding_Holder_Sob} there exists a constant $C>0$ independent of $n$ such that 
\begin{equation}
\label{eqn:temp_1}
\|v-v_n\|_{C^{0,1}(\cO_T)}
\leq\|v-v_n\|_{W^{2,1}_q(\cO_T)}\leq C\|\mathcal{H}(\cdot,\nabla u(\cdot))-\mathcal{H}(\cdot,\nabla u_n(\cdot))\|_{L_q(\cO_T)}.
\end{equation}
In order to bound \eqref{eqn:temp_1} note that due to Assumption \ref{ass:ham_structure} for any $(t,x)\in\cO_T$ and $z\in\mathbb{R}^d$ there exists $a^*_z\in A$ such that 
$\mathcal{H}(t,x,z)=H(t,x,z,a^*_z)$. Similarly for any $(t,x)\in\cO_T$ and $z'\in\mathbb{R}^d$ $\mathcal{H}(t,x,z')\leq H(t,x,z',a)$ for all $a$. Taking $a=a^*_z$ yields the following
\begin{align}
\label{eqn:lip_ham}
\mathcal{H}(t,x,z)-\mathcal{H}(t,x,z')\geq H(t,x,z,a^*_z)- H(t,x,z',a^*_z)\geq -K|z-z'|.
\end{align}
An upper bound is attained by interchanging $z$ and $z'$.
Combining both inequalities it follows that
there exists a constant $C>0$ such that for any $(t,x)\in\cO_T$ and $z,z'\in\mathbb{R}^d$,
$|\mathcal{H}(t,x,z)-\mathcal{H}(t,x,z)|\leq C|z-z'|$.
Returning to \eqref{eqn:temp_1} there exists a constant $C>0$ independent of $n\in\mathbb{N}$ such that 
\begin{align*}
\|v-v_n\|_{C^{0,1}(\cO_T)}
\leq C\| \nabla u-\nabla u_n\|_{L_q(\cO_T)}
\leq C\|\nabla (u- u_n)\|_{C^{0}(\cO_T)}
\leq C\|u-u_n\|_{C^{0,1}(\cO_T)}.
\end{align*}
This concludes the proof of the continuity of the map $T$.

It remains to prove $T$ is a compact, that is we want to show $T$ maps any bounded set in
$C^{0,1}(\cO_T)$ to a precompact set.
Assume $\{u_n\}_{n\in\mathbb{N}}\subset C^{0,1}(\cO_T)$ is bounded, so that $M := \sup_n \|u_n\|_{C^{(0,1)}(\mathcal O_T)}<\infty$, and let $v_n=Tu_n$. 
Recall that
for $0<\mu\leq 1$ the set
$C^{1+\mu}(\cO_T)$ is precompact in $C^{0,1}(\cO_T)$.
Let $\mu=1-\frac{d+2}{q}$.
Then there exists a constant $C>0$ such that for any $n\in\mathbb{N}$
\begin{align*}
\|v_n\|_{C^{1+\mu}(\cO_T)}&\leq\|v_n\|_{W^{2,1}_q(\cO_T)}\leq C(\|\mathcal{H}(\cdot,\nabla u_n)\|_{L_q(\cO_T)}+\|v_n\|_{L_q(\cO_T)})\\
&\leq C(1+\|u_n\|_{C^{0,1}(\cO_T)})\leq C(1+M),
\end{align*}
where the first inequality follows from \eqref{eqn:embedding_Holder_Sob}, the second from Lemma \ref{lem:parabolic_PDE}, the third from Corollary \ref{cor:max_princ_ext} and the fact $v^n$ satisfies \eqref{eqn:fixed_point_pde} (the details for this are omitted as the argument is similar to that in Lemma \ref{lem:LS-FP-Helper}).
Therefore $T$ maps bounded sets in $C^{0,1}(\cO_T)$ to bounded sets in $C^{1+\mu}(\cO_T)$ which is compactly embedded in $C^{0,1}(\cO_T)$.

To conclude the proof let $\eta\in[0,1]$ and $v\in C^{0,1}(\cO_T)$ be such that $v=\eta Tv$, i.e.
\begin{align}
\label{eqn:PDE_HJB_no_val_funct}
\tfrac{\partial v}{\partial t}(t,x)+\tfrac{1}{2}\text{Tr}(\sigma\sigma^\top \nabla^2v)(t,x)+\eta\mathcal{H}(t,x,\nabla v(t,x))=0,\text{ a.e. }(t,x)\in\cO_T,
\end{align}
and $v(t,x)=\eta g(x)$. 
From Lemma \ref{lem:LS-FP-Helper} there exists a constant $C>0$, independent of $\eta$ and $v$ such that $\|v\|_{W^{2,1}_q(\cO_T)}\leq C$.
It follows from 
Leray-Schauder Theorem \cite[Theorem 11.3]{gilbarg1977elliptic} there exists a $v\in C^{0,1}(\cO_T)$ such that $v=Tv$, and from \eqref{eqn:embedding_Holder_Sob} this solution is in $W^{2,1}_q(\cO_T)$.

To prove uniqueness let $v_1,v_2\in W^{2,1}_q(\cO_T)$ satisfy \eqref{eqn:hjb}. 
Let $w=v_1-v_2\in W^{2,1}_q(\cO_T)$ and note it satisfies 
\begin{align}
\label{eqn:uniqueness_1}
\tfrac{\partial w}{\partial t}(t,x)+\tfrac{1}{2}\text{Tr}(\sigma\sigma^\top \nabla^2w)(t,x)+\mathcal{H}(t,x,\nabla v_1(t,x))-\mathcal{H}(t,x,\nabla v_2(t,x))=0 ,
\end{align}
a.e. $(t,x)\in\cO_T$ and $w=0$ on $\partial\cO_T$. 
Let $a^*:\cO_T\rightarrow A$ be the map for which 
\begin{align*}
\mathcal{H}(t,x,\nabla v_2(t,x))=H(t,x,\nabla v_2(t,x),a^*(t,x)).
\end{align*}
Then 
\begin{align*}
\mathcal{H}&(t,x,\nabla v_1(t,x))- \mathcal{H}(t,x,\nabla v_2(t,x)) =\mathcal{H}(t,x,\nabla v_1(t,x))-H(t,x,\nabla v_2(t,x),a^*(t,x))\\
&\leq H(t,x,\nabla v_1(t,x),a^*(t,x))-H(t,x,\nabla v_2(t,x),a^*(t,x))= b(t,x,a^*(t,x))\cdot \nabla w(t,x)
\end{align*}
Substituting this into \eqref{eqn:uniqueness_1} implies 
\begin{align*}
0&=\tfrac{\partial w}{\partial t}(t,x)+\tfrac{1}{2}\text{Tr}(\sigma\sigma^\top \nabla^2w)(t,x)+\mathcal{H}(t,x,\nabla v_1(t,x))-\mathcal{H}(t,x,\nabla v_2(t,x))\\
&\leq \tfrac{\partial w}{\partial t}(t,x)+\tfrac{1}{2}\text{Tr}(\sigma\sigma^\top \nabla^2w)(t,x)+b(t,x,a^*(t,x))\cdot \nabla w(t,x)
\end{align*}
a.e. $(t,x)\in\cO_T$ and $w=0$ on $\partial\cO_T$.
Applying the Maximum Principle \cite[Appendix E]{fleming2012deterministic} gives $w(t,x)\leq 0$ a.e.
Interchanging the roles of $v_1$ and $v_2$ gives the reverse inequality.
Therefore there exists one and only solution.
It remains to identify the unique solution in $W^{2,1}_q(\cO_T)$ of \eqref{eqn:hjb} with the optimal value function and that $u^*(t,x)=a^*(t,x,\nabla V^*(t,x))$ is an optimal control. 
This follows from the usual verification argument which utilizes the
generalized Ito's formula \cite[Ch.~2., Sec.~10, Theorem 1]{krylov2008controlled}.

Finally to show this control is indeed optimal from the definition of the Hamiltonian
\begin{align*}
\mathcal{H}(t,x,\nabla V^*(t,x))\leq H(t,x,\nabla V^*(t,x),u^{(0)}(t,x)).
\end{align*}
Therefore from the non-negativity of the Bregman divergence 
\begin{align*}
0&\leq\tau\rho^{u^*}(t,x)\\
&\leq H(t,x,\nabla V^*(t,x),u^{(0)}(t,x))- b(t,x,u^*(t,x))\cdot \nabla V^*(t,x)-f(t,x,u^*(t,x))\\
&\leq b(t,x,u^{(0)}(t,x))\cdot \nabla V^*(t,x)+f(t,x,u^{(0)}(t,x))+K(1+\|\nabla V^*\|_{C^0(\cO_T)})\\
&\leq C(1+\|\nabla V^*\|_{C^0(\cO_T)}),
\end{align*}
which concludes the proof.
\end{proof}

\section{Validating that examples satisfy assumptions}
\label{sec validation of assumptions for our examples}
\subsection{Validation of assumptions for Example \ref{sec:ex_LQR}}
\label{sec:ap_ex2}
Let $\psi^*:\mathbb{R}^2\rightarrow\mathbb{R}$ 
denote the Legendre conjugate of $\psi$. 
We will show 
\begin{align}
\label{eqn:leg_conj_lqr}
\psi^*(y)=-1+\sqrt{1+R^2|y|^2}+\log\tfrac{-1+\sqrt{1+R^2|y|^2}}{|y|^2}+\log 2.
\end{align}
To this end recall that 
$\psi^*(y)=\sup_{a\in B_R(0)}a\cdot y + \log(R^2-|a|^2)$. 
Since $-\log(R^2-|a|^2)$ converges to $+\infty$ near the boundary the first order optimality condition asserts that $y = \tfrac{2a^*(y)}{R^2-|a^*(y)|^2}$.
This suggests that for each $y$, $a^*(y)=\alpha\tfrac{ y}{|y|}$ for some $\alpha\in(0,R)$.
To find $\alpha$ we solve $|y|\alpha^2+2\alpha-|y|R^2=0$ which implies $\alpha=|y|^{-1}(-1+\sqrt{1+|y|^2R^2})$.
A direct computation shows that $\nabla\psi^*:\mathbb{R}^2\rightarrow A$ is given by \eqref{eqn:ex_leg_grad} and 
$[\nabla^2\psi^*(y)]_{ij}=\alpha(|y|)\delta_{ij}-\beta(|y|)y_iy_j$
where $\alpha(|y|)=R^2[1+\sqrt{1+R^2|y|^2}]^{-1}$ and $\beta(|y|)=R^4(\sqrt{1+R^2|y|^2})^{-1}(1+\sqrt{1+R^2|y|^2})^{-2}$.
The details for these calculations are omitted. 

To show Assumption \ref{ass:mirror_reg} holds
we need to show the operator norm of the Hessian is bounded. To that end 
let $y,y'\in\mathbb{R}^2$ then
\begin{align*}
|\nabla^2\psi^*(y)[y']|\leq|\alpha(|y|)y'|+|\beta(|y|)(y^\top y)[y']|\leq |\alpha(|y|)||y'|+\beta(|y|)|y|^2|y'|\leq C_R|y'|.
\end{align*}
Next note that since $\psi(\nabla\psi^*(y))=\log\tfrac{2R^2}{(1+\sqrt{1+R^2|y|^2})}$ and $\tfrac{R^2}{\sqrt{1+R^2|y|^2}}\leq\tfrac{2R^2}{(1+\sqrt{1+R^2|y|^2})}\leq 2R^2$
it follows that there a constant $C>0$ such that for any $y,y'\in\mathbb{R}^2$ we have
$|\psi(\nabla\psi^*(y))|\leq C\max\{1, |\log(\sqrt{1+R^2|y|^2})|\}\leq C(1+|y|)$.
Moreover $\nabla\psi(\nabla\psi^*(y))=y$ meaning $\mathcal{C}(y)=y$ for all $y\in\mathbb{R}^2$.
Next let $y,y'\in\mathbb{R}^2$ 
and define 
$v_\varepsilon(y,y')=\varepsilon\nabla\psi^*(y)+(1-\varepsilon)\nabla \psi^*(y')$.
Note that
\begin{align*}
|\nabla\psi(a)|=\tfrac{2|a|}{R^2-|a|^2}, \text{ and }|v_\varepsilon(y,y')|\leq \tfrac{R^2\max\{|y|,|y|'\}}{1+\sqrt{1+R^2\max\{|y|^2,|y'|^2\}}}.
\end{align*}
Since the function $|a|\mapsto\tfrac{2|a|}{R^2-|a|^2}$ is (strictly) increasing in $|a|$ 
and assuming WLOG $|y|\geq |y'|$ 
we have that 
\begin{align*}
|\nabla\psi(v_\varepsilon(y,y'))|&=\tfrac{2|v_\varepsilon(y,y')|}{R^2-|v_\varepsilon(y,y')|^2}\leq \tfrac{2\tfrac{R^2|y|}{1+\sqrt{1+R^2|y|^2}}}{R^2-\left(\tfrac{R^2|y|}{1+\sqrt{1+R^2|y|^2}}\right)^2}=\tfrac{2|y|}{1+\sqrt{1+R^2|y|^2}}\tfrac{1}{1-\tfrac{R^2|y|^2}{(1+\sqrt{1+R^2|y|^2})^2}}\\
&=2|y|\tfrac{1+\sqrt{1+R^2|y|^2}}{2(1+\sqrt{1+R^2|y|^2})}=|y|.
\end{align*}

This is sufficient to show Assumption \ref{ass:mirror_reg} holds.
For $\tau>0$ we show the Hamiltonian and optimal control satisfy the conditions outlined in Assumption \ref{ass:ham_structure}. 
From the definition of the Bregman divergence given in \eqref{eqn:Bregman_1} we have 
\begin{align*}
\mathcal{H}(t,x,z)
&=x^\top M^\top_1 z +\tfrac{|x|^2}{2}+\inf_{a\in B_R(0)} \bigg[a^\top N^\top z + \tfrac{|a|^2}{2}+\tau\log\tfrac{R^2}{R^2-|a|^2}\bigg].
\end{align*}
Let $a^*(t,x,z)$ be the optimal $a\in B_R(0)$ for a given $(t,x,z)$. 
The first order optimality condition reads as follows: $\left(1+\tfrac{2\tau}{R^2-|a^*|^2}\right) a^*=-N^\top z$. 
Taking absolute values implies
\begin{align}
\label{eqn:FoC_LQR}
|a^*|\left(1+\tfrac{2\tau}{R^2-|a^*|^2}\right)=|N^\top z|. 
\end{align}
In particular we see that the magnitude of $a^*(t,x,z)$ depends on the magnitude of $N^\top z$ in the following way.
When $|N^\top z|$ is large, the only way \eqref{eqn:FoC_LQR} can hold is if $|a^*(t,x,z)|$ is close to $R$.
Similarly if $|N^\top z|=0$ then $a^*(t,x,z)=0$ and if $|N^\top z|$ is small $|a^*(t,x,z)|$ must be close to $0$.
Observe that from the optimality condition we have that for each $(t,x,z)$, $a^*(t,x,z)=-(R-\varepsilon(z))\tfrac{N^\top z}{|N^\top z|}$.
Substituting this into \eqref{eqn:FoC_LQR} and taking absolute values implies 
\begin{align}
\label{eqn:help_bound}
|N^\top z|=(R-\varepsilon(z))\bigg(1+\tfrac{2\tau}{R^2-(R-\varepsilon(z))^2}\bigg).
\end{align}
Therefore we see that for each $z$, $\varepsilon(z)$ is a root of the cubic equation $P(\varepsilon)=0$ where 
\begin{align*}
P(\varepsilon)=\varepsilon^3+(|z^\top N|-3R)\varepsilon^2+(2R^2-2\tau - 2R|z^\top N|)\varepsilon+2R\tau.
\end{align*}
Note $P(0)=2R\tau$ and $P(R)= -|z^\top N|R^2$ meaning the polynomial admits a real root in the interval $(0,R)$.
The Hamiltonian is therefore measurable and is given by 
\begin{align*}
\mathcal{H}(t,x,z)=x^\top M^\top_1 z +\tfrac{|x|^2}{2}+\bigg[-(R-\varepsilon(z))|z^\top N| + \tfrac{(R-\varepsilon(z))^2}{2}+\tau\log\tfrac{R^2}{R^2-(R-\varepsilon(z))^2}\bigg],
\end{align*}
and the optimal control 
satisfies $u^*(t,x)=a^*(t,x,\nabla V^*(t,x))$.
It remains to show $\mathcal{H}$ grow linearly in $|z|$ which will be addressed in Lemmas \ref{lem:LQR_structural1}, \ref{lem:LQR_structural2} and \ref{lem:LQR_structural3} below.
To this end note that since $\varepsilon(z)\in(0,R)$ for all $z\in\mathbb{R}^d$
\begin{align*}
|\mathcal{H}(t,x,z)|\leq C(1+|z|)+\tau\left|\log \tfrac{R^2}{2R\varepsilon(z)-\varepsilon(z)^2}\right|.
\end{align*}
The following three lemmas are sufficient to show Assumption \ref{ass:ham_structure} holds.
\begin{lemma}
\label{lem:LQR_structural1}
There exists a $\tilde{z}\in\mathbb{R}^d$ such that whenever $z\in\mathbb{R}^d$ satisfies $|N^\top z|\geq |N^\top \tilde{z}|$   we have $\varepsilon(z)<1$.
\end{lemma}
\begin{lemma}
\label{lem:LQR_structural2}
Let $\tilde{z}\in\mathbb{R}^d$ be the vector from Lemma \ref{lem:LQR_structural1}.
Whenever $z\in\mathbb{R}^d$ satisfies $|z^\top N|\geq C=:|\tilde{z}^\top N|$ then 
$|\mathcal{H}(t,x,z)|\leq C(1+|z|+\log|z^\top N|)$.
\end{lemma}
\begin{lemma}
\label{lem:LQR_structural3}
Let $\tilde{z}\in\mathbb{R}^d$ be the vector from Lemma \ref{lem:LQR_structural1}.
Whenever $z\in\mathbb{R}^d$ satisfies $|z^\top N|\leq C=:|\tilde{z}^\top N|$ then $|\mathcal{H}(t,x,z)|\leq C(1+|z|)$.
\end{lemma}

\begin{proof}
[Proof of Lemma \ref{lem:LQR_structural1}]
We first show if $|N^\top z_1|\geq |N^\top z_2|$ then $\varepsilon(z_1)\leq\varepsilon(z_2)$.

To this end define $\mathcal{F}:(0,R)\rightarrow[0,\infty)$ by $\mathcal{F}(\varepsilon)=(R-\varepsilon)\left(1+\tfrac{2\tau}{R^2-(R-\varepsilon)^2}\right)$
and note that $\mathcal{F}$ is continuous and decreasing in $\varepsilon$.
To see this observe that when $\varepsilon_1,\varepsilon_2\in(0,R)$ satisfy 
$\varepsilon_1\geq \varepsilon_2$
\begin{align*}
0\leq R-\varepsilon_1\leq R-\varepsilon_2\text{ and }0\leq 1+\tfrac{2\tau}{R^2-(R-\varepsilon_1)^2}\leq 1+\tfrac{2\tau}{R^2-(R-\varepsilon_2)^2}.
\end{align*}
Therefore $\mathcal{F}(\varepsilon_1)\leq\mathcal{F}(\varepsilon_2)$.
Since we are assuming $|N^\top z_1|\geq|N^\top z_2|$ we have from \eqref{eqn:help_bound} that 
\begin{align*}
\mathcal{F}(\varepsilon(z_1))=|z_1^\top N|
\geq|z_2^\top N|=\mathcal{F}(\varepsilon(z_2))\implies\varepsilon(z_1)\leq \varepsilon(z_2),
\end{align*}
where the implication follows from the fact $\mathcal{F}$ is decreasing.
This concludes the proof that whenever $z_1,z_2\in\mathbb{R}^d$ satisfy  
\begin{align}
\label{eqn:temp_blah}
| N^\top z_1|\geq |N^\top z_2|\implies\varepsilon(z_1)\leq\varepsilon(z_2).
\end{align}
Let $\tilde{z}\in\mathbb{R}^d$ satisfy $\varepsilon(\tilde{z})=c<1$. 
Then setting $z_2=\tilde{z}$ in \eqref{eqn:temp_blah} concludes the proof.
\end{proof}
\begin{proof}
[Proof of Lemma \ref{lem:LQR_structural2}]
Since $|z^\top N|\geq|\tilde{z}^\top N|$ 
Lemma \ref{lem:LQR_structural1} 
asserts that $\varepsilon(z)<1$.
Therefore the Hamiltonian can be written as follows
\begin{align*}
\mathcal{H}(t,x,z)=x^\top M^\top_1 z +\tfrac{|x|^2}{2}+\inf_{0<\varepsilon <1}\bigg[-(R-\varepsilon)|z^\top N| + \tfrac{(R-\varepsilon)^2}{2}+\tau\log\tfrac{R^2}{R^2-(R-\varepsilon)^2}\bigg],
\end{align*}
which can be bound for above and below by
\begin{equation}
\label{eqn:Ham_up_down_bound}
\begin{split}
&x^\top M^\top_1 z +\tfrac{|x|^2}{2}+\inf_{0<\varepsilon <1}\bigg[-(R-\varepsilon)|z^\top N|+\tau\log\tfrac{R^2}{R^2-(R-\varepsilon)^2}\bigg]\\
&\leq\mathcal{H}(t,x,z)\leq x^\top M^\top_1 z +\tfrac{|x|^2}{2}+\tfrac{R^2}{2}+\inf_{0<\varepsilon <1}\bigg[-(R-\varepsilon)|z^\top N|+\tau\log\tfrac{R^2}{R^2-(R-\varepsilon)^2}\bigg].
\end{split}
\end{equation}
Since $|z^\top N|\geq|\tilde{z}^\top N|$ we have that $\varepsilon(z)\leq\varepsilon(\tilde{z})<1$
and so 
$\varepsilon(2R-1)\leq 2R\varepsilon-\varepsilon^2\leq2R\varepsilon$.
Therefore $\log\tfrac{R^2}{2R\varepsilon}\leq\log\tfrac{R^2}{R^2-(R-\varepsilon)^2}=\log\tfrac{R^2}{2R\varepsilon-\varepsilon^2}\leq\log\tfrac{R^2}{\varepsilon(2R-1)}$.
Combining this with \eqref{eqn:Ham_up_down_bound} we have 
\begin{equation}
\begin{split}
&x^\top M^\top_1 z+\tfrac{|x|^2}{2}+\tau\log\tfrac{R^2}{2R}-R|z^\top N|+\inf_{0<\varepsilon <1}\bigg[\varepsilon|z^\top N|-\tau\log\varepsilon\bigg]\\
&\leq\mathcal{H}(t,x,z)\\
&\leq x^\top M^\top_1 z +\tfrac{|x|^2}{2}+\tfrac{R^2}{2}+\tau\log\tfrac{R^2}{2R-1}-R|z^\top N|+\inf_{0<\varepsilon <1}\bigg[\varepsilon|z^\top N|-\tau\log\varepsilon\bigg].
\end{split}
\end{equation}
The infimums are satisfied when $\varepsilon=\tfrac{\tau}{|z^\top N|}$ if $\tau<|z^\top N|$ 
and as $\epsilon\rightarrow 1^{-}$ otherwise 
meaning
\begin{equation}
\begin{split}
&x^\top M^\top_1 z +\tfrac{|x|^2}{2}+\tau\log\tfrac{R^2}{2R}-R|z^\top N|+\tau(1+\log|z^\top N|)\\
&\leq\mathcal{H}(t,x,z)\\
&\leq x^\top M^\top_1 z +\tfrac{|x|^2}{2}+\tfrac{R^2}{2}+\tau\log\tfrac{R^2}{2R-1}-R|z^\top N|+\tau(1+\log|z^\top N|).
\end{split}
\end{equation}
whenever $\tau<|z^\top N|$ and 
\begin{equation}
\begin{split}
&x^\top M^\top_1 z+\tfrac{|x|^2}{2}+\tau\log\tfrac{R^2}{2R}-R|z^\top N|+|z^\top N|\\
&\leq\mathcal{H}(t,x,z)\\
&\leq x^\top M^\top_1 z +\tfrac{|x|^2}{2}+\tfrac{R^2}{2}+\tau\log\tfrac{R^2}{2R-1}-R|z^\top N|+|z^\top N|,
\end{split}
\end{equation}
whenever $\tau\geq |z^\top N|$.
This concludes the proof.
\end{proof}
\begin{proof}
[Proof of Lemma \ref{lem:LQR_structural3}]
Since $|z^\top N|\leq|\tilde{z}^\top N|$  using the same argument as in Lemma \ref{lem:LQR_structural1} we have $c=\varepsilon(\tilde{z})\leq\varepsilon(z)$ where $c$ is the constant from the proof of Lemma \ref{lem:LQR_structural1}.
Therefore it follow from \eqref{eqn:Ham_up_down_bound} that the
\begin{equation}
\begin{split}
&x^\top M^\top_1 z +\tfrac{|x|^2}{2}-R|z^\top N|+\inf_{c\leq\varepsilon \leq R}\bigg[\varepsilon|z^\top N|+\tau\log\tfrac{R^2}{R^2-(R-\varepsilon)^2}\bigg]\\
&\leq\mathcal{H}(t,x,z)\leq x^\top M^\top_1 z +\tfrac{|x|^2}{2}+\tfrac{R^2}{2}-R|z^\top N|+\inf_{c\leq\varepsilon \leq R}\bigg[\varepsilon|z^\top N|+\tau\log\tfrac{R^2}{R^2-(R-\varepsilon)^2}\bigg].
\end{split}
\end{equation}
Since the logarithmic term is bounded on the interval $\varepsilon\in(c,R)$ we have 
\begin{equation}
\begin{split}
x^\top M^\top_1 z &+\tfrac{|x|^2}{2}-R|z^\top N|+\tau C+|z^\top N|\\
&\leq\mathcal{H}(t,x,z)\leq x^\top M^\top_1 z +\tfrac{|x|^2}{2}+\tfrac{R^2}{2}-R|z^\top N|+\tau C+R|z^\top N|.
\end{split}
\end{equation}
This concludes the proof.
\end{proof}

We now show that for $\tau>0$ the optimal control can be written as $u^*=\nabla\psi^*(Z^*)$ with $Z^*\in B_b(\cO_T;\mathbb{R}^2)$.
By Theorem~\ref{thm:HJB_well_posed} and~\eqref{eqn:embedding_Holder_Sob}, $m:=\sup_{(t,x)\in\cO_T}|N^\top\nabla V^*(t,x)|<\infty$.
If $m=0$ then $u^*=0$.
Otherwise, the function $\mathcal{F}$ from the proof of Lemma~\ref{lem:LQR_structural1} is continuous and strictly decreasing on $(0,R)$, with $\mathcal{F}(\varepsilon)\to\infty$ as $\varepsilon\downarrow 0$ and $\mathcal{F}(\varepsilon)\to 0$ as $\varepsilon\uparrow R$, so there is a unique $\varepsilon_m\in(0,R)$ with $\mathcal{F}(\varepsilon_m)=m$.
By~\eqref{eqn:help_bound}, $\mathcal{F}(\varepsilon(\nabla V^*(t,x)))=|N^\top\nabla V^*(t,x)|\le m=\mathcal{F}(\varepsilon_m)$, hence $\varepsilon(\nabla V^*(t,x))\geq\varepsilon_m$ and $|u^*(t,x)|\le R-\varepsilon_m$ for all $(t,x)\in\cO_T$.
Therefore $Z^*:=\nabla\psi(u^*)=\tfrac{2u^*}{R^2-|u^*|^2}$ is bounded and, since $\nabla\psi^*$ is the inverse of $\nabla\psi$ on the interior of $A$, $\nabla\psi^*(Z^*)=u^*$.
It remains to show the convexity condition in Assumption \ref{ass:ham_convex} holds. 
To this end 
\begin{align*}
H(t,x,z,a)-H(t,x,z,a')&=\left[\tfrac{a+a'}{2}+N^\top z\right]\cdot (a-a')+\tau\bigg(\log\tfrac{R^2}{R^2-|a|^2}-\log\tfrac{R^2}{R^2-|a'|^2}\bigg)\\
&=\left[\tfrac{a+a'}{2}+N^\top z\right]\cdot (a-a')+\tau\bigg(\log\tfrac{R^2-|a'|^2}{R^2-|a|^2}\bigg)\\
&=\tfrac{|a-a'|^2}{2}+\bigg[a'+N^\top z + \tfrac{2\tau a'}{R^2-|a'|^2}\bigg]\cdot (a-a')+\tau D_\psi(a,a')\\
&\geq \nabla_a H(t,x,z,a')\cdot(a-a')+\tau D_\psi(a, a').
\end{align*}

\subsection{Validation of assumptions for Example \ref{sec:example_finite_actions}}
\label{sec_app_2nd_ex}
We first show that the formulation presented in Example \ref{sec:example_finite_actions} is equivalent to considering an entropy regularized Markovian control problem with an action space of finite cardinality.
For a given  $\pi\in\mathcal{P}(\{1,\dots,p\}|\cO_T)$ let $X^{t,x,\pi}$ correspond to the solution of the following controlled SDE
\begin{equation}
\label{eqn:relaxed_SDE}
\begin{split}
dX_s &= \int_{A}\beta(s,X_s,a)\pi(da|s,X_s)dt + \sigma(s,X_s)dW_s,\enspace s\geq t \enspace X_t=x\\
&=\sum_{i=1}^p\beta(s,X_s,i)\pi(a=i| s,X_s)ds+\sigma(s,X_s)dW_s.
\end{split}
\end{equation}
The value function is given by 
\begin{equation}
\label{eqn:relaxed_VF}
\begin{split}
V^\pi(t,x)
& = \E\bigg[\int_t^{T_\cO}\sum_{i=1}^p\bigg(
\varphi(s,X_s,i)+\tau \log\tfrac{\pi(a=i|s,X_s)}{\mu(i)}\bigg)\pi(a=i|s,X_s)ds + g(X_{T_\cO})\bigg]
\end{split}
\end{equation}
Define $u:\cO_T\rightarrow A$ by $u_i(t,x)=\pi(a=i|t,x)$ and $\mu(i)=a^{(0)}_i$.
It is then clear displays \eqref{eqn:relaxed_SDE} and \eqref{eqn:relaxed_VF}
are identical to \eqref{eqn:eg_finite_actionion_SDE} and \eqref{eqn:eg_finite_action_VF} respectively.
A straightforward computation shows that  
\begin{align*}
u^*(t,x)\in\argmin_{a\in A}H(t,x,z,a)\iff u^*(t,x)=\left(\tfrac{a_i^{(0)}\exp\bigg(-\tfrac{\beta(t,x,i)\cdot z + \varphi_i(t,x)}{\tau}\bigg)}{\sum_{j=1}^pa_j^{(0)}\exp\bigg(-\tfrac{\beta(t,x,j)\cdot z + \varphi(t,x,j)}{\tau}\bigg)}\right)_{i=1}^p,
\end{align*}
and 
\begin{align*}
\mathcal{H}(t,x,z)=\inf_{a\in A}H(t,x,z,a)=-\tau\log\sum_{i=1}^pa_i^{(0)}\exp\bigg(-\tfrac{\beta(t,x,i)\cdot z + \varphi(t,x,i)}{\tau}\bigg).
\end{align*}
Therefore
\begin{align*}
|\mathcal{H}(t,x,z)|&\leq \tau\log\sum_{i=1}^p\exp\left(\tfrac{\sup_{1\leq i \leq p}|z|\|\beta(\cdot,i)\|_{B_b(\cO_T;\mathbb{R}^d)}+\sup_{1\leq i\leq p}\|\varphi(\cdot,i)\|_{B_b(\cO_T;\mathbb{R})}}{\tau}\right)\\
&\leq p\left(\sup_{1\leq i \leq p}\|b_i\|_{B_b(\cO_T;\mathbb{R}^d)}|z|+\sup_{1\leq i \leq p}\|\varphi(\cdot,i)\|_{B_b(\cO_T;\mathbb{R})}\right)\leq C(1+|z|),
\end{align*}
which concludes the proof of Assumption \ref{ass:ham_structure}.

Moreover, $\nabla V^*$ is bounded by Theorem~\ref{thm:HJB_well_posed} and~\eqref{eqn:embedding_Holder_Sob}.
Hence, with $z=\nabla V^*$, the optimal control is $u^*=\text{softmax}(Z^*)=\nabla\psi^*(Z^*)$, where
\[
Z^*_i(t,x):=\log a^{(0)}_i-\tfrac{1}{\tau}\big(\beta(t,x,i)\cdot\nabla V^*(t,x)+\varphi(t,x,i)\big),\quad i=1,\dots,p,
\]
which is bounded provided $a^{(0)}_i>0$ for all $i$ (e.g.\ the uniform reference policy used in Section~\ref{sec numerical experiments}).
Next we validate Assumption \ref{ass:mirror_reg}.  Let $y,y'\in \mathbb{R}^p$
\begin{align*}
|\nabla^2\psi^*(y)[y']|\leq|\text{diag}(\text{softmax}(y))[y']|+|\text{softmax}(y)\text{softmax}(y)^\top[y']|\leq C|y'|.
\end{align*}
Since $\psi(a)=\sum_{i=1}^pa_i\log a_i$, $\nabla\psi(a)=(1+\log a_i)_{i=1}^p$ and $\nabla\psi^*(y)=\tfrac{e^{y_i}}{\sum_{j=1}^p e^{y_j}}$, we have
\begin{align*}
|\psi(\nabla\psi^*(y))|&=\bigg|\sum_{i=1}^p\tfrac{e^{y_i}}{\sum_{j=1}^pe^{y_j}}(y_i-\log\sum_{j=1}^pe^{y_j})\bigg|\leq |y|+|\log\sum_{j=1}^pe^{y_j}\leq 2|y|,\\
\nabla\psi(\nabla\psi^*(y))&= (1+y_i-\log\sum_{j=1}^pe^{y_j})_{i=1}^p.
\end{align*}
so that $\mathcal{C}(y)=(1+\log\sum_{j=1}^pe^{y_j}))_{i=1}^p$.
Next let $y,y'\in \mathbb{R}^p$, $\varepsilon\in[0,1]$ and $v_\varepsilon(y,y')=\varepsilon\nabla\psi^*(y)+(1-\varepsilon)\nabla\psi^*(y')$ then 
\begin{align*}
\nabla\psi(v_\varepsilon(y,y'))= \nabla\psi(\varepsilon\nabla\psi^*(y)+(1-\varepsilon)\nabla\psi^*(y'))
=\bigg(1+\log\left(\tfrac{\varepsilon e^{y_i}}{\sum_{j=1}^pe^{y_j}}+\tfrac{(1-\varepsilon)e^{y'_i}}{\sum_{j=1}^pe^{y'_j}}\right)\bigg)_{i=1}^p.
\end{align*}
For each $1\leq i \leq p$ from Jensen's inequality
\begin{align*}
[\nabla\psi(v_\varepsilon(y,y'))]_i\geq 1+\varepsilon(y_i - \log \sum_{j=1}^p e^{y_j})+(1-\varepsilon)(y'_i - \log \sum_{j=1}^p e^{y'_j}).
\end{align*}
For the upper bound note that since the map $x\mapsto\log x$, satisfies $\log x \leq x + 1$ for all $x>0$,
\begin{align*}
[\nabla\psi(v_\varepsilon(y,y'))]_i\leq 2+\tfrac{\varepsilon e^{y_i}}{\sum_{j=1}^pe^{y_j}}+\tfrac{(1-\varepsilon)e^{y'_i}}{\sum_{j=1}^pe^{y'_j}}.
\end{align*}
Combining the above inequalities concludes the proof for Assumption \ref{ass:mirror_reg}.
It remains to prove that for any $y,y'\in\mathbb{R}^p$, $D_\psi(\nabla\psi^*(y),\nabla\psi^*(y'))=D_{\psi^*}(y',y)$.
Let $Z=\sum_{j=1}^p e^{y_j}$ and $Z'=\sum_{j=1}^p e^{y'_j}$.
Then 
\begin{align*}
&D_\psi(\nabla\psi^*(y),\nabla\psi^*(y'))\\
&=\psi(\nabla\psi^*(y))-\psi(\nabla\psi^*(y'))-\nabla\psi(\nabla\psi^*(y'))\cdot(\nabla\psi^*(y)-\nabla\psi^*(y'))\\
&=\sum_{i=1}^p \tfrac{e^{y_i}}{Z}(y_i-\log Z)
- \sum_{i=1}^p \tfrac{e^{y'_i}}{Z'}(y'_i-\log Z')-
\sum_{i=1}^p
\left[1+y'_i-\log Z'\right]
\cdot
\left[\tfrac{e^{y_i}}{Z}-\tfrac{e^{y'_i}}{Z'}\right]\\
&=\sum_{i=1}^p \tfrac{y_ie^{y_i}}{Z}-\log Z
+\log Z'-\sum_{i=1}^p\tfrac{y'_ie^{y_i}}{Z}=\sum_{i=1}^p\tfrac{e^{y_i}}{Z}(y_i-y'_i)+\log\tfrac{Z'}{Z}\\
&=\psi^*(y')-\psi^*(y)-\tfrac{e^{y_i}}{\sum_{j=1}^pe^{y_j}}(y'_i-y_i)=D_{\psi^*}(y',y).
\end{align*}
Finally we want to show the convexity condition in Assumption \ref{ass:ham_convex} holds.
To that end 
\begin{align*}
H(t,x,z,a)&-H(t,x,z,a')=H(t,x,z,a)-H(t,x,z,a')+\tau (\KL(a', a^{(0)}) - \KL(a',a^{(0)}))\\
&=\nabla_aH(t,x,z,a')\cdot(a-a')+\tau (\KL(a', a^{(0)})-\KL(a', a^{(0)}))\\
&\geq \nabla_aH(t,x,z,a')\cdot(a-a')+\tau \KL(a, a'),
\end{align*}
where the first equality follows from the fact $H$ is linear in $a$ and the first inequality follows from standard properties of $\KL$ divergence.

\begin{figure}
\centering
\includegraphics[width=0.49\textwidth]{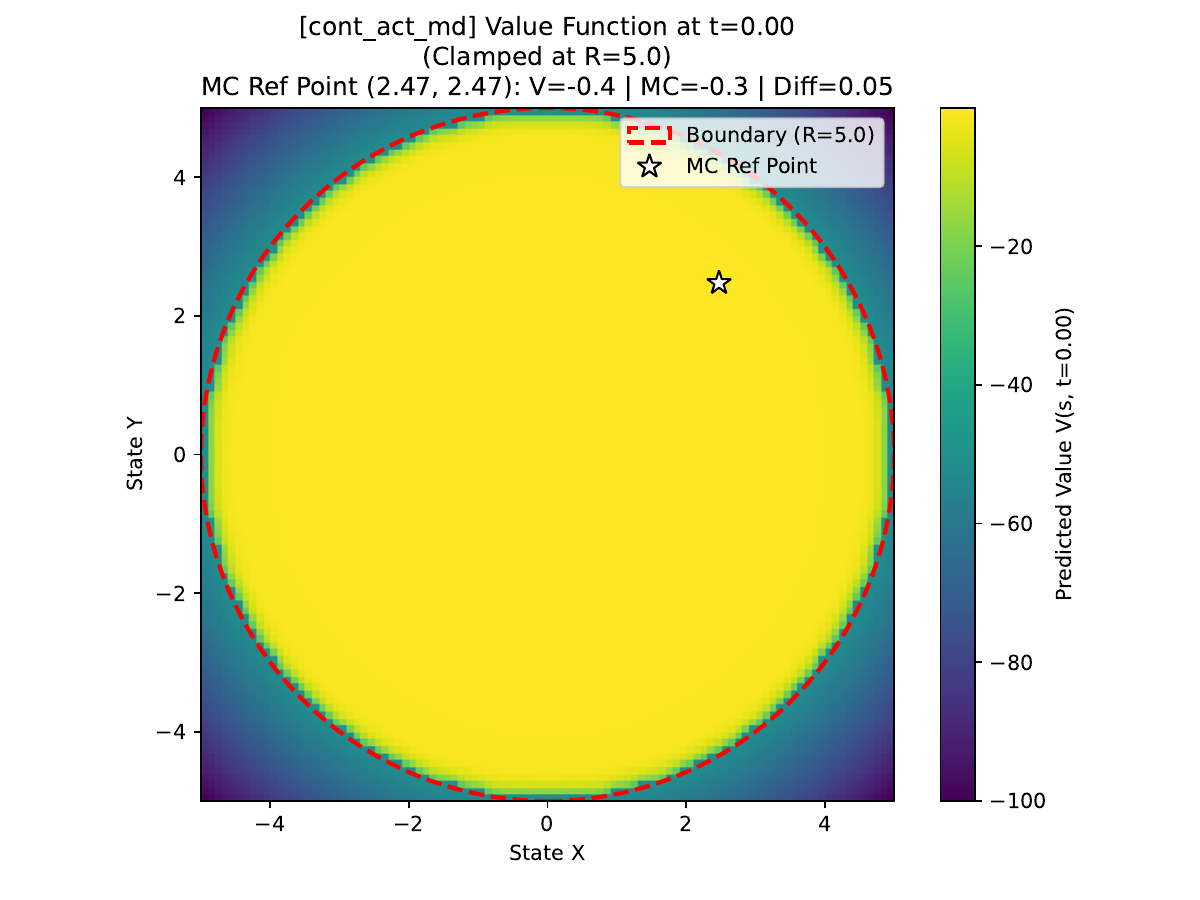}
\includegraphics[width=0.38\textwidth]{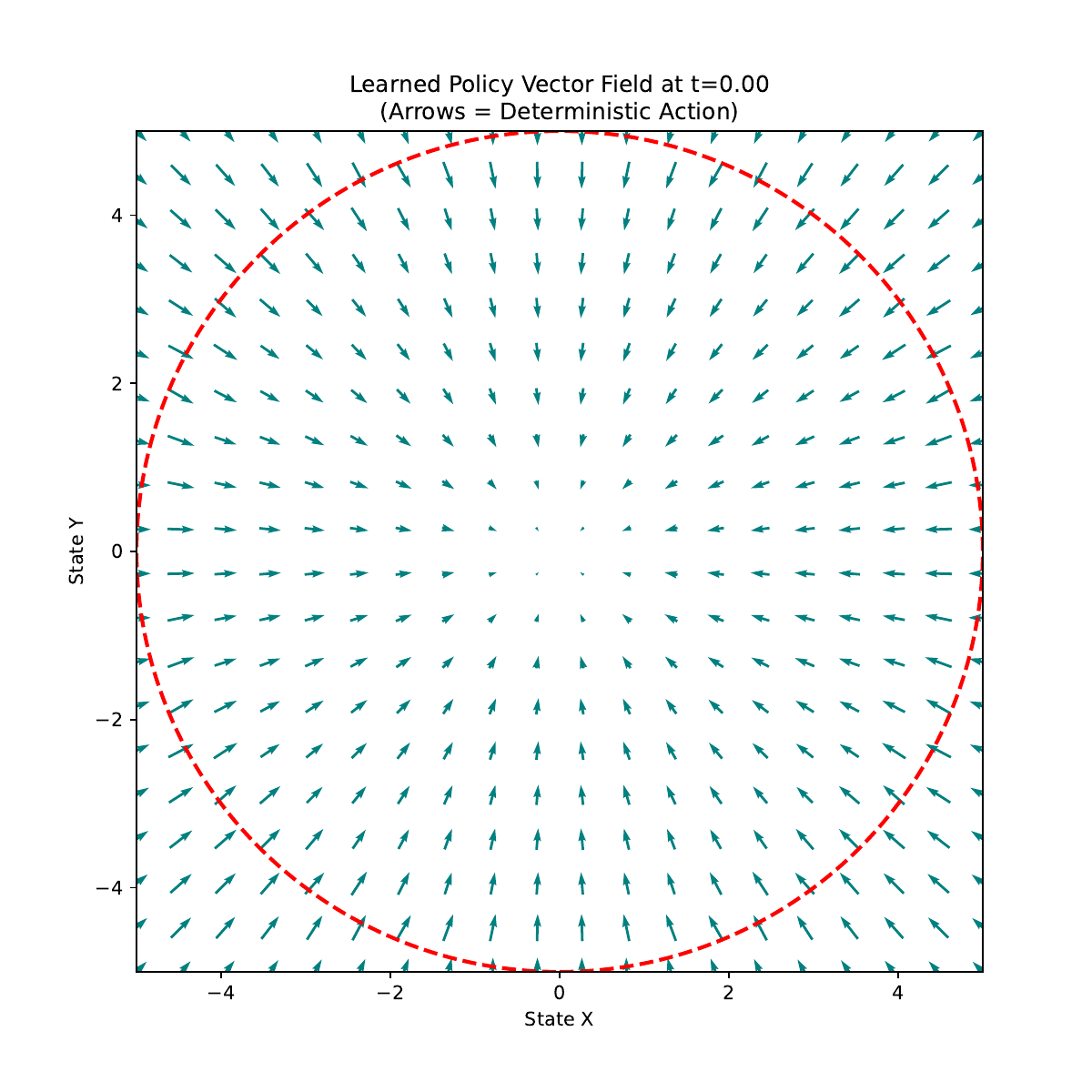}
\caption{Value function (left) and policy vector field (right) at time $t=0$ found with the mirror descent algorithm for the problem of Section~\ref{sec:ex_LQR}.}
\label{fig value and field}
\end{figure}

\begin{figure}
\centering
\includegraphics[width=0.9\textwidth]{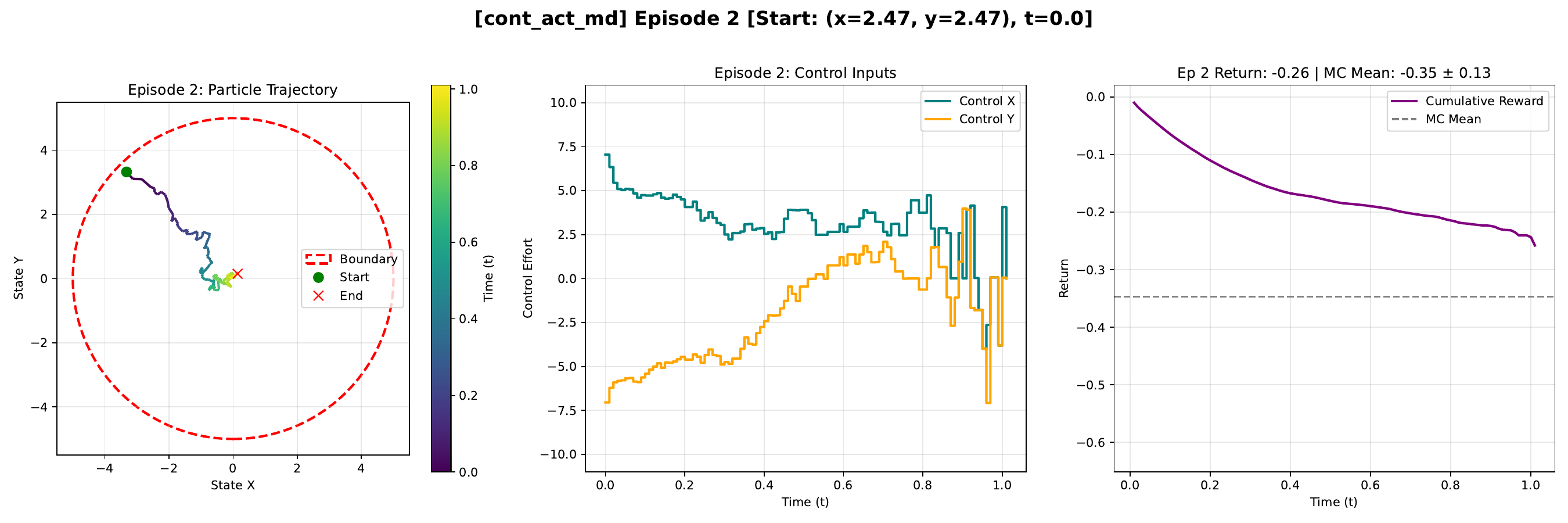}
\caption{Example trajectories with control found with the mirror descent algorithm for the problem of Section~\ref{sec:ex_LQR}.}
\label{fig validation}
\end{figure}

\section{Numerical experiment parameters}
\label{sec experiment parameters}

The parameters used for running the mirror descent combined with FDM solver Algorithm~\ref{alg:dynamic_md} as well as the projected PG methods for the environment of Section~\ref{sec:ex_LQR} (LQR-like with restricted actions and spatial boundary) are in Table~\ref{tab:dpp_hyperparameters}.

\begin{table}[htpb]
\centering
\begin{tabular}{@{}llc@{}}
\toprule
\textbf{Category} & \textbf{Parameter} & \textbf{Value} \\
\midrule
\multicolumn{3}{c}{\textit{Environment configuration (example of Section~\ref{sec:ex_LQR})}} \\
\midrule
Dimensions & State dimension ($d$) & $2$ \\
& Action dimension ($p$) & $2$ \\
& Action space radius ($R$) & $10.0$ \\
& State domain radius & $5.0$ \\
\addlinespace
Dynamics \& Costs & Time horizon ($T$) & $1.0$ \\
& Simulation time step ($dt$) & $0.01$ \\
& Diffusion scale ($\sigma$) & $1.0$ \\
& Control cost factor & $0.01$ \\
& Terminal space cost factor & $2.0$ \\
\midrule
\multicolumn{3}{c}{\textit{Algorithm \& FDM parameters}} \\
\midrule
Discretization & PDE time step & $0.01$ \\
& Time step of the policy table & $0.05$ \\
& Spatial grid bins & $101$ per dimension \\
\addlinespace
Optimization & Total iterations ($N$) & $30$ and varying for convergence plots \\
& Dual step size, mirror descent ($\eta$) & $0.03$ \\
& Step size, projected PG & $30$ \\
& Regularization parameter ($\tau$) & $0.0$ \\
\midrule
\multicolumn{3}{c}{\textit{Monte Carlo evaluation}} \\
\midrule
Rollouts & Number of trajectories ($N_{\text{mc}}$) & $8,192$ \\
& Parallel environments & $16$ \\
& Evaluation seed & $42$ \\
\bottomrule
\end{tabular}
\caption{Hyperparameters for the environment and PDE/FDM-based algorithms (Algorithm \ref{alg:dynamic_md} and projected PG)}
\label{tab:dpp_hyperparameters}

\end{table}

The parameters used for running the mirror descent combined with FDM solver Algorithm~\ref{alg:dynamic_md} as well as the projected PG methods for the environment of Section~\ref{sec:example_finite_actions} (probabilities over discrete actions) are in Table~\ref{tab:dpp_discrete_hyperparameters}.

\begin{table}[htpb]
\centering
\begin{tabular}{@{}llc@{}}
\toprule
\textbf{Category} & \textbf{Parameter} & \textbf{Value} \\
\midrule
\multicolumn{3}{c}{\textit{Environment configuration (example of Section~\ref{sec:example_finite_actions})}} \\
\midrule
Dimensions & State dimension ($d$) & $2$ \\
& Physical action dimensions & $2$ \\
& Actions per dimension & $21$ \\
& Total discrete actions ($p$) & $441$ \\
& Physical action bounds & $[-10.0, 10.0]$ \\
& State domain radius & $5.0$ \\
\addlinespace
Dynamics \& Costs & Time horizon ($T$) & $1.0$ \\
& Simulation time step ($dt$) & $0.01$ \\
& Diffusion scale ($\sigma$) & $1.0$ \\
& Control cost factor & $0.01$ \\
& Terminal space cost factor & $2.0$ \\
\midrule
\multicolumn{3}{c}{\textit{Algorithm \& FDM parameters}} \\
\midrule
Discretization & PDE time step & $0.01$ \\
& Time step of the policy table & $0.05$ \\
& Spatial grid bins & $41$ per dimension \\
\addlinespace
Optimization & Total iterations ($N$) & $10$ and varying for convergence plots\\
& Dual step size, mirror descent ($\eta$) & $1.0$ \\
& Step size, projected PG & $0.01$ \\
& Regularization parameter ($\tau$) & $10^{-6}$ \\
& Discount factor & $1.0$ \\
\midrule
\multicolumn{3}{c}{\textit{Monte Carlo evaluation}} \\
\midrule
Rollouts & Number of trajectories ($N_{\text{mc}}$) & $2,048$ \\
& Parallel environments & $16$ \\
& Evaluation seed & $42$ \\
\bottomrule
\end{tabular}
\caption{Hyperparameters for the environment and PDE/FDM-based algorithms (soft policies / example of Section~\ref{sec:example_finite_actions})}
\label{tab:dpp_discrete_hyperparameters}
\end{table}

The parameters used for running the mirror descent RL with advantage rate critic Algorithm~\ref{alg:rl_md} are in Table~\ref{tab:rl_hyperparameters}. 

\begin{table}[htpb]
\centering
\begin{tabular}{@{}llc@{}}
\toprule
\textbf{Category} & \textbf{Parameter} & \textbf{Value} \\
\midrule
\multicolumn{3}{c}{\textit{Environment configuration}} \\
\midrule
Dimensions & State dimension ($d$) & $2$ \\
& Action dimension ($p$) & $2$ \\
& Actions per dimension ($n$) & $21$ \\
& Action bounds & $[-10.0, 10.0]$ \\
\addlinespace
Dynamics \& Costs & Time horizon ($T$) & $1.0$ \\
& Environment time step ($dt$) & $0.05$ \\
& Domain radius ($R$) & $5.0$ \\
& Diffusion scale ($\sigma$) & $0.1$ \\
& Control cost factor & $0.01$ \\
& Terminal space cost factor & $2.0$ \\
\midrule
\multicolumn{3}{c}{\textit{Algorithm \& Training parameters}} \\
\midrule
Architecture & Policy type & Quantized Tabular ($V, V_{\text{tgt}}, q$) \\
& Spatial grid bins & $41$ per dimension \\
\addlinespace
Rollout & Total timesteps & $2,000,000$ \\
& Parallel environments ($N_{\text{envs}}$) & $64$ \\
& Buffer size & $262,144$ \\
& Learning starts & $100$ transitions \\
\addlinespace
Optimization & Batch size & $16,384$ \\
& Gradient steps per train freq & $8$ \\
& Critic learning rate ($\alpha$) & Linear schedule $0.1 \to 0.01$ \\
& Mirror descent step size ($\eta$) & Linear schedule $10^{-4} \to 10^{-7}$ \\
& Target network update rate ($\tau_{\text{polyak}}$) & $0.001$ \\
& Entropy regularization ($\tau$) & $10^{-12}$ \\
& Discount factor ($\gamma$) & $1.0$ \\
\bottomrule
\end{tabular}
\caption{Hyperparameters for the environment and mirror descent RL algorithm}
\label{tab:rl_hyperparameters}
\end{table}

\section*{Acknowledgements}
The authors would like to sincerely thank the anonymous reviewers whose insightful comments helped to improve the manuscript significantly from its initially submitted version.
\bibliographystyle{siam}
\bibliography{bibliography}

\end{document}